\documentclass[oneside,reqno,a4paper,11pt]{amsart}
\usepackage{amsfonts}
\usepackage{amsmath,amsfonts,amssymb}
\usepackage{CJK,bbm,color,soul,supertabular,longtable,verbatim}
\usepackage[warning, easyold,math]{easyeqn}
\usepackage{graphicx}
\usepackage{color}
\usepackage{latexsym, amsfonts,mathrsfs, amsmath, amssymb, amscd, epsfig}
\usepackage{mathrsfs}
\usepackage{ esint }
\usepackage{amsthm}

\newtheorem {theorem}{Theorem}[section]
\newtheorem {lemma}[theorem]{{\bf Lemma}}
\newtheorem {proposition}[theorem]{\bf Proposition}

\newtheorem {prop}[theorem]{{\bf Proposition}}
\theoremstyle{remark}
\newtheorem {remark}{{\bf Remark}}[section]

\theoremstyle{problem}

\theoremstyle{definition}
\newtheorem {definition}{{\bf Definition}}[section]
\theoremstyle{plain} \numberwithin {equation}{section}

\begin{document}
\vspace{1cm}

\title[Incompressible impinging jet flow with gravity]{Incompressible impinging jet flow with gravity$^*$}
\author[Jianfeng Cheng,\\Lili Du,\\ Zhouping Xin]{Jianfeng Cheng$^1$,\ \ Lili Du$^1$,\ \ Zhouping Xin$^2$}

\thanks{$^*$Cheng and Du are supported in part by NSFC grant 11571243, 11622105. Xin is supported in part by the Zheng Ge Ru
foundation, Hong Kong RGC Grants: CUHK-14300917, CUHK14305315 and
CUHK-4048/13P, NSFC/RGC Joint Research Grant N-CUHK 443/14, and a
Focus Area Grant from The Chinese University of Hong Kong}
\thanks{ E-Mail: jianfengcheng@126.com (J. Cheng),  dulili@scu.edu.cn (L. Du), zpxin@ims.cuhk.edu.hk (Z. Xin)}
 \maketitle
\begin{center}
$^1$ Department of Mathematics, Sichuan University,

          Chengdu 610064, P. R. China.

$^2$ The Institute of Mathematical Sciences,

The Chinese University of Hong Kong,

Shatin, N.T., Hong Kong.

\end{center}

\begin{abstract} In this paper, we investigate steady
two-dimensional free-surface flows of an inviscid and incompressible
fluid emerging from a nozzle, falling under gravity and impinging
onto a horizontal wall. More precisely, for any given atmosphere
pressure $p_{atm}$ and any appropriate incoming total flux $Q$, we
establish the existence of two-dimensional incompressible impinging
jet with gravity. The two free surfaces initiate smoothly at the
endpoints of the nozzle and become to be horizontal in downstream.
By transforming the free boundary problem into a minimum problem, we
establish the properties of the flow region and the free boundaries.
Moreover, the asymptotic behavior of the impinging jet in upstream
and downstream is also obtained.

\end{abstract}

\

\begin{center}
\begin{minipage}{5.5in}

\

\end{minipage}
\end{center}

\

\everymath{\displaystyle}
\newcommand {\eqdef }{\ensuremath {\stackrel {\mathrm {\Delta}}{=}}}

\tableofcontents

% math symbols
\def\Xint #1{\mathchoice
{\XXint \displaystyle \textstyle {#1}} %
{\XXint \textstyle \scriptstyle {#1}} %
{\XXint \scriptstyle \scriptscriptstyle {#1}} %
{\XXint \scriptscriptstyle \scriptscriptstyle {#1}} %
\!\int}
\def\XXint #1#2#3{{\setbox 0=\hbox {$#1{#2#3}{\int }$}
\vcenter {\hbox {$#2#3$}}\kern -.5\wd 0}}
\def\ddashint {\Xint =}
\def\dashint {\Xint -}
\def\clockint {\Xint \circlearrowright } % GOOD !
\def\counterint {\Xint \rotcirclearrowleft } % Good for Computer Modern !
\def\rotcirclearrowleft {\mathpalette {\RotLSymbol { -30}}\circlearrowleft }
\def\RotLSymbol #1#2#3{\rotatebox [ origin =c ]{#1}{$#2#3$}}

\def\aint{\dashint}

\def\arraystretch{2}
\def\eps{\varepsilon}

\def\s#1{\mathbb{#1}} % set
\def\t#1{\tilde{#1}} %new variables
\def\b#1{\overline{#1}}
\def\N{\mathcal{N}} %Nozzle
\def\M{\mathcal{M}} %Mach number
\def\R{{\mathbb{R}}}
\def\B{{\mathcal{B}}}
\def\BB{\mathfrak{B}}
\def\F{{\mathcal{F}}}
\def\G{{\mathcal{G}}}
\def\ba{\begin{array}}
\def\ea{\end{array}}
\def\be{\begin{equation}}
\def\ee{\end{equation}}

\def\bes{\begin{mysubequations}}
\def\ees{\end{mysubequations}}

\def\cz#1{\|#1\|_{C^{0,\alpha}}}
\def\ca#1{\|#1\|_{C^{1,\alpha}}}
\def\cb#1{\|#1\|_{C^{2,\alpha}}}
\def\psir{\left|\frac{\nabla\psi}{r}\right|^2}
\def\lb#1{\|#1\|_{L^2}}
\def\ha#1{\|#1\|_{H^1}}
\def\hb#1{\|#1\|_{H^2}}
\def\th{\theta}
\def\Th{\Theta}
\def\cin{\subset\subset}
\def\Ld{\Lambda}
\def\ld{\lambda}
\def\ol{{\Omega_L}}
\def\sla{{S_L^-}}
\def\slb{{S_L^+}}
\def\e{\varepsilon}
\def\C{\mathbf{C}} %%unit cylinder
\def\cl#1{\overline{#1}}
\def\ra{\rightarrow}
\def\xra{\xrightarrow}
\def\g{\nabla}
\def\a{\alpha}
\def\b{\beta}
\def\d{\delta}
\def\th{\theta}
\def\fai{\varphi}
\def\O{\Omega}
\def\ol{{\Omega_L}}
\def\psirk{\left|\frac{\nabla\psi}{r+k}\right|^2}
\def\tO{\tilde{\Omega}}
\def\tu{\tilde{u}}
\def\tv{\tilde{v}}
\def\trho{\tilde{\rho}}
\def\W{\mathcal{W}}
\def\f{\frac}
\def\p{\partial}
\def\m{\omega}
\def\B{\mathcal{B}}
\def\H{\Theta}
\def\msS{\mathscr{S}}
\def\bq{\mathbf{q}}
\def\msE{\mathcal{E}}
\def\mfa{\mathfrak{a}}
\def\mfb{\mathfrak{b}}
\def\mfc{\mathfrak{c}}
\def\mfd{\mathfrak{d}}
\def\Div{\text{div}}
\def\Rot{\text{rot}}
\def\Curl{\text{curl}}
\def\mcL{\mathcal{L}}
\def\mcR{\mathcal{R}}
\def\f{\frac}
\def\p{\partial}
\def\o{\omega}
\def\h{_2^{\frac{1}{2}}}
\def\hh{_2^2}
\def\hhh{_2^{\frac{2}{3}}}
\def\k{_2^{\frac{3}{2}}}
\def\ii{\int_{0}^{t}\int}
\def\xiao{\leq}

\section{Introduction and main results}

\subsection{Introduction}

The problem describing a jet emerging from a nozzle and impacting on
a solid wall is important. It possesses the key ingredients of a
challenging mathematical problem with numerous practical
applications, such as the industrial process of fabricating glassy
metals, a jet of molten metal emerging from a nozzle is directed
onto a moving plate and develops a glassy structure due to the rapid
cooling. The more obvious industrial applications involve a
Vertical/Short Takeoff and Landing (V/STOL) aircraft, terrestrial
rocket launch, the formulation of continuous fibers from molten
liquid extruded through an orifice \cite{GI}. Although these
practical applications are different in many respects, the one
underlying feature is the fundamental fluid mechanics of the
impinging jet. Consequently, an accurate but mathematically
tractable model to describe the essential features of the impinging
jet would be desirable.

 The impinging jet problem has been studied from the computational
 and experimental expects, by Dias-Elcrat and Trefethen, who
 presented an efficient procedure for solving the impinging jet
 problem numerically in \cite{DEF}. See also in \cite{JB} for a numerical solution to
 the incompressible impinging jet in the absence of gravity. King and
 Bloor introduced a method in \cite{KB} for determining the free streamline of a
 free jet of ideal weightless fluid impacting on a wall. We also would
 like to refer the surveys \cite{BZ} by Birkhoff and Zarantonello,
 \cite{JA} by Jacob, \cite{Gu} by Gurevich,  \cite{MT} by
 Milne-Thomson, \cite{Wu} by Wu for the mathematical theory on jets of ideal fluid.

 The first systematic existence theory on the incompressible
 impinging jet of ideal weightless fluids was mentioned in the
 monograph \cite{FA1} of Friedman (Page 365 and Page 416), and some existence results of an impinging jet emerging
 from a two-dimensional nozzle and impacting on an infinite wall has
 been established in an unpublished paper by Caffarelli and Friedman.
 Very recently, Cheng, Du and Wang in \cite{CDW} extended the existence and non-existence
 results
 to the oblique impinging jet. The "oblique" means that the orifice
 of the nozzle is not parallel to the rigid wall. The main idea
 follows from the techniques introduced by Alt, Caffarelli and Friedman in \cite{AC1,ACF1,ACF4}, the original physical problem was transformed into a
 minimum problem, and then the existence of incompressible jet and
 the properties of the free boundary is established. The main
 purpose of this paper is to extend the existence result by
 Caffarelli and Friedman to the case where the gravity field is
 present. The analysis of jets in the presence of gravity is
 difficult because of the strong nonlinearity of the dynamic
 boundary  condition that must be satisfied on both unknown free
 surfaces bounding the jet. In this paper, we consider the steady irrotational flow
 of an incompressible inviscid fluid emerging from a nozzle, falling
 vertically under the effects of the gravity and impinging on a rigid
 wall. The total incoming flux and atmospheric pressure are imposed,
 and some existence results on incompressible impinging jets under
 gravity are established.

 Another motivation to investigate the impinging jet problem under
 gravity comes from the classical result on a falling jet under
 gravity without the horizontal plate in the significant work \cite{ACF2} by Alt,
 Caffarelli and Friedman. As mentioned in Page 59 in \cite{ACF2},

 {\it "The mathematical literature on jets with gravity is very meager.
 The reason for this is that the hodograph method which has been
 successfully used in steady 2-dimensional problems for jets and cavities without gravity cannot be extended to
 the case where gravity is present."}

  They considered the
 incompressible jet under gravity emerging from a channel in axially
 symmetric case and two-dimensional asymmetric case, assuming the
 height of the channel to be infinite. In the axially symmetric case, existence and uniqueness of an incompressible jet with gravity were established, and
 surprisingly, some non-existence result on asymmetric jet under
 gravity was also obtained in Section 14 in \cite{ACF2}. They gave a
 counterexample to the existence of asymmetric jets under gravity,
 and showed that in general the two free streamlines can not connect
 the endpoints at the same time. Nevertheless, in this paper, we
 consider the jet emerging from the nozzle with finite height, which
 does not fall into the case of their counterexamples, and try to seek a
 mechanism to establish the well-posedness theory of the asymmetric
 jets under gravity. This seems to be the first well-posedness result on
 asymmetric jets under gravity.

 \subsection{Statement of the physical problem}

The problem we will address is that the two-dimensional steady flow
of an incompressible inviscid fluid emerges from a semi-infinitely
long nozzle, and impinges onto the ground. We shall take into
account of the gravity but neglect all other forces, such as surface
tension and air resistance.

The situation is shown in Figure \ref{f1}. $N: y=0$ is denoted as
the ground, and an open semi-infinite nozzle is bounded by the
nozzle walls $N_1$ and $N_2$. $N_i$ is bounded by $x=g_i(y)\in
C^{2,\alpha}[H,H_i)$ with $g_1(y)<g_2(y)$ and $H<H_1<H_2$, $i=1,2$.
$H$ is the distance between the orifice of the nozzle and the
ground. Without loss of generality, we assume that
$$g_i(H)=(-1)^{i},\quad\quad \lim_{y\rightarrow
H_i^-}g_i(y)=-\infty\ \ \text{ and the width of the orifice is 2}.$$
Denote $A_1=(-1,H)$ and $A_2=(1,H)$ be the endpoints of the
semi-infinite nozzle. The gravity acts in the negative
$y$-direction.

\begin{figure}[!h]
\includegraphics[width=100mm]{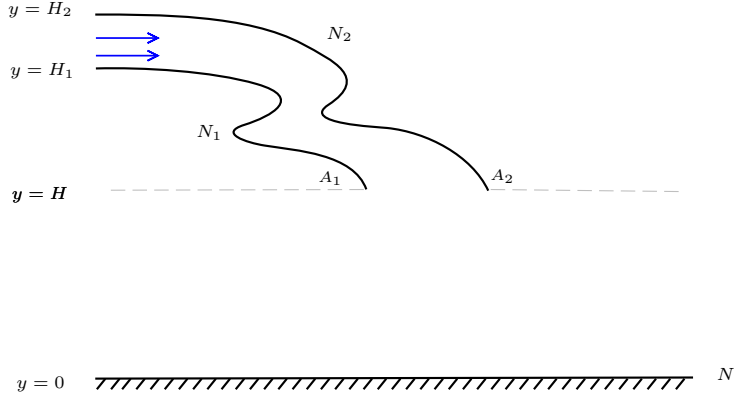}
\caption{The nozzle walls and the ground}\label{f1}
\end{figure}

The resulting dynamics of the impinging jet is governed by the
continuity equation and momentum equations in two dimensions,
\be\label{a3} \left\{ \ba{l}
u_x+v_y=0,\\
uu_x+vu_y+p_x=0,\\
uv_x+vv_y+p_y=-g.
 \ea \right.\ee
Here, $(u,v)$ is the velocity field, $p$ is the pressure of the
incompressible fluid, and $g$ is the acceleration due to the
gravity. The irrotational condition is written as
\be\label{a4}v_x-u_y=0.\ee

The nozzle walls $N_1\cup N_2$ and the ground $N$ are assumed to be
impermeable, and then the velocity satisfies the following no-flow
condition, \be\label{a6}(u,v)\cdot\vec{n}=0\ \ \text{on $N\cup
N_1\cup N_2$},\ee where $\vec n$ is the outer normal to the
boundaries.

In this paper, we will seek an impinging jet acted on by gravity and
with two free streamlines, and the free streamlines initiate from
the endpoints $A_1$ and $A_2$ of the nozzle, and extend to infinity.

Let $\Gamma_1$ and $\Gamma_2$ be the free streamlines connecting the
nozzle walls $N_1$ and $N_2$, respectively. However, the location of
the free streamlines is not known a priori. If the surface tension
is negligible, the pressure on the free surface should be the
constant atmospheric value $p_{atm}$. $\O_0$ denotes the flow region
(which is unknown), bounded by the nozzle wall $N_1\cup N_2$, the
free boundaries $\Gamma_1\cup\Gamma_2$ and the ground $N$ (see
Figure \ref{f2}).

\begin{figure}[!h]
\includegraphics[width=100mm]{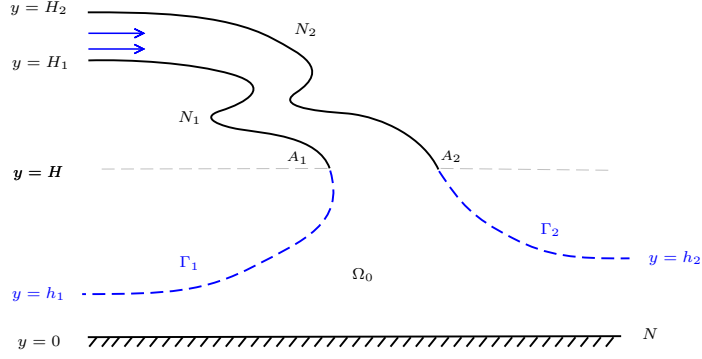}
\caption{Impinging jet flow under gravity}\label{f2}
\end{figure}

Furthermore, the following notations will be used,
\be\label{a10}\begin{aligned}&\text{$Q$~:~~the total incoming flux
in upstream},\\ &\text{$Q_1$~:~~the effluent flux in negative
$x$-direction},\\
&\text{$Q_2$~:~~the effluent flux in positive
$x$-direction}.\end{aligned}\ee It should be noted that the total
incoming flux $Q=Q_1+Q_2$ is imposed in our problem and the
quantities $Q_1$ and $Q_2$ are unknown a priori, which can be
determined by the solution itself.

Now the impinging jet flow problem can be reformulated as: For any
given total mass flux $Q$, determine an incompressible flow
$(u,v,p)$ with free streamlines $\Gamma_1$ and $\Gamma_2$ connecting
smoothly at the endpoints $A_1$ and $A_2$, on which the pressure is
the given atmospheric pressure $p_{atm}$.

More precisely, a solution to the incompressible impinging jet
problem is defined as follows.
 %It
%follows from \eqref{a3} that
%\be\label{a5}(u,v)\cdot\g\left(\f{u^2+v^2}2+p+gy\right)=0\ \
%\text{in the fluid field}.\ee

%In this paper, we will seek an incompressible impinging jet with
%gravity satisfying the following properties.

\begin{definition}\label{def1} (a solution to the impinging jet
problem) A vector $(u,v,p,\Gamma_1,\Gamma_2)$ is called a solution
to the impinging jet problem, provided that the following properties
hold:

{\bf Property 1.} The free streamlines $\Gamma_1$ and $\Gamma_2$ can
be described by $C^1$-smooth functions $x=k_1(y)$ and $x=k_2(y)$,
respectively, and there exist two constants $h_1,h_2\in (0,H)$, such
that
$$\lim_{y\rightarrow h_1^+}k_1(y)=-\infty\ \ \text{and}\ \ \lim_{y\rightarrow h_2^+}k_2(y)=+\infty.$$
%The constant $h_1$ and $h_2$ are called the asymptotic widths of the
%impinging jet in left and right, respectively. Furthermore,

{\bf Property 2.} The free boundaries $\Gamma_1$ and $\Gamma_2$ are
analytic, and satisfy \be\label{a11}g_1(H)=k_{1}(H)=-1 \ \
\text{and}\ \ g_2(H)=k_{2}(H)=1,\ \ee and
\be\label{a12}g'_1(H+0)=k'_{1}(H-0)\ \ \text{and} \ \
g'_2(H+0)=k'_{2}(H-0).\ee

{\bf Property 3.} $(u,v,p)\in \left(C^{1,\alpha}(\O_0)\cap
C^{0}(\overline{\O}_0)\right)^3$ solves the steady incompressible
Euler system \eqref{a3}, the irrotational condition \eqref{a4} and
the boundary condition \eqref{a6}.

{\bf Property 4.} $p=p_{atm}$ on $\Gamma_1$ and $\Gamma_2$.

\end{definition}

\begin{remark} The constants $h_1$ and $h_2$ are indeed the {\it asymptotic
widths} of the impinging jet in left and right downstream,
respectively. The property 1 in Definition \ref{def1} implies that
the free boundaries $\Gamma_1$ and $\Gamma_2$ can not oscillate in
downstream.
\end{remark}

\begin{remark} The conditions \eqref{a11} and \eqref{a12} are so-called {\it continuous fit conditions}  and {\it smooth fit conditions} for the impinging jet,
 which mean that
the free boundaries connect the nozzle walls smoothly.
\end{remark}

\begin{remark}The constant atmosphere pressure $p_{atm}$ is imposed
arbitrarily in advance.
\end{remark}

\subsection{Main results} The main results of this paper are stated as
follows. The first one is on the existence.
\begin{theorem}\label{the1} For any given atmosphere pressure $p_{atm}$
and total incoming flux $Q>2\sqrt{gH^3}$, there exists a solution
$(u,v,p,\Gamma_1,\Gamma_2)$ to the incompressible impinging jet
problem with gravity. Furthermore,\\
(1) the vertical velocity of the impinging jet is negative, namely,
$v<0$ in $\O_0\cup\Gamma_1\cup\Gamma_2$.\\ (2) there exists a unique
smooth streamline $\Gamma:x=k(y)$, which separates the impinging jet
with two different downstream,  and $\Gamma$ goes to the inlet of
the nozzle and intersects the ground $N$ at the unique point $S$
(see Figure \ref{f3}).\\
(3) The streamline $\Gamma$ is perpendicular to the ground $N$ at
$S$, namely, $k'(0+0)=0$.

\end{theorem}

\begin{figure}[!h]
\includegraphics[width=100mm]{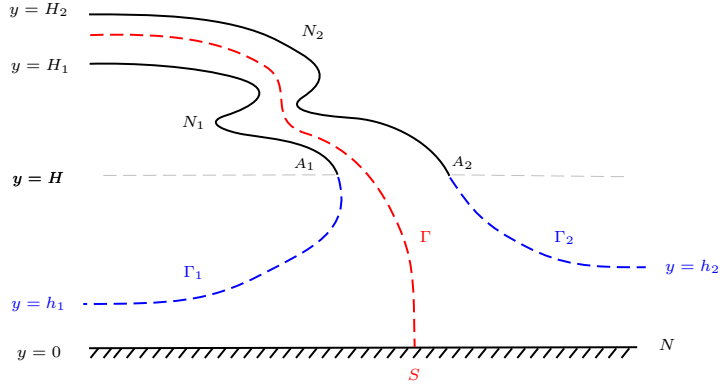}
\caption{The interface $\Gamma$}\label{f3}
\end{figure}

Next, we will give the asymptotic behavior of the incompressible
impinging jet under gravity.

\begin{theorem}\label{the2} Assume that there exists a large $R_0>1$,
such that the nozzle wall $N_i$ can be described by $y=g_i^{-1}(x)$
for $x<-R_0$, $i=1,2$, then the impinging jet flow obtained in
Theorem \ref{the1} satisfies the following asymptotic behavior in
downstream, \be\label{a17}\left(u,v,p\right)\rightarrow
\left(-\f{Q_1}{h_1},0, p_1(y)\right),\ \g (u,v)\rightarrow 0 \
\text{and}\ \ \g p\rightarrow \left(0,-g\right)\ee
 uniformly in any compact subset of $(0,h_1)$ as $
x\rightarrow -\infty$, and
\be\label{a18}\left(u,v,p\right)\rightarrow \left(\f{Q-Q_1}{h_2},0,
p_2(y)\right),\ \g (u,v)\rightarrow 0\ \text{and}\ \ \g p\rightarrow
\left(0,-g\right)\ee
 uniformly in any compact subset of $(0,h_2)$ as $
x\rightarrow +\infty$, where $Q_1$ is determined uniquely by
$$\f{Q_1^2}{h_1^2}+2gh_1=\f{(Q-Q_1)^2}{h_2^2}+2gh_2,$$
$p_1(y)=p_{atm}+g(h_1-y)$ for $y\in(0,h_1)$ and
$p_2(y)=p_{atm}+g(h_2-y)$ for $y\in(0,h_2)$.

Similarly, in upstream, \be\label{a19}\left(u,v,p\right)\rightarrow
\left(\f{Q}{H_2-H_1},0, p_0(y)\right),\ \g (u,v)\rightarrow 0\
\text{and}\ \ \g p\rightarrow \left(0,-g\right)\ee
 uniformly in any compact subset of $(H_1,H_2)$ as $
x\rightarrow -\infty$, where
$p_0(y)=p_{atm}+g(h_1-y)+\f{Q_1^2}{2h_1^2}-\f{Q^2}{2(H_1-H_2)^2}$
for $y\in(H_1,H_2)$.

\end{theorem}

%{\color{red}In the significant work \cite{ACF10} by Alt, Caffarelli
%and Friedman, they obtained the higher regularity of the free
%boundary at the separation point. Consequently, applying the results
%in Section 7 in \cite{ACF10}, one has
%\begin{theorem}\label{the4} $N_i\cup\Gamma_{i}$ is $C^2$ at $A_i$ or $N_i\cup\Gamma_{i}$ is only in $C^{1,\f12}$ at $A_i$
%and the curvature of $\Gamma_{i}$ goes to $\pm\infty$ as
%$x\rightarrow0^+$, $i=1,2$.

%\end{theorem}

%\begin{remark}
%The results of Theorem \ref{the4} imply that $$\text{the curvature
%along $\Gamma_{i}$ tends to the curvature of $N_i$ at $A_i$,}$$ or
%$$\text{the curvature of $\Gamma_{i}$ tends to $\pm\infty$ in absolute value as one approaches $A_i$ along $\Gamma_{i}$}$$ for $i=1,2$.
%The second case is so-called {\it abrupt separation}.
%\end{remark}}

We would like to give some comments on the results as follows.

\begin{remark} In the significant work \cite{ACF2},
the authors showed that in general there does not exist an
asymmetric jet in gravity field issuing from a nozzle with infinite
height, which satisfies the continuous fit conditions. However, the
semi-infinite channel considered here possesses finite height, which
seems physically reasonable and yet does not fall into the class
studied in \cite{ACF2}. We establish the existence of the
incompressible asymmetric impinging jet with continuous fit
conditions under gravity for this class nozzles. This is the one of
main differences to the work \cite{ACF2}. Meanwhile, ideas developed
in this paper may provide a different way to establish the
well-posedness of asymmetric jet flows under gravity without the
horizontal plate $N$, by taking the limit $H\rightarrow+\infty$.
This will be considered in the future.
\end{remark}

\begin{remark} The mechanical behavior of the jet near the orifice
is  quite complicated because the flow has to adjust from a
distribution compatible with the flow in a nozzle to the flow in a
jet. Here we use the continuous and smooth fit conditions near the
orifice, which is physically acceptable in the sense of Brillouin in
\cite{BM} that the detachment is smooth at the fixed detached point.
Yet, to ensure that such conditions are fulfilled mathematically is
one of key elements to establish the existence of the impinging jet
flow in this paper. The main new ideas are as follows. First, due to
the Bernoulli's law, the constant pressure condition on the free
streamlines implies that
$$\f12(u^2+v^2)+gy=constant \ \ \text{on}\ \ \Gamma_1\cup\Gamma_2,$$
and this constant, denoted as $\ld$ is undetermined at the present
stage. On the other hand, as pointed out before, the effluent flux
$Q_1$ is also not imposed here. Thus, to solve the impinging jet
problem, we can treat the two quantities $(\ld,Q_1)$ as a pair of
free parameters which will be adjusted to guarantee the continuous
and smooth fit conditions. This is also the main difference from the
problem of free jets without solid nozzle walls. It should be also
mentioned that there are literatures on numerical and analytical
results on the free jets of ideal fluids, see \cite{CS,HW,TU} and
the references therein.
\end{remark}

\begin{remark} Recall that in the absence of gravity, the existence of an impinging jet with two asymptotic directions and non-existence of an impinging jet with only one asymptotic direction have been shown in
\cite{CDW}.
However, in this paper, our proof still excludes the critical cases
$Q_1=0$ and $Q_1=Q$ and shows that $Q_1\in(0,Q)$ is the sufficient
condition to the existence of impinging jet under gravity. This
coincides with the results on impinging jets without gravity.
\end{remark}

\begin{remark} One of the main differences between the impinging jets in \cite{ACF1,ACF2,ACF3} and
general jets is the occurrence of the interface between the two
fluids with different downstream. Here, we have to show the
existence of the interface, and the uniqueness of the intersection
point of the interface and the ground. In fact, we will show that
the intersection point is a unique stagnation point in the fluid
domain, furthermore, the interface $\Gamma$ intersects the ground
perpendicularly at the stagnation point $S$ in Proposition
\ref{lg1}.
\end{remark}

\begin{remark} It should be noted that the assumption $Q>2\sqrt{gH^3}$ in Theorem
\ref{the1} is a sufficient condition to exclude the stagnation point
in the fluid domain (see Remark \ref{rb1}), especially on the free
boundary. The advantage of this fact lies in exclusion the possible
singularity on the free boundary of water waves with gravity. The
singularity of the free surface flows with gravity at the stagnation
point has been studied extensively in the elegant works
\cite{AFT,P,W,V,VW1,VW2}, which is closely related to a very
interesting problem, the so-called {\it Stokes Conjecture} (Stokes
\cite{S} conjectured in 1880 that, at any stagnation point the free
surface has a symmetric corner of $\f{2\pi}{3}$). Therefore, unlike
the results in \cite{ACF2}, we do not restrict the deflection angle
of the nozzle wall at the endpoints in this paper.
\end{remark}

The remainder of this paper is organized as follows. Section 2
describes the set-up of the physical problem and formulates a free
boundary problem for the stream function with two parameters
$(\ld,Q_1)$. The solvability of the free boundary problem for any
parameters follows from the standard variational approach in Section
2. To verify the continuous fit and smooth fit conditions, we will
show that there exists a pair of parameters $(\ld,Q_1)$, such that
the desired conditions hold in Section 3.  We give the asymptotic
behavior of the impinging jet in Section 4. Finally, we will
investigate the existence, regularity and the properties of the
interface $\Gamma$ and the branching point $S$ in Section 5.

\section{Mathematical formulation of the impinging flow problem}

In this section, we will formulate a boundary value problem with
free boundaries, and furthermore, solve the free boundary problem by
the variational method developed by Alt, Caffarelli and Friedman in
\cite{AC1,ACF1,ACF2}. For convenience to the readers, we will sketch
the details and point out the differences as follows.

\subsection{A free boundary problem} Due to the continuity equation,
there exists a stream function $\psi(x,y)$, such that
 \be\label{b1}\f{\p\psi}{\p y}=u\ \ \text{and}\ \
\f{\p\psi}{\p x}=-v.\ee The slip boundary conditions on the nozzle
walls and the free boundaries imply that $N_1\cup\Gamma_1$ and
$N_2\cup\Gamma_2$ are level sets of the stream function, and then
without loss of generality, one assumes that
$$\psi=0\ \ \text{on}\ \  N_1\cup\Gamma_1,\  \psi=Q_1\ \ \text{on}\ \ N \ \ \text{and}\ \psi=Q\ \ \ \text{on}\ \ \ N_2\cup \Gamma_2.$$
As mentioned before, in this work, we will seek an impinging jet
under gravity with negative vertical velocity, and then it is
reasonable to consider the free boundaries below the horizontal line
$y=H$. Hence, we may assume that the possible fluid region is
contained in the domain $\O$, bounded by $N_1,N_2,L_1,L_2$ and $N$,
where
$$L_1=\{(x,H)\mid x\leq-1\}\ \ \text{and}\ \ L_2=\{(x,H)\mid x\geq1\}.$$

%Denote $\O$ as the possible fluid field, bounded by $N,N_1,N_2$,
%$\CHI_1=\{(x,0)\mid x<-1\}$ and $\CHI_2=\{(x,0)\mid x>1\}$.
%$D=\O\cap\{y<H\}$.

Moreover, we can define the free streamlines $\Gamma_1$ and
$\Gamma_2$ as\be\label{b4}\Gamma_1=\{0<y<H\}\cap\p\{\psi>0\}\ \
\text{and}\ \ \Gamma_2=\{0<y<H\}\cap\p\{\psi<Q\},\ee and the
interface $\Gamma$ as
$$\Gamma=\O\cap\{\psi=Q_1\}.$$ Thus, the fluid region
$\O_0=\O\cap\{0<\psi<Q\}$.

The irrotational condition implies
$$\Delta\psi=0\ \ \text{in}\ \ \O\cap\{0<\psi<Q\},$$ and the
Bernoulli's law gives that
\be\label{b2}\f{|\g\psi|^2}2+gy+p=\text{constant}\ \ \text{in}\ \
\O\cap\{0<\psi<Q\}.\ee This, together with the constant pressure
condition on the free streamlines, implies that $\f{|\g\psi|^2}2+gy$
remains a constant on $\Gamma_1\cup\Gamma_2$, we denote this
constant as $\ld$. Hence, we formulate the following free boundary
value problem in terms of a stream function.
\be\label{b6}\left\{\ba{ll} &\Delta\psi=0  \ \text{in}~~\O\cap\{0<\psi<Q\},\\
&\f{\p\psi}{\p\nu}=-\sqrt{2\lambda-2gy}\ \ \text{on}~~\Gamma_1, \ \f{\p\psi}{\p\nu}=\sqrt{2\lambda-2gy}\ \ \text{on}~~\Gamma_2,\\
&\psi=0\ \text{on}\  N_1 \cup\Gamma_1,\ \psi=Q \ \text{on} \ N_2\cup
\Gamma_2,
 \ \ \psi=Q_1\
\text{on}\ N, \ea\right. \ee where $\nu$ is the outer normal to the
free boundaries.

The problem is now to find a stream function $\psi(x,y)$ solving the
free boundary problem \eqref{b6} with the continuous fit and smooth
fit conditions. It should be emphasized that there are two
undetermined parameters $\ld$ and $Q_1$ in the free boundary
problem, which will be determined by continuous fit conditions. In
other words, we will solve the free boundary problem for any $\ld$
and $Q_1$ first, and then show the existence of a pair of parameters
$(\ld,Q_1)$ to guarantee the continuous fit conditions.

Once the free boundary problem \eqref{b6} has been solved, the
velocity field is given by \eqref{b1} and the pressure is given by
Bernoulli's law, and then the free boundaries $\Gamma_1$ and
$\Gamma_2$ are determined by \eqref{b4}.

\subsection{On the asymptotic widths}
In this subsection, we will find the relationship between the
parameters $(\ld,Q_1)$ and the asymptotic widths $h_1$ and $h_2$ of
the impinging jet in downstream if it exists.

\begin{prop}\label{lb0}  For any given $Q>2\sqrt{gH^3}$,

(1)  the parameters $\ld>0$ and $Q_1\in[0,Q]$ can be determined
uniquely by the asymptotic heights $h_1\in[0,H]$ and $h_2\in[0,H]$.

(2) for any given $Q_1\in(0,Q)$ and
$\ld\geq\f{\max\{Q_1^2,(Q-Q_1)^2\}}{2H^2}+gH$, there exists a unique
pair of asymptotic heights ($h_1,h_2$)  of the impinging jet flow
with $h_1\in(0,H]$ and $h_2\in(0,H]$.
 Furthermore, $h_1$ increases
with respect to $Q_1$ while $h_2$ decreases with respect to $Q_1$.
\end{prop}

\begin{proof}  (1). It is clear that \be\label{b3}\text{$Q_1=\sqrt{2\ld-2gh_1}h_1$
and $Q-Q_1=\sqrt{2\ld-2gh_2}h_2$}.\ee

 {\bf Case 1}. $h_1=h_2>0$, one has
$$Q_1=\f{Q}2\ \ \text{and}\ \ \ld=\f{Q^2}{8h_1^2}+gh_1=\f{Q^2}{8h_2^2}+gh_2.$$

{\bf Case 2}. $h_1h_2=0$, one has
$$Q_1=0\ \ \text{and}\ \ \ld=\f{Q^2}{2h_2^2}+gh_2\ \ \text{for}\ \
h_1=0,$$ and $$Q_1=Q\ \ \text{and}\ \ \ld=\f{Q^2}{2h_1^2}+gh_1\ \
\text{for}\ \ h_2=0.$$

{\bf Case 3}. $h_1\neq h_2$ and $h_1h_2>0$, it follows from
\eqref{b3} that
$$Q_1^2=2\ld h_1^2-2gh_1^3\ \ \text{and}\ \ (Q-Q_1)^2=2\ld
h_2^2-2gh_2^3,$$ which implies that
$$(h_2^2-h_1^2)Q_1^2+2h_1^2QQ_1-(h_1^2Q^2+2gh_1^2h_2^2(h_2-h_1))=0.$$
Then we have
\be\label{bb3}Q_1=\f{-Qh_1^2+\sqrt{Q^2+2g(h_2^2-h_1^2)(h_2-h_1)}h_1h_2}{h_2^2-h_1^2},\ee
or$$Q_1=\f{-Qh_1^2-\sqrt{Q^2+2g(h_2^2-h_1^2)(h_2-h_1)}h_1h_2}{h_2^2-h_1^2}\left\{\ba{ll}
<0,\ \ \ \ \ \ &\text{if}\ \
h_1<h_2,\\
>Q,\ \ \ \ \ \ &\text{if}\ \
h_1>h_2,\ea\right.$$
%$$Q_1=\f{-Qh_1^2-\sqrt{Q^2+2g(h_2^2-h_1^2)(h_2-h_1)}h_1h_2}{h_2^2-h_1^2}<0\ \ \text{{\color{blue}(which can be
%dropped)}}.$$
which can be dropped. It should be noted that
$$\sqrt{Q^2+2g(h_2^2-h_1^2)(h_2-h_1)}>Q\ \ \text{for any $h_1\neq h_2$}.$$ It suffices to show that
$Q_1$ described in the formula \eqref{bb3} lies in $[0,Q]$. In fact,
a direct calculation gives that$$\ba{rl}Q_1
=\f{h_1(Q^2+2g(h_2-h_1)h_2^2)}{Qh_1+\sqrt{Q^2+2g(h_2^2-h_1^2)(h_2-h_1)}h_2}.\ea$$
and$$\ba{rl}Q_1-Q=&\f{-Qh_2^2+\sqrt{Q^2+2g(h_2^2-h_1^2)(h_2-h_1)}h_1h_2}{h_2^2-h_1^2}\\
%=&\f{h_2(-Q^2h_2^2+(Q^2+2g(h_2^2-h_1^2)(h_2-h_1))h_1^2)}{(h_2^2-h_1^2)(Qh_2+\sqrt{Q^2+2g(h_2^2-h_1^2)(h_2-h_1)}h_1)}\\
=&\f{h_2(-Q^2+2g(h_2-h_1)h_1^2)}{Qh_2+\sqrt{Q^2+2g(h_2^2-h_1^2)(h_2-h_1)}h_1}.\ea$$
Noting that $Q_1\geq 0$ and $Q_1-Q\leq 0$, due to the condition
$Q>2\sqrt{gH^3}$.

Hence, $Q_1\in[0,Q]$ can be uniquely determined by $h_1$ and $h_2$.
Furthermore, $\ld=\f{Q_1^2}{2h_1^2}+gh_1$ can be uniquely determined
by $h_1$ and $Q_1$.

(2).  Since $Q>2\sqrt{gH^3}$, we have $Q_1>\sqrt{gH^3}$ or
$Q-Q_1>\sqrt{gH^3}$. Without loss of generality, we assume that
$Q_1>\sqrt{gH^3}$.

Recall that $\ld=\f{Q_1^2}{2h_1^2}+gh_1=\f{(Q-Q_1)^2}{2h_2^2}+gh_2$.
Set
$$\Lambda_1(t)=\f{Q_1^2}{2t^2}+gt\ \ \text{and}\ \
\Lambda_2(t)=\f{(Q-Q_1)^2}{2t^2}+gt$$ for any fixed $Q_1\in[0,Q]$
and $t>0$. It is easy to check that
\be\label{b002}\Lambda'_1(t)=\f{gt^3-Q_1^2}{t^3}<0\ \ \text{for any
$t\in(0,H]$}.\ee Since
$\ld\geq\f{\max\{Q_1^2,(Q-Q_1)^2\}}{2H^2}+gH\geq\Lambda_1(H)$, it
follows from \eqref{b002} that $h_1\in(0,H]$ is determined uniquely
by $\ld$ and $Q_1$.

Similarly, we have $$\text{ $\Lambda_2'(t)=-\f{(Q-Q_1)^2}{t^3}+g$
and $\Lambda_2''(t)=\f{3(Q-Q_1)^2}{t^4}>0$ for any $t>0$.}$$ Set
$t_0=\f{(Q-Q_1)^{\f23}}{g^{\f13}}$, such that $\Lambda_2'(t_0)=0$.
Then we consider the following two cases.

{\bf Case 1.} $t_0\geq H$ (see Figure \ref{f17}).
\begin{figure}[!h]
\includegraphics[width=80mm]{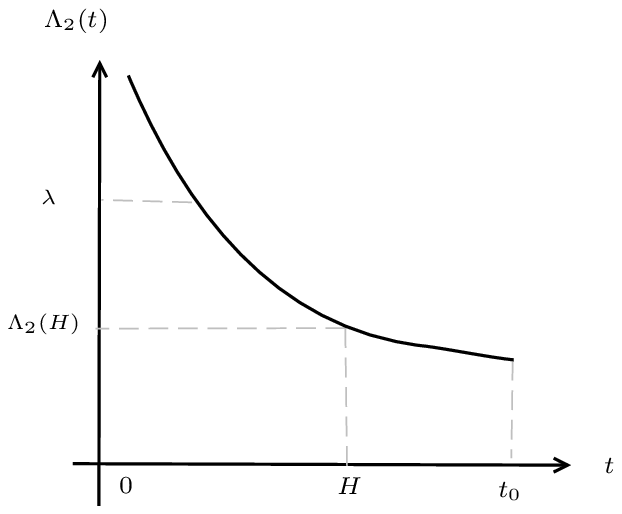}
\caption{The relation $\Lambda_2(t)$ and $t$}\label{f17}
\end{figure}

Then $\Lambda_2'(t)<0$ for any $t\in(0,H]$. On another side, the
fact $\ld\geq\f{\max\{Q_1^2,(Q-Q_1)^2\}}{2H^2}+gH\geq\Lambda_2(H)$
implies that $h_2\in(0,H]$ can be determined uniquely by
$\ld=\f{(Q-Q_1)^2}{2h_2^2}+gh_2$.

{\bf Case 2.} $t_0<H$ (see Figure \ref{f4}), which is equivalent to
$Q-Q_1<\sqrt{gH^3}$. This together with $Q_1>\sqrt{gH^3}$ gives that
$$\Lambda_2(H)=\f{(Q-Q_1)^2}{2H^2}+gH<\f{Q_1^2}{2H^2}+gH\leq\ld,$$ which implies that $h_2\in(0,H]$ is determined uniquely
by $\ld$ and $Q_1$, namely, $\ld=\f{(Q-Q_1)^2}{2h_2^2}+gh_2$,
\begin{figure}[!h]
\includegraphics[width=80mm]{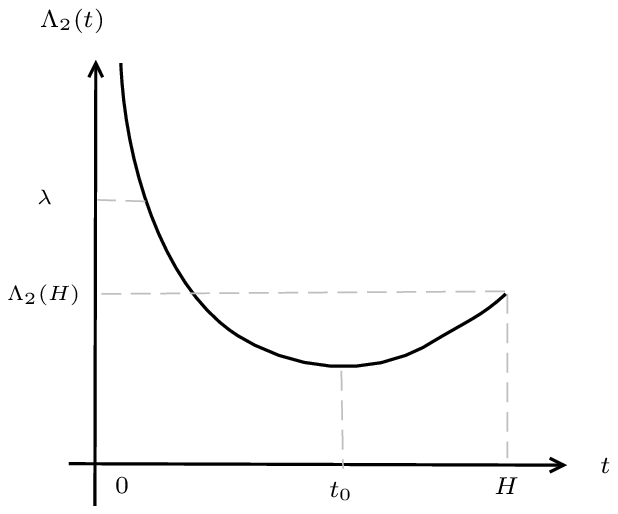}
\caption{The relation $\Lambda_2(t)$ and $t$}\label{f4}
\end{figure}
Furthermore,  we can conclude that $h_1$ is increasing with respect
to $Q_1$ and $h_2$ is decreasing with respect to $Q_1$.

\end{proof}

\begin{remark}\label{rb1}
Note that $\ld=\f{Q_1^2}{2h_1^2}+gh_1=\f{(Q-Q_1)^2}{2h_2^2}+gh_2$.
It follows from the proof of Proposition \ref{lb0} that the
condition $Q>2\sqrt{gH^3}$ ensures that one can choose $\ld$
depending on $h_1$ and $h_2$ monotonically provided that
$$\ld\geq\f{\max\{Q_1^2,(Q-Q_1)^2\}}{2H^2}+gH.$$ Clearly, the right
hand side has the following lower bound
$$\f{Q^2}{8H^2}+gH,$$ which excludes the stagnation point on the free boundary $\Gamma_1\cup\Gamma_2$.

In particularly, for any
$(Q_1,\ld)\in[0,Q]\times\left[\f{\max\{Q_1^2,(Q-Q_1)^2\}}{2H^2}+gH,+\infty\right)$,
there exists a unique pair $(h_1,h_2)$, such that
$$\ld=\f{Q_1^2}{2h_1^2}+gh_1=\f{(Q-Q_1)^2}{2h_2^2}+gh_2.$$
%Noting that $\ld=\f{Q_1^2}{2h_1^2}+gh_1=\f{(Q-Q_1)^2}{2h_2^2}+gh_2$,
%the condition $Q>2\sqrt{gH^3}$ guarantees the monotonicity of $\ld$
%with respect to $h_1$ and $h_2$, which also gives a lower bound to
%the parameter $\ld$ as
%$$\ld\geq\f{\max\{Q_1^2,(Q-Q_1)^2\}}{2H^2}+gH\geq\f{Q^2}{8H^2}+gH.$$
%On the other hand, we have
%\be\label{b002}\f{d}{dt}\left(\sqrt{2\ld-2gt}t\right)=\f{2\ld-2gt}{\sqrt{2\ld-3gt}}\geq\f{3gH-2gt}{\sqrt{2\ld-3gt}}>0\
%\ \text{for any $t\in(0,H)$}.\ee
\end{remark}

\subsection{The variational approach}
To solve the boundary value problem \eqref{b6}, we will use the
variational method, which was developed by Alt, Caffarelli and
Friedman in \cite{AC1,ACF1,ACF2} for two-dimensional asymmetric jets
and jets under gravity.

Define a functional with a parameter $\ld$ as follows,
 \be\label{b7}J_{\ld}(\psi)=\int_{\Omega}|\nabla
\psi|^2+(2\ld-2gy)\chi_{\{0<\psi<Q\}\cap D} dxdy,\ee where $\chi_E$
is the characteristic function of a set $E$ and $D=\{0<y<H\}$.
Define an admissible set with a parameter $Q_1$,
$$\ba{rl} K_{Q_1}=\{\psi\in H^1_{loc}(\mathbb{R}^2)\mid &\psi=Q_1\ \text{below}\ N, \psi=Q\ \text{in the right of}\ N_2\cup L_2,\\
 &\psi=0\ \text{in the left of}\ N_1\cup L_1\}.\ea $$

However, noting that the functional $J_{\ld}(\psi)=+\infty$ for any
$\psi\in K_{Q_1}$, we will consider the variational problem in a
truncated domain. For any $\mu>1$, denote
\be\label{b70}\begin{aligned}&\text{$H_{1,\mu}=\max\{y\mid g_1(y)=-\mu\}$}\ \ &\text{$H_{2,\mu}=\min\{y>H_{1,\mu}\mid g_2(y)=-\mu\}$},\\
&\text{$N_{1,\mu}=\{(x,y)\mid x=g_1(y), H\leq y\leq H_{1,\mu}\}$},\ \ &\text{$L_{1,\mu}=L_1\cap\{x\geq-\mu\}$},\\
&\text{$N_{2,\mu}=\{(x,y)\mid x=g_2(y), H\leq y\leq H_{2,\mu}\}$},\
\
&\text{$L_{2,\mu}=L_2\cap\{x\leq\mu\}$},\\
&\text{$\sigma_{1,\mu}=\{x=-\mu,0\leq y\leq H\}$},\ \ &\text{
$\sigma_{2,\mu}=\{x=\mu,0\leq y\leq H\}$},\\
 &\text{$N_{\mu}=N\cap\{-\mu \leq
x\leq\mu\}$},\ \ &\text{$\sigma_{\mu}=\{x=-\mu, H_{1,\mu}\leq y\leq
H_{2,\mu}\}$},\end{aligned}\ee  and $\O_{\mu}$ is bounded by
$N_{1,\mu}, N_{2,\mu}, L_{1,\mu},L_{2,\mu},\sigma_\mu,
\sigma_{1,\mu}, \sigma_{2,\mu}$ and $N_\mu$ (see Figure \ref{f5}).
\begin{figure}[!h]
\includegraphics[width=100mm]{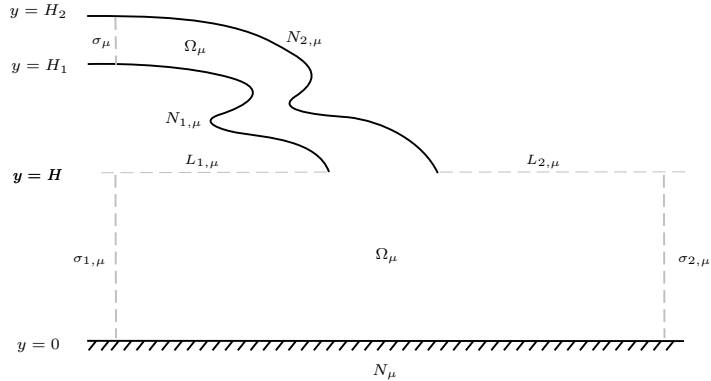}
\caption{The truncated domain $\O_\mu$}\label{f5}
\end{figure}

Hence, one may consider the following truncated variational problem
in the bounded domain $\O_{\mu}$. Define a functional
$$J_{\ld,\mu}(\psi)=\int_{\Omega_\mu}|\nabla
\psi|^2+(2\ld-2gy)\chi_{\{0<\psi<Q\}\cap D_\mu} dxdy,$$ where
$D_\mu=\{-\mu<x<\mu,0<y<H\}$.

 {\bf The truncated variational problem $(P_{\ld,Q_1,\mu})$:} For any
 $\mu>1$, $Q_1\in[0,Q]$ and $\ld\geq\f{\max\{Q_1^2,(Q-Q_1)^2\}}{2H^2}+gH$,
find a $\psi_{\ld,Q_1,\mu}\in K_{\ld,Q_1,\mu}$, such that
$$\ \ \   J_{\ld,\mu}(\psi_{\ld,Q_1,\mu})=\min_{\psi\in K_{\ld,Q_1,\mu}} J_{\ld,\mu}(\psi),$$
where the admissible set $K_{\ld,Q_1,\mu}=\{\psi\in K_{Q_1}\mid
\psi=\Psi_{\ld,Q_1,\mu}(x,y)\ \text{on}\ \p\O_\mu\}$ and the
boundary function
$$\Psi_{\ld,Q_1,\mu}(x, y)=\left\{\ba{ll} 0,\ \ \ \ \ \ &\text{on}\ \
N_{1,\mu}\cup L_{1,\mu},\\
Q,\ \ \ \ \ \ &\text{on}\ \
N_{2,\mu}\cup L_{2,\mu},\\
Q_1,\ \ \ \ \ \ &\text{on}\ \
N_\mu,\\
\max\left\{-\sqrt{2\ld-2gh_1}y+Q_1,0\right\},\ \ \ \ \ \ &\text{on}\
\
\sigma_{1,\mu},\\
\min\left\{\sqrt{2\ld-2gh_2}y+Q_1,Q\right\},\ \ \ \ \ \ &\text{on}\
\
\sigma_{2,\mu},\\
\f{(y-H_{1,\mu})Q}{H_{2,\mu}-H_{1,\mu}}\ \ \ \ \ \ &\text{on}\ \
\sigma_{\mu}.\ea\right.$$ Here, $h_1$ and $h_2$ are the asymptotic
heights of the impinging jets in left and right downstream,
depending uniquely on $\ld$ and $Q_1$ as in \eqref{b3}.

The existence of a minimizer $\psi_{\ld,Q_1,\mu}$ follows directly
from Theorem 1.3 in \cite{AC1}, so the details are omitted here.

The corresponding free boundaries $\Gamma_{1,\ld,Q_1,\mu}$ and
$\Gamma_{2,\ld,Q_1,\mu}$ are defined as
$$\text{$\Gamma_{1,\ld,Q_1,\mu}=D_\mu\cap\p\{\psi_{\ld,Q_1,\mu}>0\}$ and
$\Gamma_{2,\ld,Q_1,\mu}=D_\mu\cap\p\{\psi_{\ld,Q_1,\mu}<Q\}$}. $$

It follows from Lemma 1.3, Lemma 2.3 and Lemma 2.4 in \cite{AC1}
that the following properties hold for the minimizer.
\begin{prop}\label{lb1} The
minimizer $\psi_{\ld,Q_1,\mu}$ satisfies that

(1) $\psi_{\ld,Q_1,\mu}\in C^{0,1}(\O_\mu)$ and is harmonic in
$\Omega_{\mu}\cap\{0<\psi_{\ld,Q_1,\mu}<Q\}$.

(2) $0\leq\psi_{\ld,Q_1,\mu}\leq Q$ in $\O_\mu$ and
$0<\psi_{\ld,Q_1,\mu}<Q$ in $\O_\mu\cap\{y>H\}$.

(3) The free boundaries $\Gamma_{1,\ld,Q_1,\mu}$ and
$\Gamma_{2,\ld,Q_1,\mu}$ are analytic.

(4) $|\g\psi_{\ld,Q_1,\mu}|=\sqrt{2\ld-2gy}$ on
$\Gamma_{1,\ld,Q_1,\mu}\cup\Gamma_{2,\ld,Q_1,\mu}$.

\end{prop}

Next, we will give a uniform estimate to the gradient of the
minimizer away from the endpoints of the nozzle. The proof follows
from the similar arguments in Section 5 in \cite{ACF1} immediately,
and then the details are omitted here.

\begin{lemma}\label{lb60} The minimizer $\psi_{\ld,Q_1,\mu}$ is
continuous in $\bar\O_\mu$, and for any $\delta>0$, if $X\in \bar
\O_\mu$ with $|X-A_1|>\delta$ and $|X-A_2|>\delta$, then there
exists a constant $C$ depending on $\delta$ and $\O_\mu$, such that
$$|D\psi_{\ld,Q_1,\mu}(X)|\leq C(\sqrt{\ld}+Q).$$ Furthermore, for any
domain $K\Subset\O_\mu$ containing a free boundary point, the
Lipschitz coefficient of $\psi_{\ld,Q_1,\mu}$ in $K$ is estimated by
$C\sqrt{\ld}$, where the constant $C$ depends on $K$ and $\O_\mu$,
independent of $Q$.
\end{lemma}

\subsection{The properties of the minimizer and the free boundaries}
Next, we will establish some important properties of the minimizer
and the free boundaries, such as monotonicity, uniqueness and
regularity. First, we consider the monotonicity of the minimizer
with respect to $x$. To this end, we will give a priori bounds of
the minimizer as follows.
\begin{lemma}\label{lb3}
 The minimizer $\psi_{\ld,Q_1,\mu}$ to the truncated variational
problem $(P_{\ld,Q_1,\mu})$ satisfies
\be\label{b80}\max\left\{-\sqrt{2\ld-2gh_1}y+Q_1,0\right\}\leq\psi_{\ld,Q_1,\mu}(x,y)\leq\min\left\{\sqrt{2\ld-2gh_2}y+Q_1,Q\right\}\
 \text{in} \ D_\mu,\ee and
\be\label{b81}\max\left\{\f{(y-H_{1,\mu})Q}{H_{2,\mu}-H_{1,\mu}},0\right\}<\psi_{\ld,Q_1,\mu}(x,y)<
Q\  \quad \text{in}\ \ \O_{\mu}\cap\{y> H\},\ee where $H_{1,\mu}$
and $H_{2,\mu}$ are defined in \eqref{b70}.
\end{lemma}

\begin{remark}\label{rb2} The bound \eqref{b80} gives that
$$
0<-\sqrt{2\ld-2gh_1}y+Q_1\leq \psi_{\ld,Q_1,\mu}(x,y)\quad\quad
\text{for}\quad y< h_1,
$$
and $$\psi_{\ld,Q_1,\mu}(x,y)\leq\sqrt{2\ld-2gh_2}y+Q_1<Q \quad\quad
\text{for}\quad y< h_2,
$$
 which in fact imply that the free boundary
$\Gamma_{i,\ld,Q_1,\mu}$ lies above
 the asymptotic height $h_i$ for $i=1,2$. On another side, the
 bounds in \eqref{b81} imply that the free boundaries lies below
 $y=H$.
 \end{remark}
\begin{proof} Consider the upper bound in \eqref{b80} first. Since $\psi_{\ld,Q_1,\mu}\leq Q$ in $\O_\mu$,
it suffices to obtain the upper bound in \eqref{b80} for the case $Q_1<Q$. %Since
%$\psi_{\ld,Q_1,\mu}$ is Lipschitz continuous near $\{y=0\}$ in Lemma
%\ref{lb60}, we have that $dist(\Gamma_{2,\ld,Q_1,\mu},\{y=0\})\geq
%c>0$.

Define an auxiliary comparison function
$$\text{$\o_\e(y)=\min\left\{\sqrt{2\ld-2g(h_2-\e)}y+Q_1,Q\right\}$
for any $\e\geq 0$,}$$ and a strip
$$D_{\mu,\e}=\O_\mu\cap\left\{y<\f{Q-Q_1}{\sqrt{2\ld-2g(h_2-\e)}}\right\}\ \ \text{for any $\e\geq 0$}.$$

Due to Lemma \ref{lb60}, $\psi_{\ld,Q_1,\mu}$ is Lipschitz
continuous near $N_\mu$, this implies that there is a positive
distance between the right free boundary and $N_\mu$. Thus, take a
sufficiently large $\e>0$, such that
$$\text{$y=\f{Q-Q_1}{\sqrt{2\ld-2g(h_2-\e)}}$ lies below the right free boundary
$\Gamma_{2,\ld,Q_1,\mu}$.}$$ It is easy to check that
$$\text{$\o_\e>\psi_{\ld,Q_1,\mu}$\quad\quad on
$\sigma_{1,\mu}\cup\sigma_{2,\mu}$}$$ for any $\e\geq0$, and it
follows from Proposition \ref{lb1} that
$$\text{$\psi_{\ld,Q_1,\mu}\leq Q=\o_\e\ $\quad\quad on
$\ y=\f{Q-Q_1}{\sqrt{2\ld-2g(h_2-\e)}}$,}$$ for $Q_1<Q$. If $\e$ is
sufficiently large, $\psi_{\ld,Q_1,\mu}$ and $\o_\e$ are harmonic in
$D_{\mu,\e}$. Applying the maximum principle in $D_{\mu,\e}$ gives
that
$$\text{$\psi_{\ld,Q_1,\mu}<\o_\e$ in $D_{\mu,\e}$,}$$ provided that $\e$ is
sufficiently large.

Let $\e_0\geq 0$ be the smallest one and
$y_0=\f{Q-Q_1}{\sqrt{2\ld-2g(h_2-\e_0)}}$ (see Figure \ref{f7}),
such that
$$\psi_{\ld,Q_1,\mu}(x,y)<\o_{\e_0}(y)\ \ \text{in}\ \ D_{\mu,\e_0},\ \ \text{and}\ \ \o_{\e_0}(y_0)=\psi_{\ld,Q_1,\mu}(x_0,y_0)
\ \ \text{for some $x_0\in(-\mu,\mu)$. }$$

\begin{figure}[!h]
\includegraphics[width=100mm]{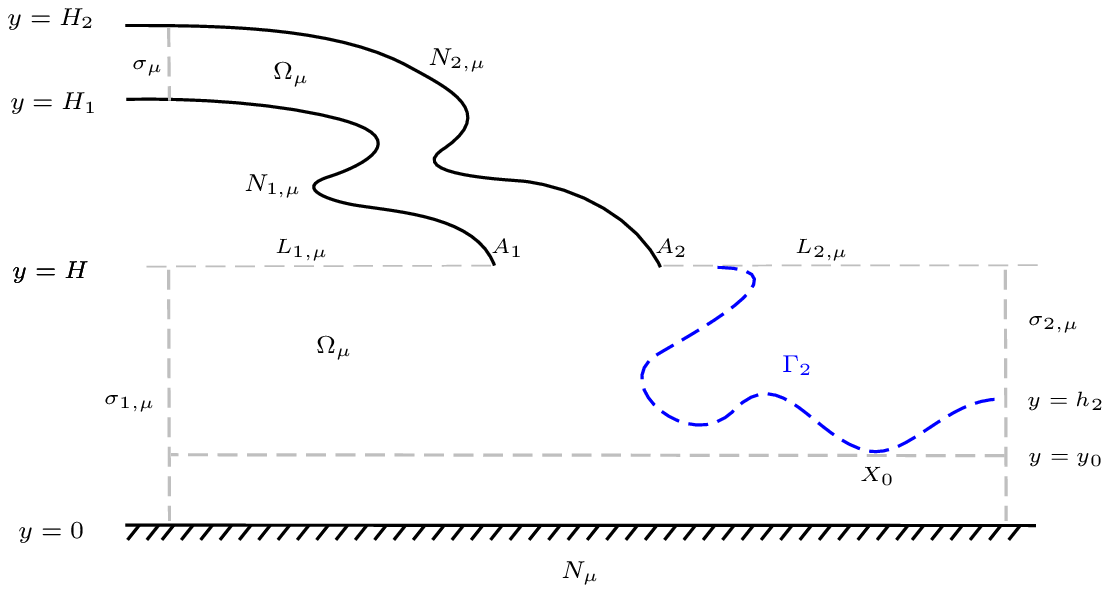}
\caption{$\o_{\e_0}$ and $\psi_{\ld,Q_1,\mu}$}\label{f7}
\end{figure}

It suffices to show that $\e_0=0$, and then the upper bound of the
minimizer follows. If not, then $\e_0>0$, which implies that
$$y_0=\f{Q-Q_1}{\sqrt{2\ld-2g(h_2-\e_0)}}<h_2.$$ Denote $X_0=(x_0,y_0)$. Since
the free boundary $\Gamma_{2,\ld,Q_1,\mu}$ is analytic at $X_0$, it
follows from Hopf's lemma that
$$\sqrt{2\ld-2gy_0}=\f{\p\psi_{\ld,Q_1,\mu}}{\p\nu}>\f{\p\o_{\e_0}}{\p\nu}=\sqrt{2\ld-2g(h_2-\e_0)}=\f{Q-Q_1}{y_0}=\f{\sqrt{2\ld-2gh_2}h_2}{y_0}\ \ \text{at}\ \ X_0,$$
where $\nu=(0,1)$ is the outer normal vector to
$\Gamma_{2,\ld,Q_1,\mu}$ at $X_0=(x_0,y_0)$. This is nothing but
\be\label{b800}y_0\sqrt{2\ld-2gy_0}>h_2\sqrt{2\ld-2gh_2}\ \
\text{for $y_0<h_2$}.\ee On another side, it is clear that the
function $t\sqrt{2\ld-2gt}$ is strictly increasing with respect to
$t\in(0,H)$, which contradicts to \eqref{b800}.

%it is easy to check that $\min\{\psi,\psi_1\}\in K_{Q_1,Q_2,L}$.
%Then we have
%\be\label{b9}\ba{rl}0\leq&J_{Q_1,Q_2,L}(\min\{\psi,\psi_1\}))-J_{Q_1,Q_2,L}(\psi)\\
%=&\int_{\Omega_L}|\nabla\min\{
%\psi,\psi_1\}|^2-|\g\psi|^2dxdy\\
%&+\int_{D_L}(\ld-2gy)\left(I_{\{-Q_2<min\{\psi,\psi_1\}<Q_1\}}-I_{\{0<\psi<Q\}}\right)dxdy\\
%=&I_1+I_2, \ea\ee

%For the right hand terms of \eqref{b9}, we have
%\be\label{b10}\ba{rl}I_1
%=&-\int_{\Omega_L\cap\{\psi>\psi_1\}}|\nabla
%(\psi_1-\psi)|^2dxdy+2\int_{\Omega_L\cap\{\psi>\psi_1\}}\g\psi_1\cdot\nabla
%(\psi_1-\psi)dxdy\\
%=&-\int_{\Omega_L\cap\{Q_1=\psi>\psi_1\}}|\nabla
%\psi_1|^2dxdy-\int_{\Omega_L\cap\{Q_1>\psi>\psi_1\}}|\nabla
%(\psi_1-\psi)|^2dxdy\\
%=&-\int_{\Omega_L\cap\{Q_1=\psi>\psi_1\}}\ld+H
%dxdy-\int_{\Omega_L\cap\{Q_1>\psi>\psi_1\}}|\nabla
%(\psi_1-\psi)|^2dxdy, \ea\ee and \be\label{b11}\ba{rl}I_2
%=\int_{D_L\cap\{Q=\psi>\psi_1\}}(\lambda-y) dxdy, \ea\ee Combining
%\eqref{c9}-\eqref{c11}, we have
% \be\label{b12}\ba{rl}0\leq-\int_{\Omega_L\cap\{Q_1=\psi>\psi_1\}}H+y
%dxdy-\int_{(\Omega_L\setminus D_L)\cap\{Q_1>\psi>\psi_1\}}|\nabla
%(\psi_1-\psi)|^2dxdy, \ea\ee which implies that
%$$\psi(x,y)\leq\psi_1(y)=\min\left\{\sqrt{\ld+H}(y+H),Q_1\right\}\ \ \text{in}\ \ D_L.$$

Similarly, one can show that
$$\psi(x,y)\geq\max\left\{-\sqrt{2\ld-2gh_1}y+Q_1,0\right\}\ \ \text{in}\ \ D_\mu.$$

 Finally, consider the estimate \eqref{b81} in the nozzle. Denote
$G_\mu=\O_\mu\cap\{y>H_{1,\mu}\}$. Applying the maximum principle in
$G_\mu$ gives that
$\f{(y-H_{1,\mu})Q}{H_{2,\mu}-H_{1,\mu}}\leq\psi_{\ld,Q_1,\mu}$ in
$G_\mu$. Since $0<\psi_{\ld,Q_1,\mu}< Q$ in $\O_\mu\cap\{y>H\}$, it
follows
$$\max\left\{\f{(y-H_{1,\mu})Q}{H_{2,\mu}-H_{1,\mu}},0\right\}<\psi_{\ld,Q_1,\mu}(x,y)<
Q\ \quad\quad  \text{in}\ \ \O_{\mu}\cap\{y> H\}.$$

\end{proof}

With the aid of Lemma \ref{lb3}, one can obtain the uniqueness and
monotonicity of the minimizer $\psi_{\ld,Q_1,\mu}$.
\begin{prop}\label{lb4}
 The minimizer $\psi_{\ld,Q_1,\mu}(x,y)$ to the truncated variational
problem $(P_{\ld,Q_1,\mu})$ is increasing with respect to $x$.
Furthermore, the minimizer $\psi_{\ld,Q_1,\mu}$ is unique for any
fixed $\ld,Q_1$ and $\mu$.

\end{prop}

\begin{proof}
%Consider the uniqueness of the minimizer first.
Suppose that $\psi(x,y)=\psi_{\ld,Q_1,\mu}(x,y)$ and
$\t\psi(x,y)=\t\psi_{\ld,Q_1,\mu}(x,y)$ are two minimizers to the
truncated variational problem $(P_{\ld,Q_1,\mu})$. Define
$\t\psi_\e(x,y)=\t\psi(x-\e,y)$ for any $\e\geq0$, the corresponding
functional $J_{\ld,\mu}^\e$ with admissible set $K_{\ld,Q_1,\mu}^\e$
in truncated domain $\O_{\mu}^\e=\{(x,y)\mid(x-\e,y)\in\O_{\mu}\}$.
Extend $\psi(x,y)$ and $\t\psi_\e(x,y)$ to larger domain,
respectively, as
$$\text{$\psi(x,y)=\min\left\{\sqrt{2\ld-2gh_2}y+Q_1,Q\right\}$ for
$\mu\leq x\leq \mu+\e, 0\leq y\leq H$}$$ and
$$\t\psi_\e(x,y)=\left\{\ba{ll}
\max\left\{-\sqrt{2\ld-2gh_1}y+Q_1,0\right\},\  &\text{for}\
-\mu\leq x\leq -\mu+\e, 0\leq y\leq H,\\
\max\left\{\f{(y-H_{1,\mu})Q}{H_{2,\mu}-H_{1,\mu}},0\right\},\
&\text{for}\ -\mu\leq x\leq -\mu+\e, H\leq y\leq
H_{2,\mu}.\ea\right.$$

Set
$$\psi_1=\max\{\psi,\t\psi_\e\}\ \ \text{and}\ \ \psi_2=\min\{\psi,\t\psi_\e\}.$$
Then Lemma \ref{lb3} gives that \be\label{b13}\psi_1=\psi\ \
\text{in $\O_\mu^\e\setminus\O_\mu$ and }\  \psi_2=\t\psi_\e\ \
\text{in $\O_\mu\setminus\O_\mu^\e$},\ee which implies that
$\psi_1\in K_{\ld,Q_1,\mu}$ and $\psi_2\in K_{\ld,Q_1,\mu}^\e$.
Next, we claim that \be\label{b14}J_{\ld,\mu}(\psi)=
J_{\ld,\mu}(\psi_1)\ \ \text{and}\ \ J_{\ld,\mu}^\e(\t\psi_\e)=
J_{\ld,\mu}^\e(\psi_2).\ee Indeed, since $J_{\ld,\mu}(\psi)\leq
J_{\ld,\mu}(\psi_1)$ and $J_{\ld,\mu}^\e(\t\psi_\e)\leq
J_{\ld,\mu}^\e(\psi_2)$, \eqref{b14} will follow from
\be\label{b15}J_{\ld,\mu}(\psi)+J_{\ld,\mu}^\e(\t\psi_\e)=
J_{\ld,\mu}(\psi_1)+ J_{\ld,\mu}^\e(\psi_2).\ee In fact, \eqref{b13}
implies that $\psi\geq\t\psi_\e$ in
$(\O_\mu^\e\setminus\O_\mu)\cup(\O_\mu\setminus\O_\mu^\e)$.
Therefore one has
\be\label{b16}\Omega_\mu\cap\{\psi<\t\psi_\e\}=\Omega_\mu^\e\cap\{\psi<\t\psi_\e\}.\ee
Due to \eqref{b16}, one may get
 \be\label{b17}\ba{rl}&\int_{\Omega_\mu}|\nabla
\psi_1|^2dxdy+\int_{\Omega_\mu^\e}|\nabla
\psi_2|^2dxdy\\
=&\int_{\Omega_\mu\cap\{\psi\geq\t\psi_\e\}}|\nabla
\psi|^2dxdy+\int_{\Omega_\mu\cap\{\psi<\t\psi_\e\}}|\nabla
\t\psi_\e|^2dxdy\\
&+\int_{\Omega_\mu^\e\cap\{\psi\geq\t\psi_\e\}}|\nabla
\t\psi_\e|^2dxdy+\int_{\Omega_\mu^\e\cap\{\psi<\t\psi_\e\}}|\nabla
\psi|^2dxdy\\
=&\int_{\Omega_\mu\cap\{\psi\geq\t\psi_\e\}}|\nabla
\psi|^2dxdy+\int_{\Omega_\mu^\e\cap\{\psi<\t\psi_\e\}}|\nabla
\t\psi_\e|^2dxdy\\
&+\int_{\Omega_\mu^\e\cap\{\psi\geq\t\psi_\e\}}|\nabla
\t\psi_\e|^2dxdy+\int_{\Omega_\mu\cap\{\psi<\t\psi_\e\}}|\nabla
\psi|^2dxdy\\
=&\int_{\Omega_\mu}|\nabla \psi|^2dxdy+\int_{\Omega_\mu^\e}|\nabla
\t\psi_\e|^2dxdy. \ea\ee Similarly, one has
$$\ba{rl}&\int_{D_\mu}(2\ld-2gy)\chi_{\{0<\psi<Q\}}dxdy+\int_{D_\mu^\e}(2\ld-2gy)\chi_{\{0<\t\psi_\e<Q\}}dxdy\\
=&\int_{D_\mu}(2\ld-2gy)\chi_{\{0<\psi_1<Q\}}dxdy+\int_{D_\mu^\e}(2\ld-2gy)\chi_{\{0<\psi_2<Q\}}dxdy,\ea
$$ which, together with \eqref{b17}, gives \eqref{b15}. And hence
the claim \eqref{b14} is proved.

Moreover, we will show that if $0<\psi(X_0)=\t\psi_\e(X_0)<Q$ with
 $X_0\in\O_\mu\cap\O_\mu^\e$, then \be\label{b18}\text{either\quad
$\psi\geq \t\psi_\e$\quad or\quad $\psi\leq\t\psi_\e$\quad in some
neighborhood of $X_0$.}\ee By the continuity of $\psi$, $0<\psi<Q$
in $B_r(X_0)\subset\O_\mu\cap\O_\mu^\e$ for some small $r>0$. Define
a function $\phi$ as follows, \be\label{b180}\Delta\phi=0\ \
\text{in}\ \ B_r(X_0)\ \ \text{and $\phi=\psi_1$ on $\p
B_r(X_0)$}.\ee Suppose that \eqref{b18} is not true, we then claim
that $\psi_1$ is not harmonic in $B_r(X_0)$ for any small $r>0$. In
fact, if $\psi_1$ solves the boundary value problem \eqref{b180},
and it is easy to check that $\Delta\psi=0$ in $B_r(X_0)$ and
$\psi\leq\psi_1$ on $\p B_r(X_0)$, the strong maximum principle
gives that $\psi\equiv\psi_1$ in $B_r(X_0)$, due to
$\psi(X_0)=\psi_1(X_0)$. Namely, $\psi\geq \t\psi_\e$ in $B_r(X_0)$,
which contradicts to our assumption.

Since $\psi_1$ is not harmonic in $B_r(X_0)$, and $\psi_1\neq\phi$
in $B_r(X_0)$. Then it holds that
\be\label{b810}\ba{rl}&\int_{B_r(X_0)}|\g\phi|^2dxdy-\int_{B_r(X_0)}|\g\psi_1|^2dxdy\\
=&-\int_{B_r(X_0)}|\g(\phi-\psi_1)|^2dxdy+2\int_{B_r(X_0)}\g\phi\cdot\g(\phi-\psi_1)dxdy\\
=&-\int_{B_r(X_0)}|\g(\phi-\psi_1)|^2dxdy\\
<&0.\ea\ee Extend $\phi=\psi_1$ outside $B_r(X_0)$, such that
$\phi\in K_{\ld,Q_1,\mu}$. It follows from \eqref{b810} and the fact
that $0<\phi<Q$ in $\overline{B_r(X_0)}$ that
$$J_{\ld,\mu}(\phi)<J_{\ld,\mu}(\psi_1)=J_{\ld,\mu}(\psi),$$
which contradicts to the fact that $\psi$ is a minimizer to the
variational problem $(P_{\ld,Q_1,\mu})$. Hence, the claim
\eqref{b18} is proved.

Next, we will show that
\be\label{b182}\text{$\O_\mu\cap\{0<\psi<Q\}$ and
$\O_\mu^\e\cap\{0<\t\psi_\e<Q\}$ are connected}.\ee

If \eqref{b182} fails, then without loss of generality, we assume
that $\O_\mu\cap\{0<\psi<Q\}$ is not connected. Since $0<\psi<Q$ in
$\O_\mu\cap\{y>H\}$ (see \eqref{b81}), let $E$ be a maximal
connected subset of $\O_\mu\cap\{0<\psi<Q\}$, such that
$\O_\mu\cap\{y>H\}\subset E$. Then there exists a point
$Y_0\in\O_\mu\cap\{0<\psi<Q\}$, such that $Y_0\notin E$. The
continuity of $\psi$ gives that there exists a maximal connected
subset $F$ in $\O_\mu\cap\{0<\psi<Q\}$ and  $Y_0\in F$.

If $\p\O_\mu\cap \p F=\varnothing$, we first show that $\psi=0$ or
$\psi=Q$ on $\p F$. If not, there exists a point $X_0\in \p F$ with
$0<\psi(X_0)<Q$. Due to the continuity of $\psi$ again, there exists
a connected subsect $\t F$ of $\O_\mu\cap\{0<\psi<Q\}$, such that $F
\subseteq\t F$, which contradicts to the definition of the set $F$.
Since $\psi$ is harmonic in $F$, the strong maximum principle gives
that $\psi\equiv0$ or $\psi\equiv Q$ in $F$, which contradicts to
the fact $0<\psi(Y_0)<Q$ with $Y_0\in F$.

If $\p\O_\mu\cap \p F\neq\varnothing$, the definition of $F$ implies
that $\psi\equiv0$ or $\psi\equiv Q$ on $\O_\mu\cap\p F$. Without
loss of generality, one may assume that $\psi\equiv0$ on
$\O_\mu\cap\p F$. Hence, $\psi\equiv0$ on $\p\O_\mu\cap\p F$, since
$\psi$ is continuous up to $\p\O_\mu$ and is monotone with respect
to $y$ on $\p\O_\mu$. Therefore, $\psi\equiv 0$ on $\p F$, which
together with the strong maximum principle gives that $\psi\equiv 0$
in $F$. This is a contradiction to the definition of $F$.

Hence, the claim \eqref{b182} is obtained.

By virtue of \eqref{b18}, one has that\be\label{b189}\text{there
does not exist some point $X_0\in\O_\mu$, such that
$0<\psi(X_0)=\t\psi_\e(X_0)<Q$}.\ee In fact, suppose that there
exists a point $X_0=(x_0,y_0)\in\O_\mu$, such that
$0<\psi(X_0)=\t\psi_\e(X_0)<Q$. Then there are two cases to be
considered.

 {\bf Case 1.} $X_0\in\O_\mu\cap\O_\mu^\e$. In view of
the claim \eqref{b18}, we have that either $\psi\geq \t\psi_\e$ or
$\psi\leq\t\psi_\e$ in $B_r(X_0)\subset
(\O_\mu\cap\{0<\psi<Q\})\cap(\O_\mu^\e\cap\{0<\t\psi_\e<Q\})$ for
some small $r>0$. The strong maximum principle gives that
$\psi\equiv\t\psi_\e$ in $B_r(X_0)$, due to
$\psi(X_0)=\t\psi_\e(X_0)$. By virtue of \eqref{b182}, applying
strong maximum principle again, we can conclude that
$$\text{$\psi\equiv\t\psi_\e$ in $(\O_\mu\cap\{0<\psi<Q\})\cap(\O_\mu^\e\cap\{0<\t\psi_\e<Q\})$},$$ which leads to a
contradiction to the fact that $\t\psi_\e<Q=\psi$ on $N_{2,\mu}$.

{\bf Case 2.} $X_0\in\O_\mu\setminus\O_\mu^\e$. Then there are three
subcases.

{\bf Subcase 2.1.}
$X_0\in(\O_\mu\setminus\O_\mu^\e)\cap\{y>H_{1,\mu}\}$. The extension
of $\t\psi_\e$ implies that
$0<\psi(X_0)=\t\psi_\e(X_0)=\f{(y_0-H_{1,\mu})Q}{H_{2,\mu}-H_{1,\mu}}<Q,$
which contradicts to the lower bound in \eqref{b81}.

{\bf Subcase 2.2.} $X_0\in(\O_\mu\setminus\O_\mu^\e)\cap\{h_1\leq
y\leq H_{1,\mu}\}$. It is easy to see that $\t\psi_\e(X_0)=0$, which
contradicts to our assumption $0<\t\psi_\e(X_0)<Q$.

{\bf Subcase 2.3.} $X_0\in(\O_\mu\setminus\O_\mu^\e)\cap\{y<h_1\}$.
Noticing the extension of $\t\psi_\e$, one has that
$0<\psi(X_0)=\t\psi_\e(X_0)=-\sqrt{2\ld-2gh_1}y_0+Q_1<Q.$ The
continuity of $\psi$ gives that there exists a small $r>0$, such
that $0<\psi<Q$ in $B_r(X_0)\subset\O_\mu$ and $y_0+r<h_1$. Due to
the lower bound $\psi\geq-\sqrt{2\ld-2gh_1}y+Q_1$ in \eqref{b80},
the strong maximum principle implies that
$\psi(X_0)\equiv-\sqrt{2\ld-2gh_1}y_0+Q_1$ in $B_r(X_0)$. In view of
\eqref{b182}, by using the strong maximum principle again, one gets
that $\psi(X_0)\equiv-\sqrt{2\ld-2gh_1}y_0+Q_1$ in
$\O_\mu\cap\{y<h_1\}$, which leads to a contradiction to the
boundary value of $\psi$ on $\sigma_{2,\mu}$.

Since $\t\psi_\e<Q=\psi$ on $N_{2,\mu}$, it follows from
\eqref{b182} and \eqref{b189} that
 \be\label{b19} \
\psi(x,y)\geq\t\psi(x-\e,y)\ \quad\quad\text{in $\O_\mu$}.\ee

Similarly, one can obtain that \be\label{b20} \
\psi(x,y)\leq\t\psi(x+\e,y)\ \quad\quad\text{in $\O_\mu$}.\ee Taking
$\e\rightarrow 0$ in \eqref{b19} and \eqref{b20} implies that
$\psi=\t\psi$.

In particular, taking $\psi=\t\psi$ in \eqref{b19}, we conclude that
$\psi(x,y)$ is monotone increasing with respect to $x$.

\end{proof}

\subsection{The free boundaries of the minimizer}

Thanks to the monotonicity of $\psi_{\ld,Q_1,\mu}$ with respect to
$x$ in Proposition \ref{lb4}, there exist two functions
$k_{1,\ld,Q_1,\mu}(y)$ and $k_{2,\ld,Q_1,\mu}(y)$ such that
\be\label{b400}\ba{rl}&D_\mu\cap\{0<\psi_{\ld,Q_1,\mu}<Q\}\\
=&\{(x,y)\in D_\mu\mid k_{1,\ld,Q_1,\mu}(y)<x<k_{2,\ld,Q_1,\mu}(y)\
\ \text{for}\ \ 0<y<H\}.\ea\ee

 In order to establish the
continuity of free boundaries, we need the following non-oscillation
lemma, which implies that the free boundary cannot oscillate near
any free boundary point.

\begin{lemma}\label{lb5}(non-oscillation lemma) Suppose that there exist some $\alpha_1,\alpha_2$ with $\alpha_1<\alpha_2$ and a domain
 $E\subset D_\mu\cap\{0<\psi_{\ld,Q_1,\mu}<Q\}$, such that

(1) $E$ is bounded by two disjoint arcs $\gamma_1, \gamma_2$
($\gamma_1,\gamma_2\subset\Gamma_{2,\ld,Q_1,\mu})$, the lines
$\{x=\alpha_1\}$ and $\{x=\alpha_2\}$ (see Figure \ref{f6}). Denote
the endpoints of $\gamma_i$ as $(\alpha_1,\beta_i)$ and
$(\alpha_2,\eta_i)$ for $i=1,2$.

(2) dist$(E,\overline{A_1A_2})>c_0$ for some $c_0>0$.

There exists a constant $C>0$, depending only on $\ld, H, c_0$ and
$Q$, such that
$$|\alpha_2-\alpha_1|\leq C\max\{|\beta_1-\beta_2|,|\eta_1-\eta_2|\}.$$
 A similar result holds for $\gamma_1,\gamma_2\subset
\Gamma_{1,\ld,Q_1,\mu}$.

\end{lemma}
\begin{figure}[!h]
\includegraphics[width=100mm]{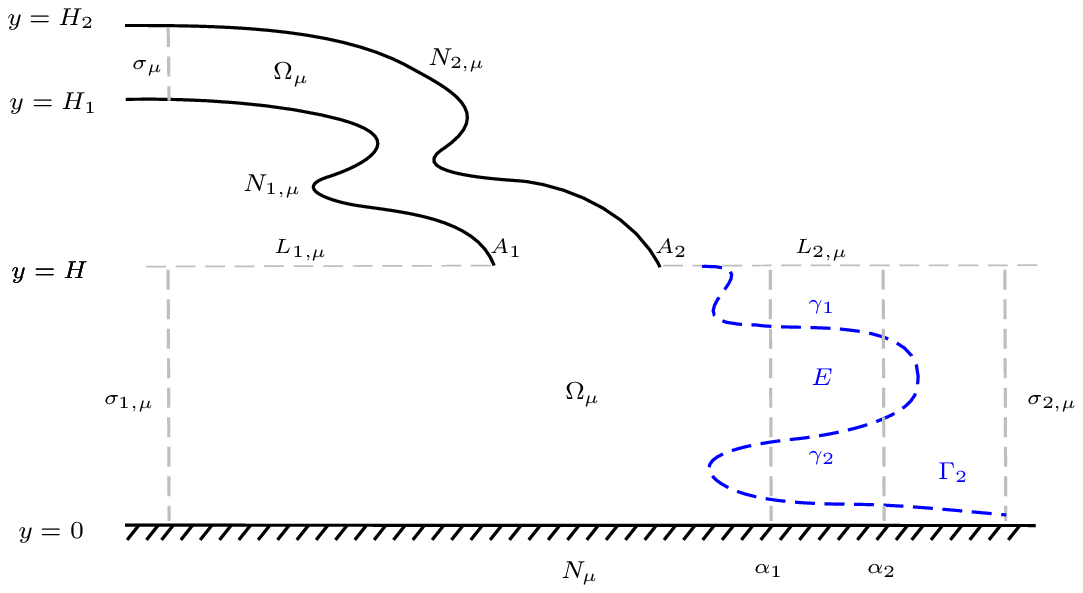}
\caption{The domain $E$}\label{f6}
\end{figure}

\begin{proof}Denote $h=\max\left\{|\beta_1-\beta_2|,|\eta_1-\eta_2|\right\}$ and $\psi=\psi_{\ld,Q_1,\mu}$. Since $\psi$ is harmonic in $E$,
it is easy to check that $\int_{\p E}\f{\p\psi}{\p\nu}dS=0$, which
yields that
$$\int_{\gamma_1\cup\gamma_2}\sqrt{2\ld-2gy} dS=\int_{\gamma_1\cup\gamma_2}\f{\p\psi}{\p\nu}dS=-\int_{\p
E\cap\{x=\alpha_2\}}\f{\p \psi} {\p x}dy+\int_{\p
E\cap\{x=\alpha_1\}}\f{\p \psi} {\p x}dy,$$ where $\gamma_1$ and
$\gamma_2$ are two arcs of the right free boundary.

By virtue of the Lipschitz continuity of $\psi$, one has
\be\label{b21}\ba{rl}-\int_{\p E\cap\{x=\alpha_2\}}\f{\p \psi} {\p
x}dy+\int_{\p E\cap\{x=\alpha_1\}}\f{\p \psi} {\p x}dy \leq
Ch.\ea\ee On the other hand, one gets
$$\int_{\gamma_1\cup\gamma_2}\sqrt{2\ld-2gy} dS\geq 2\sqrt{2\ld-2gH}(\alpha_2-\alpha_1),$$
which together with \eqref{b21} gives that
$$\alpha_2-\alpha_1\leq Ch.$$

\end{proof}
\begin{remark}\label{rl4}
The non-oscillation Lemma \ref{lb5} remains true if one of the arcs,
$\gamma_1$ for example, is a line segment (see Figure \ref{f8}), and
\be\label{b22}\f{\p\psi_{\ld,Q_1,\mu}}{\p\nu}\geq 0 \ \ \text{on}\ \
\gamma_1.\ee

\end{remark}

\begin{figure}[!h]
\includegraphics[width=100mm]{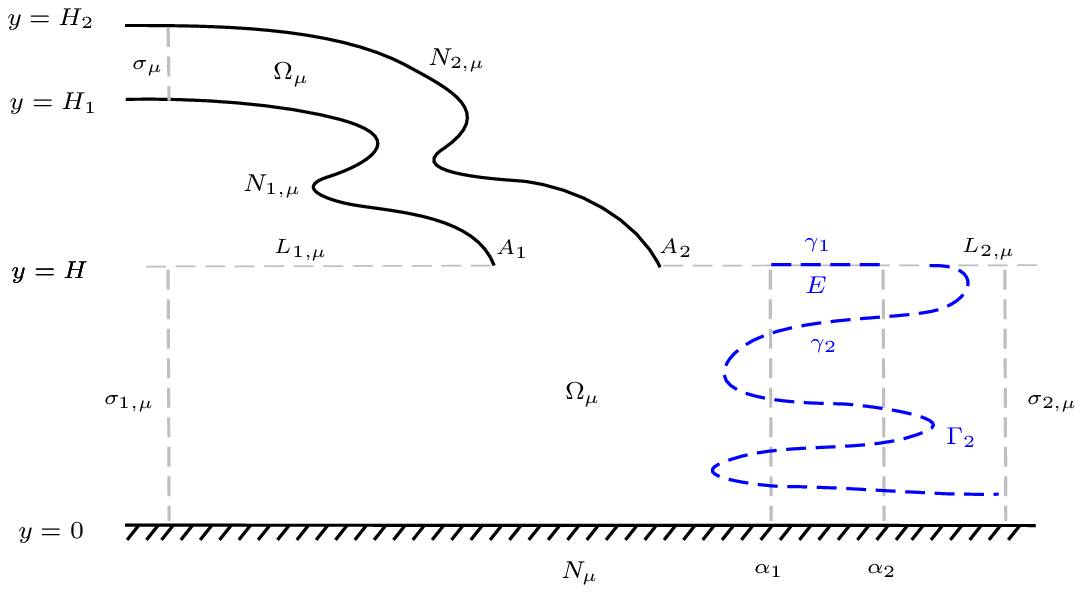}
\caption{The domain $E$}\label{f8}
\end{figure}

 With the aid of the non-oscillation Lemma \ref{lb5}, we can show the continuity of $k_{1,\ld,Q_1,\mu}(y)$ and $k_{2,\ld,Q_1,\mu}(y)$.

\begin{lemma}\label{lb6}$k_{1,\ld,Q_1,\mu}(y)$ is continuous in $(h_1,H)$ and $k_{2,\ld,Q_1,\mu}(y)$ is continuous in $(h_2,H)$.
Furthermore, $k_{i,\ld,Q_1,\mu}(H)=\lim_{y\rightarrow
H^-}k_{i,\ld,Q_1,\mu}(y)$ exists for $i=1,2$.
\end{lemma}

\begin{proof} Denote $\psi=\psi_{\ld,Q_1,\mu}$ and
$k_i(y)=k_{i,\ld,Q_1,\mu}(y)$ ($i=1,2$) for simplicity.
 We consider only  the
continuity of $k_2(y)$ in the following, and the continuity of
$k_1(y)$ can be obtained by similar arguments.

As mentioned in Remark \ref{rb2}, the free boundary
$\Gamma_{2,\ld,Q_1,\mu}$ lies above $y=h_2$. First, we will show
that the one side limit of $k_2(y)$ exists as $y\rightarrow y_0^-$
or $y\rightarrow y_0^+$ for any $y_0\in(h_2,H)$.

Suppose that there exists a $y_0\in(h_2,H)$, such that there are two
limits $\alpha_1$ and $\alpha_2$ with $\alpha_1<\alpha_2$ as
$y\rightarrow y^-_0$, that is, the free boundary
$\Gamma_{2,\ld,Q_1,\mu}$ oscillates as $y\rightarrow y_0^-$.

The monotonicity $\psi(x,y)$ with respect to $x$ implies that there
exist two sequences $\{y_n\}_{n=1}^\infty$ and $\{\t
y_n\}_{n=1}^\infty$ with $y_n<\t y_n<y_{n+1}$ for $n\geq1$, such
that $y_n\rightarrow y^-_0$, $\t y_n\rightarrow y^-_0$ and
\be\label{b23}\psi(x,y_n)=Q\quad\text{and}\quad\psi(x,\t y_n)<Q, \ee
for $\delta_1<x<\delta_2$, where
$\delta_1=\f{3\alpha_1+\alpha_2}{4}$ and
$\delta_2=\f{\alpha_1+3\alpha_2}{4}$.

Let $E_n\subset \O_{\mu}\cap\{\psi<Q\}$ be a domain, bounded by the
arcs $x=\delta_1,x=\delta_2, y=\t\gamma_n(x)\ \text{and}\
y=\gamma_n(x)$, where $(x,\gamma_n(x))$ and $(x,\t\gamma_n(x))$ are
free boundary points of $\Gamma_{2,\ld,Q_1,\mu}$ with
$\gamma_n(x)<\t\gamma_n(x)$. The existence of the domain $E_n$
follows from \eqref{b23}. It is easy to check that
$$\e_n=\sup_{\delta_1<x<\delta_2} \{\t\gamma_n(x)-\gamma_n(x)\}\rightarrow 0 \ \ \text{as}\ \ n\rightarrow\infty.$$
 The
Lipschitz continuity of $\psi$ gives that $$\psi>0\ \ \text{in $E_n$
for sufficiently large $n$}.$$

Therefore, by the non-oscillation Lemma \ref{lb5} in $E_n$ for
sufficiently large $n$, one has
$$0<\f{\alpha_2-\alpha_1}2=\delta_2-\delta_1\leq C
 \e_n\ \ \ \text{for
sufficiently large $n$},$$
 which leads to a contradiction.

 Thus, $\lim_{y\rightarrow
 y_0^-}k_2(y)$ exists for any $y_0\in(h_2,H)$. Similar arguments yield the existence of $\lim_{y\rightarrow
H^-}k_2(y)$, denoted by $k_2(H)$, and the existence of
$\lim_{y\rightarrow y_0^+}k_2(y)$ for any $y_0\in(h_2,H)$.

Finally, we will show that $k_2(y)$ is a continuous function in
$(h_2,H)$. Denote
$$k_2(y_0+0)=\lim_{y\rightarrow y_0^+}k_2(y)\ \ \text{and}\
k_2(y_0-0)=\lim_{y\rightarrow y_0^-}k_2(y),$$ and it suffices to
show that
$$k_2(y_0+0)=k_2(y_0-0)=k_2(y_0)\ \ \text{for any}\ \ y_0\in(h_2,H).$$

Suppose that there exists a $y_0\in(h_2,H)$, such that
$k_2(y_0-0)\neq k_2(y_0)$, and without loss of generality we assume
$k_2(y_0-0)>k_2(y_0)$. By virtue of the monotonicity of $\psi$ with
respect to $x$, one has that $\gamma_0=\left\{(x,y_0)\mid
x_1<x<x_2\right\}$ is a part of the free boundary
$\Gamma_{2,\ld,Q_1,\mu}$, where
$$x_1=\f{k_2(y_0-0)+3k_2(y_0)}{4}\quad\quad\text{ and}\quad\quad
x_2=\f{3k_2(y_0-0)+k_2(y_0)}{4}.$$ It follows from Proposition
\ref{lb1} that
$$\f{\p\psi(x,y_0-0)}{\p y}=\sqrt{2\ld-2gy_0} \ \ \text{and}\ \ \psi=Q\ \ \text{on}\ \
\gamma_0.$$

The continuity of $\psi$ implies that there exists a small $\e>0$
such that
$$0<\psi<Q\ \ \text{in $\mathcal{E}_{\e}$},$$ where
$\mathcal{E}_{\e}=\left\{(x,y)\mid x_1<x<x_2,y_0-\e<y<y_0\right\}$.

It follows from the Cauchy-Kovalevskaya theorem and the unique
continuation that $$\psi(x,y)=\sqrt{2\ld-2gy_0} (y-y_0)+Q \ \
\text{in} \ \t {\mathcal{E}}_\e.$$
 where $\t {\mathcal{E}}_\e=\{-\infty<x<+\infty, y_0-\e<y<y_0\}\cap\O_{\mu}$, which leads to a contradiction to the boundary condition of $\psi$.

\end{proof}

With the aid of Lemma \ref{lb6}, the free boundary can be denoted as
$$\text{$\Gamma_{i,\ld,Q_1,\mu}=\{(x,y)\in D_{\mu}\mid x=k_{i,\ld,Q_1,\mu}(y)\ \ \text{for}\ \ h_i<y<H\}$, $i=1,2$.}$$

Up to now, one can conclude that for any $Q_1\in[0,Q]$ and
$\ld\geq\f{\max\{Q_1^2,(Q-Q_1)^2\}}{2H^2}+gH$, there exists a unique
minimizer $\psi_{\ld,Q_1,\mu}$ to the variational problem
$(P_{\ld,Q_1,\mu})$ with the smooth free boundaries
$\Gamma_{1,\ld,Q_1,\mu}$ and $\Gamma_{2,\ld,Q_1,\mu}$.

\subsection{Almost continuous fit conditions of the free boundaries}

Next, we will verify the continuous fit conditions between the rigid
boundaries and the free boundaries. Namely, there exists an
appropriate pair $(\bar\ld_\mu,\bar Q_{1,\mu})$ to guarantee the
continuous fit conditions $$k_{1,\bar\ld_\mu,\bar
Q_{1,\mu},\mu}(H)=-1\quad \text{and}\quad k_{2,\bar\ld_\mu,\bar
Q_{1,\mu},\mu}(H)=1\quad \text{for any}\quad \mu>1.$$ In this
subsection, we only give almost continuous fit conditions  (in
Proposition \ref{lb11}) for $\psi_{\ld,Q_1,\mu}$ in the truncated
domain $\O_\mu$, and the continuous fit conditions for the minimizer
in the whole fluid field will be verified in the next section.

This is one of key points and the essential difference to the
asymmetric jets with gravity in \cite{ACF2}. In \cite{ACF2}, Alt,
Caffarelli and Friedman showed that there does not exist a jet under
gravity from an asymmetric nozzle with infinite height, such that
the continuous fit conditions hold. In other words, in general, the
asymmetric free boundaries can not connect the endpoints of the
nozzle at the same time. Here, a new observation is that assuming
the asymmetric nozzle is of finite height, we can obtain the
existence result on jets under gravity with continuous fit
conditions.

The desired result will be shown by the following facts.

{\bf Fact 1.} The minimizer $\psi_{\ld,Q_1,\mu}$ and the free
boundaries $x=k_{i,\ld,Q_1,\mu}(y)$ depend continuously on $\ld$ and
$Q_1$.

{\bf Fact 2.} For any $Q_1,Q_1'\in [0,Q]$ with $Q_1>Q_1'$, it holds
that
$$\psi_{\ld,Q_1,\mu}(x,y)\geq\psi_{\ld,Q_1',\mu}(x,y)\ \ \text{for any}\ \ (x,y)\in\O_\mu.$$
This fact indeed implies that the free boundary
$x=k_{i,\ld,Q_1,\mu}(y)$ is decreasing with respect to $Q_1$,
$i=1,2$.

{\bf Fact 3.} For $Q_1=\f{Q}{2}$, there exists some
$\ld>\f{Q^2}{8H^2}+gH$,
 such that
\be\label{b240}\text{$k_{1,\ld,Q_1,\mu}(H)<-1$ and
$k_{2,\ld,Q_1,\mu}(H)>1$.}\ee

This fact implies that the following set is non-empty:
$$\Sigma_{\mu}=\{\ld\mid\ \text{there exists a $Q_1\in(0,Q)$, such that \eqref{b240} holds}\}.$$

{\bf Fact 4.} The set $\Sigma_\mu$ is uniformly bounded for any
$\mu>1$.

{\bf Fact 5.} Define $\bar\ld_\mu=\sup_{\ld\in\Sigma_\mu}\ld$. There
exists a $\bar Q_{1,\mu}\in[0,Q]$, such that the free boundaries
$\Gamma_{1,\ld,Q_1,\mu}$ and $\Gamma_{2,\ld,Q_1,\mu}$ satisfy the
almost continuous fit conditions (in Proposition \ref{lb11}).

First, we show that the minimizer and the free boundaries depend
continuously on $\ld$ and $Q_1$, and complete the proof of Fact 1.

\begin{lemma}\label{lb7} For any sequences $\{\ld_n\}$ and $\{Q_{1,n}\}$ with
$$\text{$Q_{1,n}\in[0,Q]$ and $\ld_n\geq\f{\max\{Q_{1,n}^2,(Q-Q_{1,n})^2\}}{2H^2}+gH$},$$ if $\ld_n\rightarrow \ld$ and $Q_{1,n}\rightarrow Q_1$, then
 $$\psi_{\ld_n,Q_{1,n},\mu}\rightarrow\psi_{\ld,Q_1,\mu}\ \ \text{uniformly in any compact subsect of
 $\O_\mu$},$$ and
 $$k_{1,\ld_n,Q_{1,n},\mu}(H)\rightarrow k_{1,\ld,Q_1,\mu}(H)\ \ \text{and}\ \ k_{2,\ld_n,Q_{1,n},\mu}(H)\rightarrow k_{2,\ld,Q_1,\mu}(H).$$
%$$k_{1,\ld_n,Q_{1,n},\mu}(y)\rightarrow k_{1,\ld,Q_1,\mu}(y)\ \ \text{for
%any}\ \ y\in (h_1,H],$$ and
%$$k_{2,\ld_n,Q_{1,n},\mu}(y)\rightarrow k_{2,\ld,Q_1,\mu}(y)\ \ \text{for
%any}\ \ y\in (h_2,H].$$
\end{lemma}

\begin{proof}  Denote $\psi_n=\psi_{\ld_n,Q_{1,n},\mu}$ and
$k_{i,n}(y)=k_{i,\ld_n,Q_{1,n},\mu}(y)$ ($i=1,2$) for simplicity.
Thanks to Lemma \ref{lb60}, one has
$$|\g\psi_n(X)|\leq C(\sqrt{\ld_n}+Q)\leq C(\sqrt{2\ld}+Q)\ \ \text{in any compact subset $K$ in $\O_\mu$},$$
where the constant $C$ depends only on $K$ and $\O_\mu$, provided
that $n$ is sufficiently large. Then there exists a subsequence
still labeled as $\{\psi_n\}$ such that
\be\label{b24}\psi_n\rightharpoonup \o\ \ \text{weakly in}\ \
H_{loc}^1(\R^2)\ \text{and}\ \psi_n\rightarrow \o \ \text{in} \
C^\alpha(\bar E),\ee for any compact subset $E$ of $\O_\mu$ and
$0<\alpha<1$.

{\bf Step 1.} $D_\mu\cap\p \{0<\psi_n<Q\}\rightarrow D_\mu\cap\p
\{0<\o<Q\}$ in the Hausdorff distance, where the Hausdorff distance
$d(E,F)$ between two sets $E$ and $F$ is defined as follows,
$$d(E,F)=\inf\left\{\e>0\mid E\subset\bigcup_{X\in F} B_\e(X) \ \ \text{and}\ \ F\subset\bigcup_{X\in E} B_\e(X)\right\}.$$

For any $X_0\in D_{\mu}$, if $X_0\notin D_\mu\cap\p\{0<\o<Q\}$, then
there exists a small $r$ with $0<r<\f12\text{dist}(X_0,\p\O_\mu)$,
such that $B_r(X_0)\Subset D_{\mu}$ and
$B_r(X_0)\cap\p\{0<\o<Q\}=\varnothing$. We next claim that
\be\label{b25}\text{$B_{\f
r{20}}(X_0)\cap\p\{0<\psi_n<Q\}=\varnothing$ for sufficiently large
$n$.}\ee In fact, it follows from \eqref{b24} that the claim
\eqref{b25} is true for $0<\o<Q$ in $B_r(X_0)$. On another hand, if
$\o\equiv Q$ in $B_r(X_0)$, for any small $\e>0$, the Lipschitz
continuity of $\psi_n$ in Lemma \ref{lb60} gives that there exists a
positive integer $N(\e)$, such that
$$\text{ $\psi_n>0$ and $|Q-\psi_n(X)|<\e$ in $B_{\f {3r}4}(X_0)$
for any $n>N(\e)$}.$$ Then we have
$$\f2r\fint_{\partial B_{\f r2}(X_0)}(Q-\psi_n)dS<\f{2\e}r\leq c^*\sqrt{2\ld_n-2gH} \ \text{for sufficiently large
$n$},$$ where the constant $c^*$ is same as in Lemma \ref{lf3}.
Therefore, Lemma \ref{lf3} implies that $\psi_n\equiv Q$ in
$B_{\f{r}{16}}(X_0)$ for sufficiently large $n$, and thus the claim
\eqref{b25} is true. Similarly, one can obtain \eqref{b25}, provided
that $\o\equiv0$ in $B_r(X_0)$.

Reversely, for any $X_0\in D_{\mu}$, if $X_0\notin
D_\mu\cap\p\{0<\psi_n<Q\}$ for sufficiently large $n$, then there
exists a small $r$ with $0<r<\f12\text{dist}(X_0,\p\O_\mu)$, such
that $B_r(X_0)\cap\p\{0<\psi_n<Q\}=\varnothing$. Next, we claim that
\be\label{b26}\text{$B_{\f
r4}(X_0)\cap\p\{0<\o<Q\}=\varnothing$}.\ee If $0<\psi_n<Q$ in
$B_r(X_0)$ for a sequence $\{\psi_n\}$, then
$$\Delta\psi_n= 0 \ \ \text{in}\ \ B_r(X_0),$$ which implies that $$\Delta\o= 0 \ \text{and} \ \ 0\leq\o\leq Q\ \text{in} \ B_{\f r2}(X_0).$$
The strong maximum principle yields that
$$\text{either}\ \ \o\equiv Q\ \ \  \text{or}\ \o\equiv 0\ \  \ \text{or}\ \ \ 0<\o<Q\ \text{ in}\ B_{\f r2}(X_0),$$
which gives the claim \eqref{b26}.

If $\psi_n\equiv Q$ or $\psi_n=0$ in $B_r(X_0)$ for a sequence
$\{\psi_n\}$, it is easy to see that the claim \eqref{b26} is true.

 Hence, we have shown the convergence of the free boundary in the Hausdorff distance.

{\bf Step 2.} $\chi_{\{0<\psi_n<Q\}}\rightarrow \chi_{\{0<\o<Q\}}$
in $L^1(D_{\mu})$.

For any $X_n\in K\Subset D_\mu$ with $X_n\in\p\{0<\psi_n<Q\}$, it
follows from the results in Step 1 that there exists a subsequence
$\{X_n\}$, such that $X_n\rightarrow X_0\in D_\mu\cap\p\{0<\o<Q\}$.

Since $\text{dist}(X_n,A_i)\geq\text{dist}(K,\p\O_\mu)$ for $i=1,2$,
the Lipschitz continuity of $\psi_n$ in Lemma \ref{lb60} implies
that there exists a small $r_0$ with $0<r_0<\f12\text{dist}(K,\p
D_\mu)$, such that \be\label{b27}\text{$\psi_n>0$ in $B_r(X_n)$ for
$X_n\in D_\mu\cap\p\{\psi_n<Q\}$},\ee and
\be\label{b28}\text{$\psi_n<Q$ in $B_r(X_n)$ for $X_n\in
D_\mu\cap\p\{\psi_n>0\}$},\ee for any $r\in(0,r_0)$. With the aid of
\eqref{b27} and \eqref{b28}, it follows from Lemma \ref{lf1} and
Lemma \ref{lf3} for $\psi_n$ that
\be\label{b29}c^*\sqrt{2\ld_n-2gH}\leq\f1r\fint_{\partial
B_r(X_n)}(Q-\psi_n)dS\leq C^*\sqrt{2\ld_n}\ \ \text{for}\ \ X_n\in
D_\mu\cap\p\{\psi_n<Q\},\ee and
\be\label{b30}c^*\sqrt{2\ld_n-2gH}\leq\f1r\fint_{\partial
B_r(X_n)}\psi_n dS\leq C^*\sqrt{2\ld_n}\ \ \text{for}\ \ X_n\in
D_\mu\cap\p\{\psi_n>0\}\ee for any $r\in(0,r_0)$, where the constant
$C^*$ and $c^*$ are uniform constants as in Lemma \ref{lf1} and
Lemma \ref{lf3}, respectively.

Then taking $n\rightarrow+\infty$ in \eqref{b29} and in \eqref{b30},
one has \be\label{b31}c^*\sqrt{2\ld-2gH}\leq\f1r\fint_{\partial
B_r(X_0)}(Q-\o)dS\leq C^*\sqrt{2\ld}\ \ \text{for}\ \ X_0\in
D_\mu\cap\p\{\o<Q\},\ee and
\be\label{b32}c^*\sqrt{2\ld-2gH}\leq\f1r\fint_{\partial B_r(X_0)}\o
dS\leq C^*\sqrt{2\ld}\ \ \text{for}\ \ X_0\in
D_\mu\cap\p\{\o>0\},\ee  for any $r\in(0,r_0)$. Those together with
Theorem 4.5 in \cite{AC1} imply that
$$\mathcal{H}^1(D_\mu\cap\p\{0<\o<Q\})<+\infty,$$ where
$\mathcal{H}^1$ is the one-dimensional Hausdorff measure on
$\mathbb{R}^2$. Consequently,
\be\label{b33}\mathcal{L}^2(D_\mu\cap\p\{0<\o<Q\})=0,\ee where
$\mathcal{L}^2$ is the two-dimensional Lebesgue measure on
$\mathbb{R}^2$.

Let $O^1_{\e_n}$ be an $\e_n$-neighborhood of $D_\mu\cap\p\{\o>0\}$
and $O^2_{\e_n}$ be an $\e_n$-neighborhood of $D_\mu\cap\p\{\o<Q\}$,
such that
$$D_\mu\cap\p\{\psi_n>0\}\subset O^1_{\e_n}\ \ \text{and}\ \ D_\mu\cap\p\{\psi_n<Q\}\subset O^2_{\e_n},$$ and \be\label{b34}\mathcal{L}^2(D_{\mu}\cap
O_{\e_n})\rightarrow 0\ \quad \text{as}\ \e_n\rightarrow 0,\ \ \
O_{\e_n}=O^1_{\e_n}\cup O^2_{\e_n},\ee due to $D_\mu\cap\p
\{0<\psi_n<Q\}\rightarrow D_\mu\cap\p \{0<\o<Q\}$ in the Hausdorff
distance in Step 1.

Hence, one has
\be\label{b35}\int_{D_{\mu}}\left|\chi_{\{0<\psi_n<Q\}}-\chi_{\{0<\o<Q\}}\right|dxdy\leq\int_{D_{\mu}\cap
O_{\e_n}}1dxdy=\mathcal{L}^2(D_{\mu}\cap O_{\e_n}), \ee  for
sufficiently large $n$, which together with \eqref{b34} gives that
$$\text{$\chi_{\{0<\psi_n<Q\}}\rightarrow \chi_{\{0<\o<Q\}}$ in
$L^1(D_{\mu})$.}$$

 {\bf Step 3.} $\nabla\psi_n\rightarrow \nabla\o$ a.e. in
 $\O_{\mu}$.

 Let $E$ be any compact subset of $\O_{\mu}\cap\{0<\o<Q\}$. It follows from the result in Step 1 that the minimizer
 $\psi_n$ solves Laplace equation in $E$ for sufficiently large $n$. Thanks to the standard elliptic estimates for $\psi_n$, one has
\be\label{b36}\nabla\psi_n\rightarrow\nabla\o\ \ \text{uniformly in
$E$}.\ee Next, we will show that
\be\label{b37}\nabla\psi_n\rightarrow\nabla\o\ \ \text{a.e. in
$\O_\mu\cap\{\o=Q\}$}.\ee

Since $\O_\mu\cap\{\o=Q\}$ is $\mathcal{L}^2$-measurable and
$\mathcal{L}^2(\O_\mu\cap\p\{\o<Q\})=0$, it follows from Corollary 3
in \cite{EV} that
$$\lim_{r\rightarrow 0}\f{\mathcal{L}^2(B_r(X)\cap\O_\mu\cap\{\o=Q\})}{\mathcal{L}^2(B_r(X))}=1\ \ \text{for $\mathcal{L}^2$ a.e. $X\in\O_\mu\cap\{\o=Q\}$}.$$
Denote $$\mathcal{S}=\left\{X\in\O_\mu\cap\{\o=Q\}\mid
\lim_{r\rightarrow
0}\f{\mathcal{L}^2(B_r(X)\cap\O_\mu\cap\{\o=Q\})}{\mathcal{L}^2(B_r(X))}=1\right\}.$$
Then we claim that
\be\label{b38}\text{$\f{Q-\o(X_0+X)}{|X|}\rightarrow0$ as
$|X|\rightarrow0$ for any $X_0\in\mathcal{S}$}.\ee In fact, suppose
that there exists an $X_0\in \mathcal{S}$, such that $Q-\o(Y)>kr$
for some $Y\in B_r(X_0)$ with $r\rightarrow 0$ and $k>0$. With the
aid of \eqref{b31} and \eqref{b32}, it follows from Theorem 4.3 and
Remark 4.4 in \cite{AC1} that $\o\in C^{0,1}(\O_\mu)$, which implies
that
$$Q-\o>\f k2 r\ \ \text{in $B_{\e kr}(Y)\subset B_{2r}(X_0)$ for some small $\e>0$},\ \ \ 0<r<\f14\text{dist}(X_0,\p\O_\mu).$$
This gives that $\O_\mu\cap\{\o<Q\}$ has positive density at $X_0$,
namely, $$\lim_{r\rightarrow
0}\f{\mathcal{L}^2(B_{2r}(X_0)\cap\O_\mu\cap\{\o<Q\})}{\mathcal{L}^2(B_{2r}(X_0))}>\f{\e^2
k^2}{4},$$ which contradicts to the fact $X_0\in\mathcal{S}$.

With the aid of \eqref{b24} and \eqref{b38}, for any $\e>0$, one has
$$\f{Q-\psi_n}{r}<\e\ \ \text{and}\ \ \psi_n>0 \ \ \text{in $B_r(X_0)$ for small $r$ with $0<r<\f14\text{dist}(X_0,\p\O_\mu)$},$$
provided that $n$ is sufficiently large, that is $n>N(\e,r)$. It
follows from the Lemma \ref{lf3} in the appendix that $\psi_n\equiv
Q$ in $B_{\f r8}(X_0)$, which implies that $\o\equiv Q$ in $B_{\f
r{16}}(X_0)$. Thus, $\mathcal{S}$ is open, and furthermore,
$$\text{$\psi_n\equiv\o$ in any compact subset of $\mathcal{S}$ for sufficiently
large $n$.}$$ This completes the proof of \eqref{b37}.

Similarly, one can show that
$$\nabla\psi_n\rightarrow\nabla\o\ \ \text{a.e.
$\O_\mu\cap\{\o=0\}$}.$$

Since $\mathcal{L}^2(\O_\mu\cap\p\{0<\o<Q\})=0$, it holds that
$\nabla\psi_n\rightarrow \nabla\o$ a.e. in $\O_{\mu}$.

{\bf Step 4.} $\o=\psi_{\ld,Q_1,\mu}$. Denote
$\psi=\psi_{\ld,Q_1,\mu}$ in the proof for the notational
simplicity.

First, we will check that $\o$ is a minimizer to the truncated
variational problem $(P_{\ld,Q_1,\mu})$. It follows from \eqref{b24}
that $\o\in K_{\ld,Q_1,\mu}$, and
$$J_{\ld,\mu}(\psi)\leq J_{\ld,\mu}(\o).$$
It suffices to show that \be\label{b39}J_{\ld,\mu}(\psi)\geq
J_{\ld,\mu}(\o).\ee For any $\eta\in C_0^{0,1}(\O_\mu)$ with
$0\leq\eta\leq1$, set
$$\phi_n=\psi+(1-\eta)(\psi_n-\o).$$ It is easy to check that $\phi_n\in K_{\ld_n,Q_{1,n},\mu}$ and

$$\chi_{\{0<\phi_n<Q\}\cap D_\mu}\leq \chi_{\{0<\psi<Q\}\cap D_\mu}+\chi_{\{\eta<1\}\cap D_\mu}.$$
Hence, \be\label{b40}\ba{rl}&\int_{\Omega_\mu}|\nabla
\psi_n|^2+(2\ld_n-2gy)\chi_{\{0<\psi_n<Q\}\cap D_\mu}
dxdy\\
\leq&\int_{\Omega_\mu}|\nabla
\phi_n|^2+(2\ld_n-2gy)\chi_{\{0<\phi_n<Q\}\cap D_\mu}
dxdy\\
\leq&\int_{\Omega_\mu}|\nabla
\phi_n|^2+(2\ld_n-2gy)\chi_{\{0<\psi<Q\}\cap D_\mu}
dxdy+2\ld_n\int_{D_\mu}\chi_{\{\eta<1\}} dxdy.\ea\ee

Since $\mathcal{L}^2(\O_\mu\cap\p\{0<\psi<Q\})=0$, by virtue of the
results in Step 2 and Step 3 and taking $n\rightarrow+\infty$ in
\eqref{b40}, one has \be\label{b41}\ba{rl}&\int_{\Omega_\mu}|\nabla
\o|^2+(2\ld-2gy)\chi_{\{0<\o<Q\}\cap D_\mu}
dxdy\\
\leq&\int_{\Omega_\mu}|\nabla
\psi|^2+(2\ld-2gy)\chi_{\{0<\psi<Q\}\cap D_\mu}
dxdy+2\ld\int_{D_\mu}\chi_{\{\eta<1\}} dxdy.\ea\ee Choosing
$\eta(X)=d_\e(X)=\min\left\{\f1\e dist(X,\mathbb{R}^2\setminus
D_\mu),1\right\}$ for $\e>0$ in \eqref{b41} and taking
$\e\rightarrow 0$ yield \eqref{b39}. Thanks to the uniqueness of the
minimizer to the truncated variational problem $(P_{\ld,Q_1,\mu})$,
one gets that $\o=\psi=\psi_{\ld,Q_1,\mu}$.

{\bf Step 5.} %Similar to the proof in Step 2, by virtue of the
%non-degeneracy Lemma \ref{lf4}, we have
%$$k_{1,n}(y)\rightarrow k_{1,\ld,Q_1,\mu}(y)\ \ \text{for
%any}\ \ y\in (h_1,H),$$ and
%$$k_{2,n}(y)\rightarrow k_{2,\ld,Q_1,\mu}(y)\ \ \text{for
%any}\ \ y\in (h_2,H).$$
In this step, we will show the convergence of the free boundaries at
the initial points, namely,  $k_{i,n}(H)\rightarrow
k_{i,\ld,Q_1,\mu}(H)$ as $n\rightarrow+\infty$ ($i=1,2$).

Suppose not, without loss of generality, we assume that there exist
two subsequences still labeled by $\{\ld_n\}$ and $\{Q_{1,n}\}$ such
that $$\text{$k_{2,n}(H)\rightarrow k_{2,\ld,Q_1,\mu}(H)+\delta$
with $\delta\neq 0.$}$$ Denote $\psi=\psi_{\ld,Q_1,\mu}$ and
$k_2(H)=k_{2,\ld,Q_1,\mu}(H)$ for simplicity. There are three cases
to be considered.

{\bf Case 1.} $\delta<0$. We consider two subcases in the following.

{\bf Subcase 1.1.} $k_{2}(H)+\delta\geq 1$. Denote $$I_0=
\left\{(x,H)\mid k_{2}(H)+\f{3\delta}4<x
 <k_{2}(H)+\f\delta4\right\}$$ and $$U_{\delta,\e}=\left\{(x,y)\mid k_{2}(H)+\f{3\delta}4<x
 <k_{2}(H)+\f\delta4,H-\e<y<H+\e\right\}$$ for small $\e>0$.

 It is easy to check that $I_0=U_{\delta,\e}\cap\p\{\psi<Q\}$.  We
extend $\psi_n=Q$ and $\psi=Q$ in $U_{\delta,\e}\setminus \O_\mu$.
Since $\text{dist}(A_i, U_{\delta,\e})\geq-\f{\delta}4$ ($i=1,2$),
Lemma \ref{lb60} gives that $\psi_n\in C^{0,1}(\bar U_{\delta,\e})$.
Obviously, $0<\psi_n\leq Q$ in $U_{\delta,\e}$ for small $\e>0$.

Next, we show that
$$\Gamma_{2,\ld_n,Q_{1,n},\mu}\cap\{y=H\}=\varnothing\ \ \text{and}\ \ U_{\delta,\e}\cap\p\{\psi_n<Q\}\neq\varnothing\ \ \text{for sufficiently large $n$}.$$
Noting that $k_{2,n}(H)\rightarrow k_{2,\ld,Q_1,\mu}(H)+\delta\geq
1$ with $\delta<0$, then $k_{2,n}(H)\leq
k_{2,\ld,Q_1,\mu}(H)+\f{7\delta}8$ for sufficiently large $n$. The
definition of $k_{2,n}(H)$ implies that
$B_r((k_{2,n}(H),H))\cap\Gamma_{2,\ld_n,Q_{1,n},\mu}\neq\varnothing$
for any $r>0$, which yields that
$\Gamma_{2,\ld_n,Q_{1,n},\mu}\cap\{y=H\}=\varnothing$, due to that
the free boundary $\Gamma_{2,\ld_n,Q_{1,n},\mu}$ is a $y$-graph.
Recalling that $\psi_n\rightarrow\psi$ in $\bar U_{\delta,\e}$ and
$\psi<Q$ in $U_{\delta,\e}\cap\{y<H\}$, it follows from
$k_{2,n}(H)\leq k_{2,\ld,Q_1,\mu}(H)+\f{7\delta}8$ that
$U_{\delta,\e}\cap\p\{\psi_n<Q\}\neq\varnothing$ for sufficiently
large $n$.

Since the right free boundary of $\psi_n$ is analytic, it holds that
$U_{\delta,\e}\cap\p\{\psi_n<Q\}$ is $C^{1,\alpha}$ ($0<\alpha<1$),
and
$$\Delta\psi_n=0  \ \text{in}~~U_{\delta,\e}\cap\{\psi_n<Q\},\ \ \f{\p\psi_n}{\p\nu}=\sqrt{2\lambda_n-2gy}\ \ \text{on}~~ U_{\delta,\e}\cap\p\{\psi_n<Q\},$$
where $\nu$ is the outer normal. Moreover, the $C^{1,\alpha}$-norm
of  $U_{\delta,\e}\cap\p\{\psi_n<Q\}$ is independent of $n$. In
fact, for $\phi_n=\f{Q-\psi_n}{\sqrt{2\ld_n}}$, it is easy to check
that
$$\phi_n=0\ \ \text{and}\ \ \sqrt{1-\f{4gH}{3\ld}}\leq|\g\phi_n|=\sqrt{1-\f{gy}{\ld_n}}\leq 1\ \ \ \text{on}\ \ U_{\delta,\e}\cap\p\{\psi_n<Q\},$$
 for
sufficiently large $n$. Applying the results in Section 8 in
\cite{AC1}, we can conclude that the free boundary
$U_{\delta,\e}\cap\p\{\psi_n<Q\}$ is analytic and  the uniform
estimate of $|\nabla \phi_n|$ gives that the $C^{1,\alpha}$-norm of
$U_{\delta,\e}\cap\p\{\psi_n<Q\}$ is independent of $n$.

Furthermore, by virtue of the results in previous steps, one has
$$\psi_n\rightarrow\psi\ \ \text{uniformly in $U_{\delta,\e}$, $\g\psi_n\rightharpoonup\g\psi$ weakly in
$L^2(U_{\delta,\e})$},$$ and
$$\text{$U_{\delta,\e}\cap\{\psi_n<Q\}\rightarrow
U_{\delta,\e}\cap\{\psi<Q\}$ in $\mathcal{L}^2$ measure,
$I_0=U_{\delta,\e}\cap\p\{\psi<Q\}$.}$$ Therefore, we can apply the
convergence of free boundaries in Lemma 6.1 in Chapter 3 in
\cite{FA1} to obtain
$$\Delta\psi=0  \ \text{in}~~U_{\delta,\e}\cap\{\psi<Q\},\ \ \psi=Q\ \text{and}\ \f{\p\psi}{\p\nu}=\sqrt{2\lambda-2gH}\ \ \text{on}~~
I_0=U_{\delta,\e}\cap\p\{\psi<Q\},$$ which implies that
$$\f{\p\psi(x,H-0)}{\p y}=\sqrt{2\ld-2gH}\ \ \text{on}\ \
I_0.$$ It follows from the Cauchy-Kovalevskaya theorem for $\psi$ in
$U_{\delta,\e}\cap\{y<H\}$ that $$\psi=\sqrt{2\ld-2gH}(y-H)+Q\ \
\text{in}\ \ \ U_{\delta,\e}\cap\{y<H\}.$$
 By using the unique continuation for the harmonic function, one has
$$\psi=\sqrt{2\ld-2gH}(y-H)+Q\ \ \text{in}\ \left\{(x,y)\mid-\infty<x<+\infty,H-\e<y<H\right\}\cap \O_{\mu},
$$
which is impossible.

{\bf Subcase 1.2.} $k_{2}(H)+\delta<1$.
 The monotonicity of $\psi(x,y)$
gives that $\psi(x,H)=Q$ for $k_{2}(H)+\delta\leq x\leq
\min\{k_{2}(H),1\}$. Denote $I_0=\left\{(x,H)\mid
k_{2}(H)+\delta+\e<x
 <k_{2}(H)+\delta+2\e\right\}$ for
 $\e=\f{\min\{1-k_2(H),0\}-\delta}{3}$. Similar to Subcase 1.1, by virtue of the convergence
 of the free boundary of $\psi_n$, one has
$$\f{\p\psi(x,H+0)}{\p y}=\sqrt{2\ld-2gH}\ \ \text{on}\ \
I_0,$$ which also leads to a contradiction by using the
Cauchy-Kovalevskaya theorem.

 {\bf Case 2.} $\delta>0$ and $k_{2}(H)<1$. Set $\e=\f{\min\{\delta,1-k_{2}(H)\}}4$. The
 convergence of the free boundary gives that
$$\f{\p\psi(x,H+0)}{\p y}=\sqrt{2\ld-2gH}\ \ \text{if}\ \
k_{2}(H)+\e<x<1-\e,\ y=H, $$  which leads to a contradiction due to
the Cauchy-Kovalevskaya theorem.

 {\bf Case 3.} $\delta>0$ and $k_{2}(H)\geq1$.  It follows from the arguments in Lemma 5.6 (i) in \cite{ACF3} that
$$\f{\p\psi_{n}(x,H-0)}{\p y}\geq\sqrt{2\lambda_n-2gH}\ \ \text{on}\ \
\left\{ k_{2}(H)+\f{\delta}4<x<k_{2}(H)+\f{3\delta}4, y=H\right\},$$
for sufficiently large $n$.

Let $E_n$ be the domain bounded by $y=H$, $x=k_{2,n}(y)$,
$x=k_{2}(H)+\f{\delta}4$ and $x=k_{2}(H)+\f{3\delta}4$. Since
$k_{2,n}(H)\rightarrow k_{2}(H)+\delta$ with $\delta>0$, one has
$k_{2,n}(H)\geq k_{2}(H)+\f{7\delta}8$ %and
%$\Gamma_{2,\ld_n,Q_{1,n},\mu}\cap
%\left\{x=\f{3\delta}{4}\right\}\neq\varnothing$
for sufficiently large $n$. Denote $y_\delta=\max\left\{y\mid
k_{2}(y)=k_{2}(H)+\f{\delta}{8}\right\}$. It follows from
$D_\mu\cap\p \{0<\psi_n<Q\}\rightarrow D_\mu\cap\p \{0<\o<Q\}$ in
the Hausdorff distance in Step 1 that
$$D_\delta\cap\p \{\psi_n<Q\}\rightarrow D_\delta\cap\p
\{\o<Q\}\ \ \text{in the Hausdorff distance},$$ where
$D_\delta=D_\mu\cap\left\{k_{2}(H)<x<k_{2}(H)+\delta\right\}$. This
implies that
$$k_{2,n}(y_\delta)\rightarrow k_{2,\ld,Q_1,\mu}(y_\delta)=k_{2}(H)+\f{\delta}{8}.$$ Thus, one has $\Gamma_{2,\ld_n,Q_{1,n},\mu}\cap
\left\{x=k_{2}(H)+\f{\delta}{4}\right\}\neq\varnothing$ for
sufficiently large $n$. Therefore, the domain $E_n$ is well-defined.

Furthermore, it follows from $\psi_n\rightarrow \psi$ and $\psi=Q$
on $L_{2,\mu}$ that
$$y_n=\max\left\{y\mid
k_{2,n}(y)=k_{2}(H)+\f{3\delta}4\right\}\rightarrow H\ \ \text{as}\
\ n\rightarrow+\infty.$$ Thanks to the non-oscillation Lemma
\ref{lb5} and Remark \ref{rl4} for $\psi_{n}$ in $E_n$, there exists
a constant $C$ independent of $n$, such that
$$0<\f\delta2\leq C(H-y_n),$$
which gives a contradiction for sufficiently large $n$.

\end{proof}

%Next, we will introduce the continuous fit conditions of the free
%boundaries $\Gamma_{1,\ld,Q_1,L}$ and $\Gamma_{2,\ld,Q_1,L}$,
%namely, \be\label{b42}k_{1,\ld,Q_1,L}(H)=-1 \ \text{and} \
%k_{2,\ld,Q_1,L}(H)=1\ \ \ \text{for some parameters $\ld$ and
%$Q_1$}. \ee

Second, the monotonicity of the minimizer $\psi_{\ld,Q_1,\mu}$ with
respect to the parameter $Q_1$ is obtained as follows.

\begin{lemma}\label{lb8}
For any $Q_1,Q_1'\in[0,Q]$ with $Q_1>Q_1'$, it holds that
$$\text{$\psi_{\ld,Q_1,\mu}(x,y)\geq\psi_{\ld,Q_1',\mu}(x,y)$ in
$\O_{\mu}$.}$$
\end{lemma}
\begin{proof}
It is easy to check that
$$\min\{\psi_{\ld,Q_1,\mu},\psi_{\ld,Q_1',\mu}\}\in K_{\ld,Q_1',\mu}\ \ \ \text{and}\ \ \ \max\{\psi_{\ld,Q_1,\mu},\psi_{\ld,Q_1',\mu}\}\in
K_{\ld,Q_1,\mu}.$$ Since
$\psi_{\ld,Q_1',\mu}=Q_1'<Q_1=\psi_{\ld,Q_1,\mu}$ on $N_\mu$, it
follows from the similar arguments in Proposition \ref{lb4} that
$$\psi_{\ld,Q_1,\mu}(x,y)\geq\psi_{\ld,Q_1',\mu}(x,y)\ \ \text{in}\ \ \O_\mu.$$

\end{proof}

Next, we will verify the Fact 3.

\begin{lemma}\label{lb9}For $Q_1=\f Q2$, there exists a $\ld>\f{Q^2}{8H^2}+gH$, such that if $\ld-\f{Q^2}{8H^2}-gH$ is small, then $$
k_{1,\ld,Q_1,\mu}(H)<-1 \ \text{and}\ k_{2,\ld,Q_1,\mu}(H)>1.$$
\end{lemma}
\begin{proof}
Suppose not, without loss of generality, we assume that there exists
a sequence $\{\ld_n\}$ with $\ld_n>\ld_0=\f{Q^2}{8H^2}+gH$, such
that $\ld_n\downarrow\ld_0$ and $k_{2,\ld_n,Q_1,\mu}(H)\leq1$.

Then Lemma \ref{lb3} implies that
$$\max\left\{-\sqrt{2\ld_n-2gh_{1,n}}y+Q_1,0\right\}\leq\psi_{\ld_n,Q_1,\mu}(x,y)\leq\min\left\{\sqrt{2\ld_n-2gh_{2,n}}
y+Q_1,Q\right\}$$ in $D_\mu$, where $h_{1,n}=h_{2,n}\in(0,H]$ is
determined uniquely by
$$\ld_n=\f{Q^2}{8h_{1,n}^2}+gh_{1,n}.$$

With the aid of Lemma \ref{lb7}, taking $n\rightarrow+\infty$ in
above inequality gives that
$$\f Q2\left(1-\f yH\right)=-\sqrt{2\ld_0-2gH}y+\f Q2\leq\psi_{\ld_0,Q_1,\mu}(x,y)\leq\sqrt{2\ld_0-2gH}y+\f
Q2=\f Q2\left(1+\f yH\right)$$ in $D_\mu$, which implies that
$0<\psi_{\ld_0,Q_1,\mu}(x,y)<Q$ in $D_\mu$, and thus both of the
free boundaries $\Gamma_{1,\ld_0,Q_1,\mu}$ and
$\Gamma_{2,\ld_0,Q_1,\mu}$ are empty in $D_\mu$. For the case
$k_{2,\ld_n,Q_1,\mu}(H)\leq1$, we claim that the free boundary
$\Gamma_{2,\ld_n,Q_1,\mu}$ is non-empty.  In fact, if
$\Gamma_{2,\ld_n,Q_1,\mu}$ is empty, namely,
$\psi_{\ld_n,Q_1,\mu}<Q$ in $D_\mu$. %We will show that
%$$k_{2,\ld_n,Q_1,\mu}(H)=\mu.$$ Suppose that $k_{2,\ld_n,Q_1,\mu}(H)<\mu$.Similar to Proposition \ref{lb1}, we
%can show that $\psi_{\ld_n,Q_1,\mu}$ is harmonic in
%$\O_\mu\cap\{\psi_{\ld_n,Q_1,\mu}>0\}$, which implies that
%$k_{2,\ld_n,Q_1,\mu}(H)\geq 1$. Since $1\leq
%k_{2,\ld_n,Q_1,\mu}(H)<\mu$,
Since $k_{2,\ld_n,Q_1,\mu}(H)\leq1$, the definition of
$k_{2,\ld_n,Q_1,\mu}(H)$ gives that
$B_r((k_{2,\ld_n,Q_1,\mu}(H),H))\cap\Gamma_{2,\ld_n,Q_1,\mu}\neq\varnothing$
for any $r>0$. Therefore, the free boundary
$\Gamma_{2,\ld_n,Q_1,\mu}$ is non-empty. This contradicts to our
assumption that the free boundary $\Gamma_{2,\ld_n,Q_1,\mu}$ is
empty.

Denote $\psi=\psi_{\ld_0,Q_1,\mu}$ and $\psi_n=\psi_{\ld_n,Q_1,\mu}$
for simplicity. Set
$$I= \left\{(x,H)\mid \f{3+\mu}4<x
 <\f{1+3\mu}4\right\}$$ and $$U_{\e}=\left\{(x,y)\mid \f{3+\mu}4<x
 <\f{1+3\mu}4,H-\e<y<H+\e\right\}$$ for small $\e>0$.

 Obviously, $I=U_{\e}\cap\p\{\psi<Q\}$ and
$\text{dist}(A_i, U_{\e})\geq\f{1}4$ ($i=1,2$). Extend $\psi_n=Q$
and $\psi=Q$ in $U_{\e}\setminus \O_\mu$. It follows from Lemma
\ref{lb60} that $\psi_n\in C^{0,1}(\bar U_{\e})$ and $0<\psi_n\leq
Q$ in $U_{\e}$ for small $\e>0$. The analyticity of
$\Gamma_{2,\ld_n,Q_1,\mu}$ implies that $U_{\e}\cap\p\{\psi_n<Q\}$
is $C^{1,\alpha}$ ($0<\alpha<1$), and
$$\Delta\psi_n=0  \ \text{in}~~U_{\e}\cap\{\psi_n<Q\},\ \ \f{\p\psi_n}{\p\nu}=\sqrt{2\lambda_n-2gy}\ \ \text{on}~~ U_{\e}\cap\p\{\psi_n<Q\},$$
where $\nu$ is the outer normal. Moreover, following the similar
arguments in the proof of Lemma \ref{lb7}, one has
$$\psi_n\rightarrow\psi\ \ \text{uniformly in $U_{\e}$, $\g\psi_n\rightharpoonup\g\psi$ weakly in
$L^2(U_{\e})$},$$ and
$$\text{$U_{\e}\cap\{\psi_n<Q\}\rightarrow
U_{\e}\cap\{\psi<Q\}$ in $\mathcal{L}^2$ measure,
$I=U_{\e}\cap\p\{\psi<Q\}$.}$$ Therefore, we can apply the
convergence of free boundaries in Lemma 6.1 in Chapter 3 in
\cite{FA1} to get
$$\Delta\psi=0  \ \text{in}~~U_{\e}\cap\{\psi<Q\},\ \ \psi=Q\ \text{and}\ \f{\p\psi(x,H-0)}{\p\nu}=\sqrt{2\lambda-2gH}\ \ \text{on}~~
I=U_{\e}\cap\p\{\psi<Q\}.$$ It follows from the Cauchy-Kovalevskaya
theorem for $\psi$ in $U_{\e}\cap\{y<H\}$ that
$$\psi=\sqrt{2\ld-2gH}(y-H)+Q\ \ \text{in}\ \ \ U_{\e}\cap\{y<H\}.$$
 By using the unique continuation for the harmonic function, one has
$$\psi=\sqrt{2\ld-2gH}(y-H)+Q\ \ \text{in}\ \left\{(x,y)\mid-\infty<x<+\infty,H-\e<y<H\right\}\cap \O_{\mu},
$$
which is impossible.

%In view of that the free boundary $\Gamma_{2,\ld_0,Q_1,\mu}$ is
%empty and the free boundary $\Gamma_{2,\ld_n,Q_1,\mu}$ is non-empty,
%applying the similar arguments in Subcase 1.1 in the proof of Lemma
%\ref{lb7}, we can conclude that the free boundary
%$\Gamma_{2,\ld_n,Q_1,\mu}$ converges to $\bar
%D_\mu\cap\p\{\psi_{\ld_0,Q_1,\mu}<Q\}$ and the minimizer
%$\psi_{\ld_0,Q_1,\mu}$ satisfies the free boundary conditions on
%$I$, namely,
%$$\psi_{\ld_0,Q_1,\mu}(x,H)=Q,\ \ \text{and}\ \ \f{\p\psi_{\ld_0,Q_1,\mu}(x,H-0)}{\p y}=\sqrt{2\ld_0-2gH} \ \text{on}\ \
%I,$$  where
%$I=\left\{(x,H),\f{3+\mu}4<x<\f{1+3\mu}4\right\}\subset\bar
%D_\mu\cap\p\{\psi_{\ld_0,Q_1,\mu}<Q\}$, which leads to a contradiction
%via the Cauchy-Kovalevskaya theorem.

\end{proof}

The following lemma gives the Fact 4.

\begin{lemma}\label{lb10} For any $Q_1\in(0,Q)$, there exists a positive constant $C_0$ independent of
$\mu$ and $Q_1$, such that \be\label{b44}k_{1,\ld,Q_1,\mu}(H)>-1\ \
\text{or}\ \  k_{2,\ld,Q_1,\mu}(H)<1,\ee for any
$\ld>C_0$.\end{lemma}

\begin{proof}
Suppose that there exist a $Q_1\in(0,Q)$ and a large $\ld$, such
that
$$k_{1,\ld,Q_1,\mu}(H)\leq-1\ \ \text{and}\ \ k_{2,\ld,Q_1,\mu}(H)\geq1.$$  Denote the two initial
points of the free boundaries as
$\mathcal{B}_1=\left(k_{1,\ld,Q_1,\mu}(H),H\right)$ and
$\mathcal{B}_2=\left(k_{2,\ld,Q_1,\mu}(H),H\right)$. Then one has
$$\f{1}{r}\fint_{\p
B_r(X_0)}\psi_{\ld,Q_1,\mu}dS\leq\f{Q}{r}\ \ \text{for any
$X_0\in\Gamma_{1,\ld,Q_1,\mu}$}.$$ Let
$r_0=\f{Q}{c^*\sqrt{\ld-2gH}}$ ($c^*$ as in Lemma \ref{lf3}). Then
$$\f{1}{r_0}\fint_{\p
B_{r_0}(X_0)}\psi_{\ld,Q_1,\mu}dS\leq\f{Q}{r_0}=c^*\sqrt{2\ld-2gH}\
\ \text{for any $X_0\in\Gamma_{1,\ld,Q_1,\mu}$}.$$ We claim that
$$B_{r_0}(X_0)\cap\{\psi_{\ld,Q_1,\mu}=Q\}\neq\varnothing\ \
\text{for any $X_0\in\Gamma_{1,\ld,Q_1,\mu}$}.$$ In fact, suppose
that there exists an $X_0\in\Gamma_{1,\ld,Q_1,\mu}$, such that
$B_{r_0}(X_0)\cap\{\psi_{\ld,Q_1,\mu}=Q\}=\varnothing$, which
implies that $\psi_{\ld,Q_1,\mu}<Q$ in $B_r(X_0)$. The
non-degeneracy Lemma \ref{lf4} in the appendix implies that
$\psi_{\ld,Q_1,\mu}\equiv0$ in $B_{\f r8}(X_0)$, which contradicts
to the fact $X_0\in\Gamma_{1,\ld,Q_1,\mu}$.

Thus it holds that
$$B_{2r_0}(\mathcal{B}_1)\ \ \text{intersects
$\left\{\psi_{\ld,Q_1,\mu}=Q\right\}$}.$$ Similarly,
$$B_{2r_0}(\mathcal{B}_2)\ \ \text{intersects
$\left\{\psi_{\ld,Q_1,\mu}=0\right\}$}.$$ Without loss of
generality, we assume that the free boundary
$\Gamma_{2,\ld,Q_1,\mu}$ lies above of the free boundary
$\Gamma_{1,\ld,Q_1,\mu}$ near the segment $\overline{A_1A_2}$. Let
$E\subset\O_\mu$ be the domain bounded by $x=\f12$, $x=-\f12$ and
$\Gamma_{2,\ld,Q_1,\mu}$ and $y=H_0$ (see Figure \ref{f9}), where
$H_0=\min\left\{y\mid g_2(y)=\f12\right\}$. Set
$\Psi_\ld(X)=\f{Q-\psi_{\ld,Q_1,\mu}(X)}{\sqrt{2\ld}}$. It is easy
to check that $\Psi_\ld$ is harmonic in $E$ and
$$\Psi_\ld=0\ \ \text{and}\ \ \f12\leq|\g\Psi_{\ld}|=\f{\sqrt{2\ld-2gy}}{\sqrt{2\ld}}\leq 1\ \ \ \text{on}\ \ \Gamma_{2,\ld,Q_1,\mu},$$
 for sufficiently large $\ld$. By means of the results in Section 8 in \cite{AC1}, we can conclude that the free boundary $\Gamma_{2,\ld,Q_1,\mu}$ of
  $\Psi_\ld$ is analytic and the $C^{3}$-norm of $\p E\cap\Gamma_{2,\ld,Q_1,\mu}$ is independent of $\ld$.
  Applying the elliptic estimate for harmonic function $\Psi_{\ld}$ on the boundary
$\p E\cap\Gamma_{2,\ld,Q_1,\mu}$ (see Corollary 6.7 in \cite{GT}),
we have
$$\f12\leq|\g\Psi_\ld|\leq C\max_{X\in E}|\Psi_\ld(X)|\leq C\f{Q}{\sqrt{2\ld}}\ \quad \text{on}\ \
(\p E\cap\Gamma_{2,\ld,Q_1,\mu})\cap\left\{-\f14\leq
x\leq\f14\right\},$$ where the constant $C$ depends on $E$ and the
$C^{3}$-norm of $\p E\cap\Gamma_{2,\ld,Q_1,\mu}$, and does not
depend on $\ld$. This leads to a contradiction for sufficiently
large $\ld$.
\begin{figure}[!h]
\includegraphics[width=100mm]{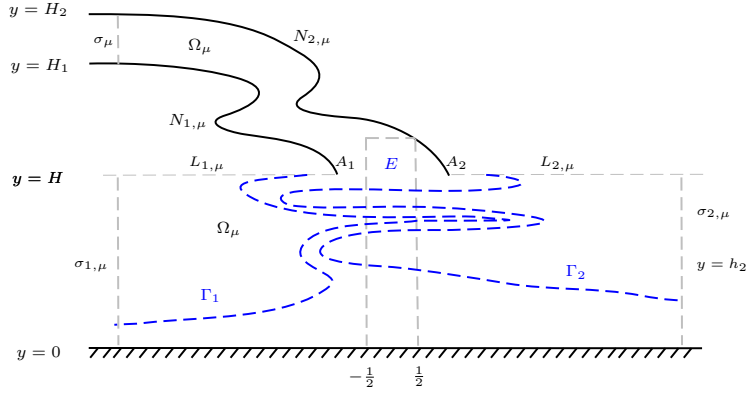}
\caption{The domain $E$}\label{f9}
\end{figure}

\end{proof}

Set \be\label{b45}\bar\ld_\mu=\sup_{\ld\in\Sigma_\mu}\ld.\ee It
follows from Lemma \ref{lb10} that there exists a $C_0>0$
independent of $\mu$ and $Q_1$, such that
 \be\label{b46}\f{\max\{Q_1^2,(Q-Q_1)^2\}}{2H^2}+gH\leq\bar\ld_\mu\leq C_0.\ee

Next, we can establish the almost continuous fit conditions to the
impinging jet flow under gravity.
\begin{prop}\label{lb11}(almost continuous fit conditions) For any $\mu>1$,  there exist a $\bar Q_{1,\mu}\in\left[0,Q\right]$ and a
$\bar\ld_\mu\geq\f{\max\{{\bar Q_{1,\mu}}^2,(Q-\bar
Q_{1,\mu})^2\}}{2H^2}+gH$ , such that

(1) $k_{1,\bar\ld_{\mu},\bar Q_{1,\mu},\mu}(H)\leq-1$ and
$k_{2,\bar\ld_{\mu},\bar Q_{1,\mu},\mu}(H)\geq1$.

(2) $k_{1,\bar\ld_{\mu},\bar Q_{1,\mu},\mu}(H)=-1$ or
$k_{2,\bar\ld_{\mu},\bar Q_{1,\mu},\mu}(H)=1$.

(3) $k_{2,\bar \ld_{\mu},\bar Q_{1,\mu},\mu}(H)=1$ for $\bar
Q_{1,\mu}<Q$.

(4) $k_{1,\bar\ld_{\mu},\bar Q_{1,\mu},\mu}(H)=-1$ for $\bar
Q_{1,\mu}>0$.

\end{prop}

\begin{remark} Obviously, as long as the critical cases $\bar Q_{1,\mu}=0$ and $\bar Q_{1,\mu}=Q$ are excluded, then
the continuous fit conditions
$$\text{$k_{1,\bar\ld_{\mu},\bar Q_{1,\mu},\mu}(H)=-1$ and
$k_{2,\bar\ld_{\mu},\bar Q_{1,\mu},\mu}(H)=1$}$$ are satisfied. This
will be done in the next section.

\end{remark}

\begin{proof} (1). By using the similar arguments in Lemma \ref{lb7},
 taking a sequence $\{\ld_n\}$ with $\ld_n\in\Sigma_\mu$ and
$\ld_n\uparrow \bar\ld_{\mu}$, we can obtain from the definition of
$\Sigma_\mu$ that there exists a sequence $\{Q_{1,n}\}$ with
$Q_{1,n}\in(0,Q)$, such that
\be\label{b470}\text{$k_{1,\ld_n,Q_{1,n},\mu}(H)<-1\ $ and $\
k_{2,\ld_n,Q_{1,n},\mu}(H)>1$}.\ee

It follows from the similar arguments in the proof of Lemma
\ref{lb7} that there exist two subsequences $\{\ld_n\}$ and
$\{Q_{1,n}\}$, such that
$$\ld_{n}\rightarrow \bar\ld_{\mu},\ Q_{1,n}\rightarrow \bar Q_{1,\mu},$$ and
$$\psi_{\ld_{n},Q_{1,n},\mu}\rightharpoonup\psi_{\bar\ld_{\mu},\bar Q_{1,\mu},\mu}\ \ \text{in $H^1(\O_\mu)$ and uniformly in any compact subset of $\O_\mu$},$$
as $n\rightarrow+\infty$. Furthermore,
$$\lim_{n\rightarrow+\infty}k_{1,\ld_{n},Q_{1,n},\mu}(H)=k_{1,\bar\ld_{\mu},\bar Q_{1,\mu},\mu}(H)
\ \text{and}\
\lim_{n\rightarrow+\infty}k_{2,\ld_{n},Q_{1,n},\mu}(H)=k_{2,\bar\ld_{\mu},\bar
Q_{1,\mu},\mu}(H).$$ Therefore, the assertion (1) follows from
\eqref{b470}.
%  \be\label{b48}k_{1,\bar\ld_{\mu},\bar
%Q_{1,\mu},\mu}(H)\leq -1\ \ \text{and}\ \ k_{2,\bar\ld_{\mu},\bar
%Q_{1,\mu},\mu}(H)\geq 1.\ee

(2). Suppose that the assertion (2) is not true. Then it follows
from the assertion (1) that \be\label{b49}k_{1,\bar\ld_{\mu},\bar
Q_{1,\mu},\mu}(H)<- 1\ \ \text{and}\ \ k_{1,\bar\ld_{\mu},\bar
Q_{1,\mu},\mu}(H)>1.\ee Due to the continuous dependence of
$k_{i,\ld,Q_1,\mu}(H)$ on $\ld$ and $Q_1$ in Lemma \ref{lb7}, there
exist $\ld'>\bar\ld_\mu$ and $Q_1'\in\left(0,Q\right)$ with
$\ld'-\bar\ld_\mu$ and $|\bar Q_{1,\mu}-Q_1'|$ small enough, such
that
$$k_{1,\ld',
Q_1',\mu}(H)<-1\ \ \text{and}\ \ k_{2,\ld', Q_1',\mu}(H)>1.$$ Then
we have that $\ld'\in \Sigma_\mu$, which contradicts to the
definition of $\bar\ld_\mu$ in \eqref{b45}.

%(3). First, we will show that \be\label{b53}\text{ if $\bar
%Q_{1,\mu}<Q$, then
% $k_{2,\bar\ld_{\mu},\bar
%Q_{1,\mu},\mu}(H)=1$}.\ee

(3). Suppose that the opposite is true. Then
$k_{2,\bar\ld_{\mu},\bar Q_{1,\mu},\mu}(H)>1$ for $\bar
Q_{1,\mu}<Q$, and the assertion (2) implies that
$k_{1,\bar\ld_{\mu},\bar Q_{1,\mu},\mu}(H)=-1$. Due to Lemma
\ref{lb8}, it holds that
$$\psi_{\bar\ld_{\mu},Q_1,\mu}(x,y)\geq\psi_{\bar\ld_{\mu},\bar
Q_{1,\mu},\mu}(x,y)\ \ \text{in}\ \ \O_\mu,$$ for any $Q_1\in(\bar
Q_{1,\mu},Q)$. Denote $$\text{ $\psi=\psi_{\bar\ld_{\mu},\bar
Q_{1,\mu},\mu}$, $\psi_1=\psi_{\bar\ld_{\mu},Q_1,\mu}(x,y)$,
$k_{i,\bar Q_{1,\mu}}(y)=k_{i,\bar\ld_{\mu},\bar Q_{1,\mu},\mu}(y)$
and $k_{i,Q_1}(y)=k_{i,\bar\ld_{\mu},Q_1,\mu}(y)$}$$ for $i=1,2$ in
the following. Therefore, $k_{1,Q_1}(H)\leq k_{1,\bar
Q_{1,\mu}}(H)=-1$. Next, we claim that \be\label{b51}k_{1,Q_1}(H)<
k_{1,\bar Q_{1,\mu}}(H)= -1,\ee for any $Q_1\in(\bar Q_{1,\mu},Q)$.
Suppose that there exists a $Q_1\in(\bar Q_{1,\mu},Q)$ such that
$k_{1,Q_1}(H)= k_{1,\bar Q_{1,\mu}}(H)$. It follows from the results
in Section 9 in \cite{ACF3} and Section 11 in Chapter 3 in
\cite{FA1} that the continuous fit condition implies the smooth fit
condition. Hence, $N_{1,\mu}\cup\Gamma_{1,\bar\ld_{\mu},\bar
Q_{1,\mu},\mu}$ and $N_{1,\mu}\cup\Gamma_{1,\bar\ld_{\mu},Q_1,\mu}$
are $C^1$ at $A_1$. Furthermore, $\g\psi$ is uniformly continuous in
a $\{\psi>0\}$-neighborhood of $A_1$, and $\g\psi_1$ is uniformly
continuous in a $\{\psi_1>0\}$-neighborhood of $A_1$, namely,
$$-\sqrt{2\ld-2gH}=\f{\p\psi}{\p\nu}=\f{\p\psi_1}{\p\nu}=-\sqrt{2\ld-2gH}\ \ \
\text{at $A_1$,}$$ where $\nu$ is the outer normal vector. It should
be noted that it's difficult to verify the inner ball property at
the point $A_1$, then one can not apply the Hopf's lemma at $A_1$.

First, we show that \be\label{b90}\Gamma_{1,\bar\ld_\mu,\bar
Q_{1,\mu},\mu}\cap\Gamma_{1,\bar\ld_\mu,Q_1,\mu}=\varnothing.\ee
Suppose not, there exists a
$X_0=(x_0,y_0)\in\Gamma_{1,\bar\ld_\mu,\bar
Q_{1,\mu},\mu}\cap\Gamma_{1,\bar\ld_\mu,Q_1,\mu}$. Since
$\Gamma_{1,\bar\ld_\mu,\bar Q_{1,\mu},\mu}$ and
$\Gamma_{1,\bar\ld_\mu,Q_1,\mu}$ are analytic at $X_0$, Hopf's lemma
gives that $$-\sqrt{2\ld-2g
y_0}=\f{\p\psi_1}{\p\nu}<\f{\p\psi}{\p\nu}=-\sqrt{2\ld-2g y_0}\ \
\text{at}\ \ X_0,$$ where $\nu$ is the outer normal vector, which
leads to a contradiction.

The continuity of $\psi$ and $\psi_1$ implies that
$$\psi<Q\ \ \text{and}\ \ \psi_1<Q\ \ \text{in
$B_r(A_1)$},$$ for any small $r>0$.

Let $\mathcal{M}=B_r(A_1)\cap N_{1,\mu}$. It follows from the strong
maximum principle that \be\label{b901}\psi<\psi_1\ \text{in
$\overline{B_r(A_1)}\cap\{\psi>0\}$}.\ee Since $N_{1,\mu}$ are
$C^{2,\alpha}$ and $\psi=\psi_1=0$ on $\mathcal{M}$, thanks to
Hopf's lemma, one has
$$\f{\p\psi_1}{\p\nu}<\f{\p\psi}{\p\nu}\ \ \text{on}\ \ \mathcal{M},$$ where $\nu$ is the outer normal vector of $\mathcal{M}$.
It follows from \eqref{b901} that there exists a small $\eta_1>0$,
such that $$\psi_1\geq(1+\eta_1)\psi\ \ \text{on $\p
B_r(A_1)\cap\{\psi>0\}$}
$$This, together with \eqref{b90} and \eqref{b901}, implies that there exists a small
$\eta\in(0,\eta_1)$ such that
$$\psi_1\geq(1+\eta)\psi\ \ \text{on $\p
(B_r(A_1)\cap\{\psi>0\})$}.$$ It follows from the maximum principle
that
$$\psi_1>(1+\eta)\psi\ \ \text{in
$B_r(A_1)\cap\{\psi>0\}$,}$$ which together with $\psi_1=\psi=0$ at
$A_1$ gives that
$$-\sqrt{2\ld-2gH}=\f{\p\psi_1}{\p\nu}\leq(1+\eta)\f{\p\psi}{\p\nu}=-(1+\eta)\sqrt{2\ld-2gH}\ \ \
\text{at $A_1$,}$$ which leads to a contradiction. Thus, the claim
\eqref{b51} holds true.

Since $k_{2,\bar Q_{1,\mu}}(H)>1$, by using the continuous
dependence of $k_{2,\bar\ld_\mu,Q_1,\mu}(H)$ with respect to $Q_1$,
we have \be\label{b52}k_{2,Q_1}(H)>1\ \ \text{for small $Q_1-\bar
Q_{1,\mu}>0$}.\ee In view of \eqref{b51} and \eqref{b52}, we can
obtain a contradiction to the definition of $\bar\ld_\mu$ by using
the continuous dependence of $k_{i,Q_1}(H)$ with respect to $\ld$
and $Q_1$.

(4). Similar to (3), one can show that
$$\text{ if $\bar Q_{1,\mu}>0$, then $k_{1,\bar\ld_{\mu},\bar
Q_{1,\mu},\mu}(H)=-1$}.$$

\end{proof}

\section{The existence of the impinging flow problem}
In this section, we will give the existence of the impinging flow
problem based on the previous results.

\begin{proposition}\label{lc1} For any $Q>2\sqrt{gH^3}$, there exist a pair $(\ld,Q_1)$ and a solution $\psi_{\ld,Q_1}$ to the free boundary problem \eqref{b6}
with $Q_1\in[0,Q]$ and
$\ld\geq\f{\max\{Q_1^2,(Q-Q_1)^2\}}{2H^2}+gH$. Moreover,

(1) $\psi_{\ld,Q_1}(x,y)$ is increasing with respect to $x$ and the
free boundary $\Gamma_{i,\ld,Q_1}$ is analytic. Furthermore, the
free boundary $\Gamma_{i,\ld,Q_1}$ can be described by a continuous
function $x=k_{i,\ld,Q_1}(y)$ for $y\in(h_i,H)$, respectively,
$i=1,2$.

(2) (almost continuous fit conditions) $k_{1,\ld,Q_{1}}(H)\leq -1$
and $k_{2,\ld,Q_{1}}(H)\geq 1$.
 Furthermore,
$$
k_{1,\lambda,Q_1}(H)=-1\quad{\text or}\quad k_{2,\lambda,Q_1}(H)=1,
$$and
 $$\text{
$k_{1,\ld,Q_{1}}(H)=-1$ if $Q_1>0$, and $ k_{2,\ld,Q_{1}}(H)=1$ if
$Q_1<Q$.}$$

\end{proposition}

\begin{proof}
Let $\{\mu_n\}$ be a sequence such that $\mu_n\rightarrow+\infty$.
It follows from the similar arguments in the proof of Lemma
\ref{lb7} that
$$\bar\ld_{\mu_n}\rightarrow \ld\ \ \text{and} \ \ \bar Q_{1,\mu_n}\rightarrow
Q_1,$$ and
$$\psi_{\bar\ld_{\mu_n},\bar Q_{1,\mu_n},\mu_n}\rightarrow\psi_{\ld,Q_1}
\ \text{weakly in $H_{loc}^1(\O)$ and uniformly in any compact
subset of $\O$}.$$ It follows from \eqref{b46} that
\be\label{c3}\ld\leq C_0.\ee

Next, we will show that $\psi_{\ld,Q_1}$ is a local minimizer to the
variational problem $(P_{\ld,Q_1})$ (see Remark \ref{rl1}). Set
$\psi_n=\psi_{\bar\ld_{\mu_n},\bar Q_{1,\mu_n},\mu_n}$. We first
claim that
 \be\label{c4}\ba{rl}&\int_{E}|\nabla \psi_n|^2+(2\bar\lambda_{\mu_n}-2gy)\chi_{\{0<\psi_n<Q\}\cap D} dxdy\\
 \leq&\int_{E}|\nabla
\t\psi|^2+(2\bar\lambda_{\mu_n}-2gy)\chi_{\{0<\t\psi<Q\}\cap D}
dxdy,\ea\ee for any $\t\psi\in H^1(E)$ and $\t\psi=\psi_n$ on $\p
E$, provided that $n$ is sufficiently large, where $E$ is any
compact subset of $\O$. In fact, there exists a $N>0$, such that
$E\subset\O_{\mu_n}$ for any $n>N$. Extend $\t\psi=\psi_n$ outside
of $E$ such that $\t\psi=\Psi_{\bar\ld_{\mu_n},\mu_n}(y)$ on
$\p\O_{\mu_n}$, then it implies that \eqref{c4} is valid.

%For any $\psi\in H^1(E)$ with $\psi=\psi_{\ld,Q_1}\ \text{on}\ \p
%E$, set
%$$\t\psi_n=\psi+(1-\eta)(\psi_n-\psi_{\ld,Q_1}),$$ where $\eta\in C_0^1(D)$
%and $0\leq\eta\leq 1$. Obviously, $\t\psi_n\in H^1(E)$ and
%$\t\psi_n=\psi_n$ on $\p E$.

%It follows from \eqref{c4} that
%$$\int_{E}|\g\psi_{n}|^2+(\ld_{L_n}-2gy)I_{\{0<\psi_{n}<Q\}\cap D}
%dxdy \leq\int_{E}|\g
%\t\psi_n|^2+(\ld_{L_n}-2gy)I_{\{0<\t\psi_n<Q\}\cap D} dxdy.$$

With the aid of \eqref{c4}, by using the similar arguments in Step 4
in the proof of Lemma \ref{lb7}, one can conclude that
$\psi_{\ld,Q_1}$ is a local minimizer to the variational problem
$(P_{\ld,Q_1})$, and thus it follows from Proposition \ref{lb1} that
$\psi_{\ld,Q_1}$ is harmonic in $\O\cap\{0<\psi_{\ld,Q_1}<Q\}$ and
the free boundary $\Gamma_{i,\ld,Q_1}$ is analytic with
$|\g\psi_{\ld,Q_1}|=\sqrt{2\ld-2gy}$ on $\Gamma_{i,\ld,Q_1}$.
%$$\int_{E}|\g\psi_{\ld,Q_1}|^2+(\ld-2gy)
%I_{\{0<\psi_{\ld,Q_1}<Q\}\cap D} dxdy \leq\int_{E}|\g
%\psi|^2+(\ld-2gy) I_{\{0<\psi<Q\}\cap D} dxdy,$$ for any $\psi\in
%H^1(E)$ with $\psi=\psi_{\ld,Q_1}\ \text{on}\ \p E$.
Furthermore, Lemma \ref{lb4} gives that the local minimizer
$\psi_{\ld,Q_1}(x,y)$ is increasing with respect to $x$. By using
similar arguments as in Lemma \ref{lb6}, one can conclude that the
free boundary $\Gamma_{i,\ld,Q_1}$ of $\psi_{\ld,Q_1}$ can be
described by a generalized continuous function $x=k_{i,\ld,Q_1}(y)$
($i=1,2$), namely, $k_{i,\ld,Q_1}(y+0)=\lim_{t\rightarrow
y+0}k_{i,\ld,Q_1}(t)$ and $k_{i,\ld,Q_1}(y-0)=\lim_{t\rightarrow
y-0}k_{i,\ld,Q_1}(t)$ exist and may be infinite, and
$k_{i,\ld,Q_1}(y+0)=k_{i,\ld,Q_1}(y-0)\in[-\infty,+\infty]$ for and
$y\in(h_i,H)$.

Next, we will show that \be\label{c1}\text{ $k_{1,\ld,Q_{1}}(H)=-1$
for $Q_1>0$, and $k_{2,\ld,Q_{1}}(H)=1$ for $Q_1<Q$}.\ee If not,
without loss of generality, one may assume that
$k_{1,\ld,Q_{1}}(H)\neq-1$ for $Q_1>0$. Since $Q_{1}>0$, so $\bar
Q_{1,\mu_n}>0$ for sufficiently large $n$, and it follows from Lemma
\ref{lb11} that $k_{1,\bar\ld_{\mu_n},\bar
Q_{1,\mu_n},\mu_n}(H)=-1$. Similar to the Step 5 in the proof of
Lemma \ref{lb7}, one can obtain a contradiction by using the
convergence of the free boundary and the Cauchy-Kovalevskaya
theorem.

\end{proof}

Next, $k_{i,\ld,Q_1}(y)$ has the following properties for $i=1,2$.

\begin{prop}\label{lc2} The free boundary $\Gamma_{i,\ld,Q_1}: x=k_{i,\ld,Q_1}(y)$ is a
bounded continuous function for any $x\in(h_i,H]$, where $h_i$ is
 determined uniquely by \eqref{b3} for $i=1,2$.
 Furthermore,$$\text{
 $\lim_{y\rightarrow
h_1^+}k_{1,\ld,Q_1}(y)=-\infty$ if $Q_1>0$, and $\lim_{y\rightarrow
h_2^+}k_{2,\ld,Q_1}(y)=+\infty$ if $Q_1<Q$.}$$

\end{prop}\begin{proof}Consider first that $Q_1<Q$. In this case, $h_2\in(0,H)$. It
follows from (2) in Proposition \ref{lc1} that
 $k_{2,\ld,Q_1}(H)=1$, and thus $k_{2,\ld,Q_1}(y)$ is
 finite near $y=H$. It remains to show that
\be\label{c7}\text{$k_{2,\ld,Q_1}(y)$ is continuous and finite in
$(h_2,H)$ and $\lim_{y\rightarrow
h_2^+}k_{2,\ld,Q_1}(y)=+\infty$},\ee which will be proved in three
steps.

{\bf Step 1.}
%$$\text{$x=k_{2,\ld,Q_1}(y)$ is continuous and finite value in $(-h_1,0)$ for $Q_1>0$},$$ and $$ \text{$x=k_{2,Q_1,Q_2}(y)$ is continuous and finite value in $(-h_2,0)$ for $Q_2>0$}.$$
%It follows from the results in Step 2 that $k_{2,\ld,Q_1}(0)=1$ for
%$Q_1>0$, which implies that $k_{2,\ld,Q_1}(y)$ is continuous and
%finite value in $(-\e,0]$ for small $\e>0$. We first claim that
%$$\text{$k_{2,\ld,Q_1}(y)$ is continuous and finite value in
%$(-\delta,0]$ and
%$\lim_{y\rightarrow\delta^+}k_{2,\ld,Q_1}(y)=+\infty$.}$$
Let $(\beta_i,\alpha_i)$ be the maximal intervals
($i=1,2,\cdot\cdot\cdot$), such that $k_{2,\ld,Q_1}(y)$ is finite
valued, $\alpha_1=H$ and $\beta_i\geq\alpha_{i+1}$. We first claim
that
$$\text{ the number of intervals $(\beta_i,\alpha_i)$ is
finite.}$$ If not, then $\alpha_i-\beta_i\rightarrow 0$ and
$\beta_{i}-\alpha_{i+1}\rightarrow 0$ as $i\rightarrow+\infty$.
There are the following two cases to be considered.

{\bf Case 1}.
$k_{2,\ld,Q_1}(\alpha_i-0)=k_{2,\ld,Q_1}(\beta_i+0)=-\infty$. (See
Figure \ref{f13})

\begin{figure}[!h]
\includegraphics[width=90mm]{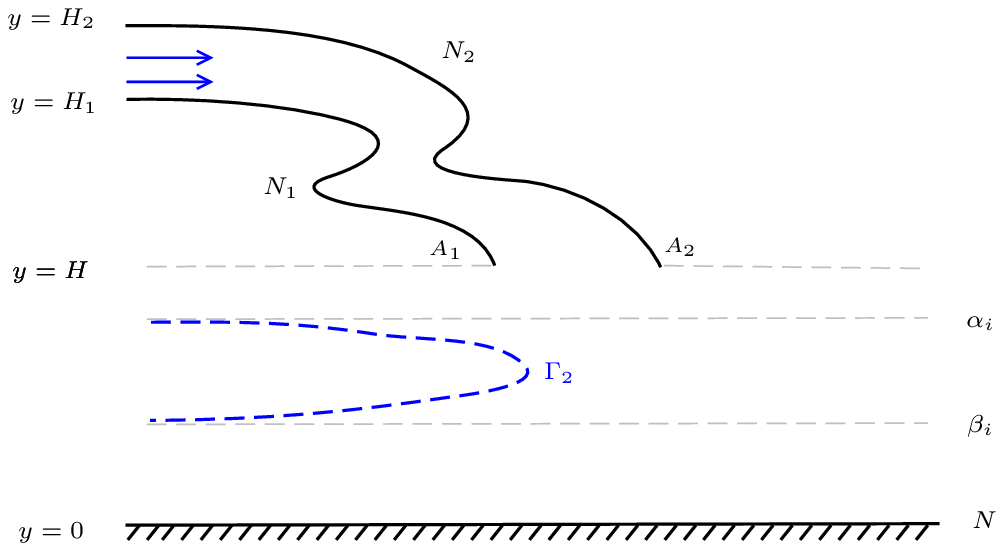}
\caption{Case 1}\label{f13}
\end{figure}

Denote $G_i\subset\{(x,y)\in D\mid x< k_{2, \ld,Q_1}(y),
\beta_i<y<\alpha_i\}$, such that $G_i$ satisfies $$\forall \
X_1=(x_1,y_1)\in G_i, \ \ \exists\ y_0\ \ \text{such that}\ \
x_1=k_{2,\ld,Q_1}(y_0) \ \ \text{and} \ \ X_0=(x_1,y_0)\in
\Gamma_{2,\ld,Q_1}.$$

Thanks to the Lipschitz continuity of $\psi_{\ld,Q_1}$, we have
$$Q-\psi_{\ld,Q_1}(X_1)=\psi_{\ld,Q_1}(X_0)-\psi_{\ld,Q_1}(X_1)\leq
C(\alpha_i-\beta_i),$$ which implies that
\be\label{c8}\text{$\psi_{\ld,Q_1}(X_1)>0$ for any $X_1\in G_i$,}\ee
provided that $\beta_i-\alpha_i$ is small enough such that
$C(\alpha_i-\beta_i)<Q$.

Hence, we can derive a contradiction to the non-oscillation Lemma
\ref{lb5} in the region $G_i\cap\{-2R<x<-R\}$ (for some $R$
sufficiently large) provided that $\alpha_i-\beta_i$ is small
enough.

{\bf Case 2}. $k_{2,\ld,Q_1}(\alpha_i-0)=+\infty$ or
$k_{2,\ld,Q_1}(\beta_i+0)=+\infty$. Without loss of generality, we
assume that $k_{2,\ld,Q_1}(\beta_i+0)=+\infty$ and consider the
following two subcases.

 {\bf Subcase 2.1}.
 $k_{2,\ld,Q_1}(\alpha_{i+1}-0)=+\infty$ (see Figure \ref{f14}).

\begin{figure}[!h]
\includegraphics[width=90mm]{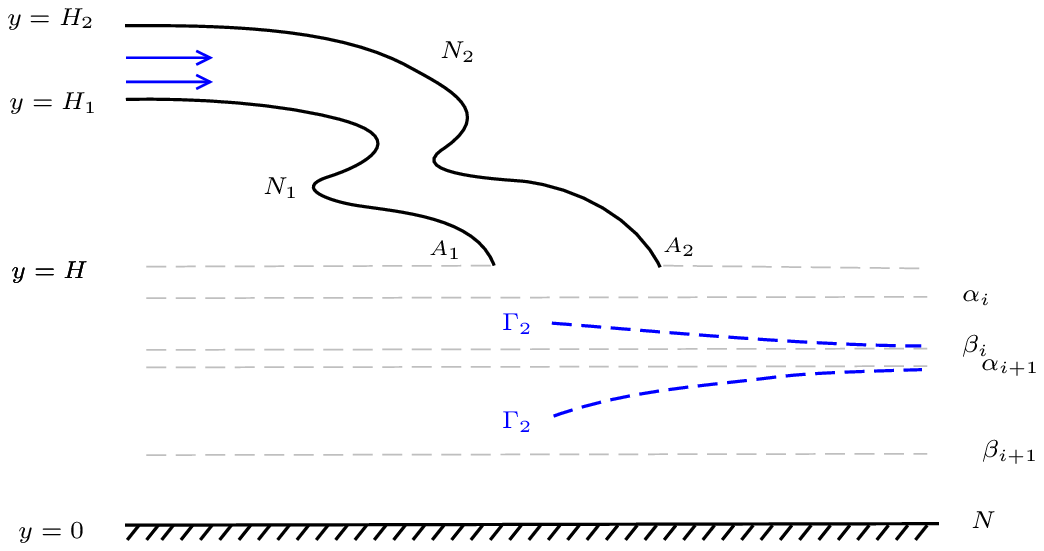}
\caption{Subcase 2.1}\label{f14}
\end{figure}

For sufficiently large $R>0$, set $$G_{i,R}=\left\{(x,y)\in D\mid
x>R,
\f{\alpha_{i+1}+\beta_{i+1}}2<y<\f{\alpha_{i}+\beta_{i}}2\right\}.$$
Similar to \eqref{c8}, we can conclude that $\psi_{\ld,Q_1}>0$ in
$G_{i,R}$ for sufficiently large $i$. This leads to a contradiction
by using the non-oscillation Lemma \ref{lb5} in the region
$G_{i,R}\cap\{x<2R\}$.

{\bf Subcase 2.2}. $k_{2,\ld,Q_1}(\alpha_{i+1}-0)=-\infty$ (see
Figure \ref{f15}).

\begin{figure}[!h]
\includegraphics[width=90mm]{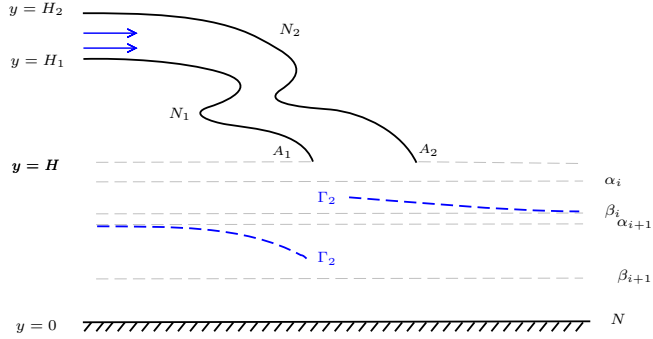}
\caption{Subcase 2.2}\label{f15}
\end{figure}

 By using
the monotonicity
 of $\psi_{\ld,Q_1}(x,y)$ with respect to $x$, we have that $\psi_{\ld,Q_1}=Q$ on $\{y=\alpha_{i+1}\}$, and thus the line $\{y=\alpha_{i+1}\}$ is the free boundary of
 $\psi_{\ld,Q_1}$. It follows from  Proposition \ref{lb1} that
 $$\psi_{\ld,Q_1}=Q\ \text{and}\ \f{\p\psi_{\ld,Q_1}}{\p y}=\sqrt{2\ld-2g\alpha_{i+1}}\ \text{on}\ \ \{y=\alpha_{i+1}\},$$
which contradicts to the Cauchy-Kovalevskaya theorem.
%Set $G_{i,R}=\left\{(x,y)\in D\mid x>R,
%\alpha_{i+1}<y<\f{\alpha_{i}+\beta_{i}}2\right\}$ for sufficiently
%large $i$ and $R$, we can obtain a contradiction by using the
%non-oscillation Lemma \ref{lb5} in $G_{i,R}\cap\{x<2R\}$.

Hence, we can conclude that the number of $(\alpha_i,\beta_i)$ is
finite.

{\bf Step 2.} In this step, we will show that \be\label{c9}\text{the
number of $(\alpha_i,\beta_i)$ is in fact one, namely, $i=1$}.\ee

If not, the number of the intervals $(\beta_i,\alpha_i)$ is at least
2. Then there are the following three cases.

{\bf Case 1.}
$k_{2,\ld,Q_1}(\beta_1+0)=k_{2,\ld,Q_1}(\alpha_2-0)=+\infty$.

The non-oscillation Lemma \ref{lb5} gives that $\alpha_2<\beta_1$.
We first claim that \be\label{c10}\text{ $k_{1,\ld,Q_1}(y)<+\infty$
for any $\alpha_2< y<\beta_1$.}\ee If not, without loss of
generality, one may assume that there exists a
$\beta_0\in(\alpha_2,\beta_1)$, such that
$$\lim_{y\rightarrow\beta_0^+}k_{1,\ld,Q_1}(y)=+\infty\ \ \text{and}\ \ k_{1,\ld,Q_1}(y)<+\infty\ \ \text{for}\ \ y\in(\beta_0,\beta_1).$$
Since $k_{2,\ld,Q_1}(\beta_1+0)=k_{1,\ld,Q_1}(\beta_0+0)=+\infty$,
we conclude that the free boundaries $\Gamma_{1,\ld,Q_1}$ and
$\Gamma_{2,\ld,Q_1}$ satisfy the flatness condition (see Section 7
in \cite{AC1}) near $y=\beta_0$ and $y=\beta_1$. Then there exists a
large $x_0>0$, such that
$$\text{$\Gamma_{1,\ld,Q_1}\cap\{(x,y)\mid x>x_0, \beta_0<y<\beta_1\}$ is described by $y=g_1(x)$, and $g_1(x)\downarrow \beta_0$ }$$as $x\rightarrow+\infty$, and $$\text{$\Gamma_{2,\ld,Q_1}\cap\{(x,y)\mid x>x_0,
\beta_1<y<\alpha_1\}$ is described by $y=g_2(x)$, and
$g_2(x)\downarrow \beta_1$ }$$as $x\rightarrow+\infty$. Furthermore,
it holds that
$$g'_i(x)\rightarrow 0\ \text{ as}\ \  x\rightarrow+\infty,\ \ \ \left|g_i^{(j)}(x)\right|\leq C \ \ \text{for} \ \ i=1,2,
j=2,3.$$

Due to the uniform elliptic estimates, there exists a sequence
$\{\psi_{\ld,Q_1}(x+n,y)\}$, such that
$$\psi_{\ld,Q_1}(x+n,y)\rightarrow\psi_0(x,y),\quad \text{in}\ C^{2,\alpha}(\mathbb{T})$$ where $\mathbb{T}=\{-\infty<x<+\infty\}\times\{\beta_0<y<\beta_1\}$, and
\be\label{c11}\left\{\ba{ll} &\Delta\psi_0=0  \quad\quad \text{in}~~\mathbb{T},\\
&\f{\p\psi_0(x,\beta_0)}{\p y}=\sqrt{2\lambda-2g\beta_0}\ \ \text{and}\  \ \f{\p\psi_0(x,\beta_1)}{\p y}=\sqrt{2\lambda-2g\beta_1} ,\\
&\psi_0(x,\beta_0)=0\ \text{and}\  \psi_0(x,\beta_1)=Q. \ea\right.
\ee A direct computation gives that
$$\psi_0(y)=\sqrt{2\lambda-2g\beta_1}(y-\beta_1)+Q\quad\quad\text{for}\ \beta_0<y<\beta_1,$$and$$\f{\p\psi_0(x,\beta_1)}{\p
y}=\sqrt{2\lambda-2g\beta_1}\quad\quad\text{for}\
-\infty<x<+\infty,$$which contradicts to $\f{\p\psi_0(x,\beta_1)}{\p
y}=\sqrt{2\lambda-2g\beta_0}$, due to
$\sqrt{2\lambda-2g\beta_1}<\sqrt{2\lambda-2g\beta_0}$.

Denote $$G_R=\{(x,y)\in D\mid R<x<2R, \alpha_2<y<H\}$$ for
sufficiently large $R>0$. Then it follows from the claim \eqref{c10}
that $\psi_{\ld,Q_1}>0$ in $G_R$, which contradicts to the
non-oscillation Lemma \ref{lb5} in $G_{R}$.

{\bf Case 2.}
$k_{2,\ld,Q_1}(\beta_1+0)=-k_{2,\ld,Q_1}(\alpha_2-0)=+\infty$.

Due to the monotonicity of $\psi_{\ld,Q_1}$ with respect to $x$, it
follows from  Proposition \ref{lb1} that
$$\psi_{\ld,Q_1}=Q\ \text{and} \ \f{\p\psi_{\ld,Q_1}}{\p y}=-\sqrt{2\ld-2g\alpha_2}\quad  \text{on}\
\{y=\alpha_2\},$$ which contradicts to the Cauchy-Kovalevskaya
theorem.

{\bf Case
3.}
$k_{2,\ld,Q_1}(\beta_1+0)=-k_{2,\ld,Q_1}(\alpha_2-0)=-\infty$.

Similarly to Case 2, one can get
$$\psi_{\ld,Q_1}=Q\ \text{and} \ \f{\p\psi_{\ld,Q_1}}{\p y}=\sqrt{2\ld-2g\beta_1}
\quad  \text{on} \ \{y=\beta_1\},$$ which also leads to a
contradiction.

%Similar to the claim \eqref{c10} in Case 1, we have
%$$\text{ $k_{1,\ld,Q_1}(y)<+\infty$ for any $\alpha_2\leq
%y\leq\beta_1$.}$$ which implies that $\psi_{\ld,Q_1}>0$ in
%$G_R=\{(x,y)\in D\mid R<x<2R, \alpha_2<y<H\}$ for sufficiently large
%$R>0$, which can obtain a contradiction by using the non-oscillation
%Lemma \ref{lb5} and Remark \ref{rl4} for $\psi_{\ld,Q_1}$ in
%$G_{R}$.

%{\bf Case 3.}
%$k_{2,\ld,Q_1}(\beta_1+0)=-k_{2,\ld,Q_1}(\alpha_2-0)=-\infty$. The
%monotonicity of $\psi_{\ld,Q_1}$ with respect to $x$ give that
%$$\psi_{\ld,Q_1}=Q\ \text{and}\ \ \f{\p\psi_{\ld,Q_1}}{\p y}\leq
%0\  \text{on}\ \ \{y=\beta_1\}.$$ Similar to Case 2, we can obtain a
%contradiction by using the non-oscillation Lemma \ref{lb5} and
%Remark \ref{rl4}.

{\bf Case 4.}
$k_{2,\ld,Q_1}(\beta_1+0)=k_{2,\ld,Q_1}(\alpha_2-0)=-\infty$.

There are the following two subcases.

{\bf Subcase 4.1.} $k_{2,\ld,Q_1}(\beta_2+0)=-\infty$. We first
claim that \be\label{c12}k_{1,\ld,Q_1}(y)=-\infty\ \ \text{for any}\
\ \beta_2<y<\alpha_2.\ee If not, without loss of generality, we
assume that there exists a $\beta_0\in(\beta_2,\alpha_2)$, such that
$$\lim_{y\rightarrow\beta_0^+}k_{1,\ld,Q_1}(y)=-\infty\ \ \text{and}\ \ k_{1,\ld,Q_1}(y)>-\infty\ \ \text{for}\ \ y\in(\beta_0,\beta_0+\e) \ \text{and small $\e>0$}.$$

Similar to Case 1 in Step 2, there exists a sequence
$\{\psi_{\ld,Q_1}(x-n,y)\}$, such that
$$\psi_{\ld,Q_1}(x-n,y)\rightarrow\psi_0(x,y),\quad\quad\text{in}\ \ C^{2,\alpha}\left(\{x<k_{2,\ld,Q_1}(y),\beta_2<y<\alpha_2\}\right),$$ and $\psi_0$ satisfies
\eqref{c11}, which leads to a contradiction. The claim \eqref{c12}
implies that
$$\psi_{\ld,Q_1}>0\ \ \text{in}\ \ \{(x,y)\in D\mid
x<k_{2,\ld,Q_1}(y),\beta_2<y<\alpha_2\},$$ which contradicts to the
non-oscillation Lemma \ref{lb5}.

{\bf Subcase 4.2.} $k_{2,\ld,Q_1}(\beta_2+0)=+\infty$. Similar to
the claim \eqref{c12}, one can get
\be\label{c13}k_{1,\ld,Q_1}(y)=-\infty\ \ \text{for any}\ \
0<y<\alpha_2.\ee The non-oscillation Lemma \ref{lb5} yields that
\be\label{c14}k_{2,\ld,Q_1}(y)=+\infty\ \ \text{for any}\ \
0<y<\beta_2.\ee

With the aid of \eqref{c13} and \eqref{c14}, by using the similar
arguments in Case 1 in Step 2, one can get a sequence
$\{\psi_{\ld,Q_1}(x+n,y)\}$, such that
$$\psi_{\ld,Q_1}(x+n,y)\rightarrow\psi_1(x,y),\quad\ \  \text{in}\ ~~C^{2,\alpha}\left(\{-\infty<x<+\infty\}\times\{0<y<\beta_2\}\right),$$
and $\psi_1$ satisfies $$\left\{\ba{ll} &\Delta\psi_1=0 \ \  \text{in}\ ~~\{-\infty<x<+\infty\}\times\{0<y<\beta_2\},\\
&\f{\p\psi_1(x,\beta_2)}{\p y}=\sqrt{2\lambda-2g\beta_2} ,\ \
\psi_1(x,\beta_2)=Q. \ea\right. $$ By the uniqueness for the
Cauchy-Kovalevskaya theorem, one has
\be\label{c19}\psi_1(x,y)=\sqrt{2\lambda-2g\beta_2}(y-\beta_2)+Q\ \
\text{in}\ \{-\infty<x<+\infty\}\times\{0<y<\beta_2\}.\ee It follows
from Lemma \ref{lb3} that
$$\max\{-\sqrt{2\ld-2gh_1}y+Q_1,0\}\leq\psi_1(x,y)\leq \min\{\sqrt{2\ld-2gh_2}y+Q_1,Q\}$$ in $D$, which
implies that $\psi_1(x,0)=Q_1$. Hence
\be\label{c15}Q-Q_1=\sqrt{2\ld-2g\beta_2}\beta_2.\ee

Similarly, there exists a sequence $\{\psi_{\ld,Q_1}(x-n,y)\}$, such
that
$$\psi_{\ld,Q_1}(x-n,y)\rightarrow\psi_2(x,y),\quad \text{in}\
C^{2,\alpha}(\mathbb{T}),$$where
$\mathbb{T}=\{-\infty<x<+\infty\}\times\{0<y<\alpha_2\}$,
and $\psi_2$ satisfies $$\left\{\ba{ll} &\Delta\psi_2=0  \quad\quad \text{in}~~\mathbb{T},\\
&\f{\p\psi_2(x,\alpha_2)}{\p y}=\sqrt{2\lambda-2g\alpha_2} ,\ \
\psi_2(x,\alpha_2)=Q. \ea\right. $$ It implies that
\be\label{c16}Q-Q_1=\sqrt{2\ld-2g\alpha_2}\alpha_2.\ee It follows
from \eqref{c15} and \eqref{c16} that we obtain a contradiction to
$\alpha_2>\beta_2$.

Hence, we have shown that
$$\text{$k_{2,\ld,Q_1}(y)$ is continuous and finite in $(\beta_1,H)$
and infinite in $(0,\beta_1)$}.$$ Furthermore,
\be\label{c17}\text{$\lim_{y\rightarrow
\beta_1^+}k_{2,\ld,Q_1}(y)=+\infty$}.\ee

{\bf Step 3.} In this step, we will show that
$$\beta_1=h_2.$$

Due to \eqref{c17}, it follows from the similar arguments in Case 1
in Step 1 that there exists a sequence $\{\psi_{\ld,Q_1}(x+n,y)\}$,
such that
$$\psi_{\ld,Q_1}(x+n,y)\rightarrow\psi_0(x,y)\quad\quad\text{in}\ C^{2,\alpha}(\mathbb T),$$where $\mathbb T=\{-\infty<x<+\infty\}\times\{0<y<\beta_1\}$, and
$$\left\{\ba{ll} &\Delta\psi_0=0  \quad\quad \text{in}~~\mathbb T,\\
&\f{\p\psi_0(x,\beta_1)}{\p y}=\sqrt{2\lambda-2g\beta_1}\
\text{and}\ \psi_0(x,\beta_1)=Q\ \ \text{for $-\infty<x<+\infty$}.
\ea\right.
$$ Similar arguments in the proof in Subcase 4.2 in Step 2 show
$$Q-Q_1=\sqrt{2\lambda-2g\beta_1}\beta_1.$$
This and Proposition \ref{lb0} give $\beta_1=h_2$.

Next we consider the case that $Q_1=Q$. It follows from similar
arguments above that
 $$\text{$k_{1,\ld,Q_1}(y)$
is continuous and finite in $(h_1,H)$ and $\lim_{y\rightarrow
h_1^+}k_{1,\ld,Q_1}(y)=-\infty$}.$$ Finally, using the
non-oscillation Lemma \ref{lb5}, one can show that
$$h_2=0,\ \ \ \text{$k_{2,\ld,Q_1}(y)$
is continuous and finite in $(0,H)$}.$$

\end{proof}

The previous result gives the almost continuous fit conditions. The
continuous fit conditions will follow once the critical value
$Q_1=0$ or $Q_1=Q$ can be excluded.

\begin{proposition}\label{lc3}The value $Q_1$ in Proposition \ref{lc1} lies in $(0,Q)$.
\end{proposition}
\begin{proof}
Without loss of generality, we assume that $Q_1=0$, which implies
that $h_2\in(0,H)$. The non-oscillation Lemma \ref{lb5} gives that
$$k_{1, \ld,0}(0)=\lim_{y\rightarrow0^+} k_{1,\ld,0}(y)\ \text{exists}.$$ Next, we consider the following
three cases.

{\bf Case 1.} $k_{1,\ld,0}(0)=+\infty$ (see Figure \ref{f10}).

\begin{figure}[!h]
\includegraphics[width=100mm]{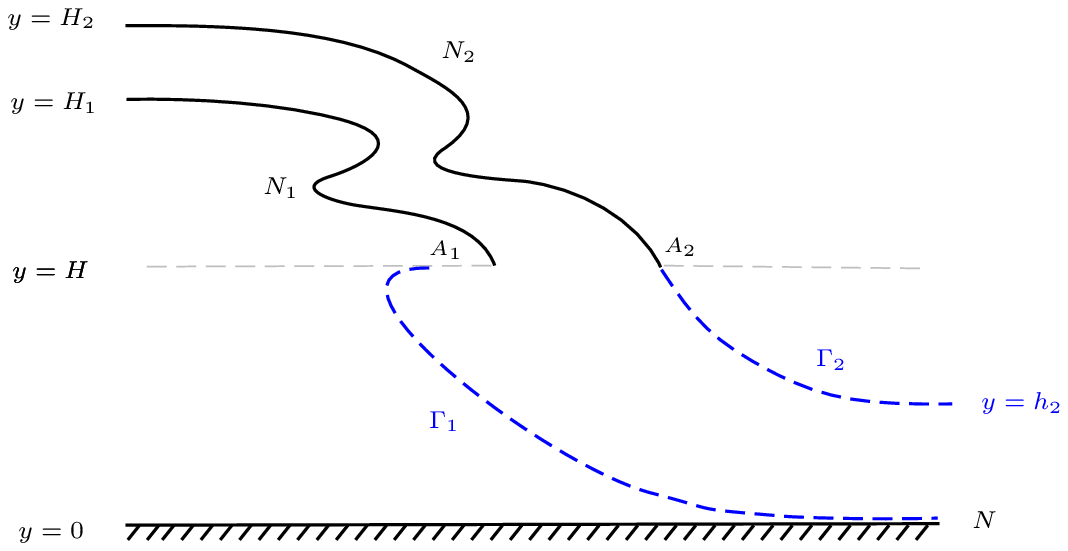}
\caption{Case 1.}\label{f10}
\end{figure}

Since $k_{2,\ld,0}(h_2)=+\infty$, it follows from the similar
arguments in Case 1 in Step 2 in the proof of Proposition \ref{lc2}
that there exists a sequence $\{\psi_{\ld,0}(x+n,y)\}$, such that
$$\psi_{\ld,0}(x+n,y)\rightarrow\psi_0(x,y)\quad\quad \text{in}\ C^{2,\alpha}(\{-\infty<x<+\infty\}\times\{0<y<h_2\}),$$
and $\psi_0$ satisfies $$\left\{\ba{ll} &\Delta\psi_0=0  \quad\quad \text{in}~~\{-\infty<x<+\infty\}\times\{0<y<h_2\},\\
&\f{\p\psi_0(x,0+0)}{\p y}=\sqrt{2 \lambda}\ \ \text{and}\ \ \f{\p\psi_0(x,h_2-0)}{\p y}=\sqrt{2\lambda-2gh_2},\\
&\psi_0(x,0)=0\ \text{and}\  \psi_0(x,h_2)=Q. \ea\right. $$ This is
an overdetermined problem, due to
$\sqrt{2\lambda-2gh_2}<\sqrt{2\lambda}$.

{\bf Case 2.} $k_{1,\ld,0}(0)=-\infty$ (see Figure \ref{f11}).

\begin{figure}[!h]
\includegraphics[width=100mm]{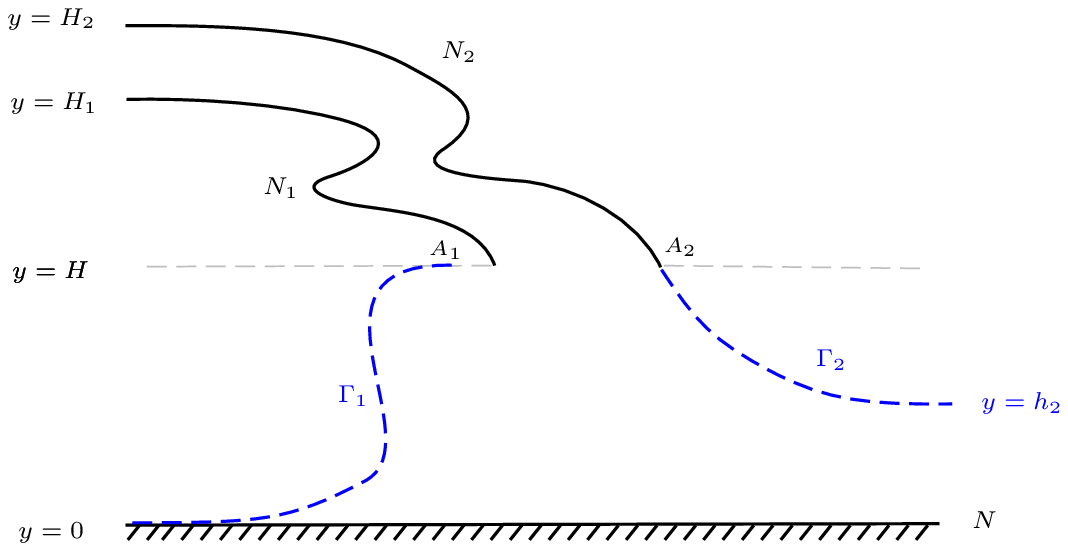}
\caption{Case 2.}\label{f11}
\end{figure}

The maximum principle gives that $$\f{\p\psi_{\ld,0}}{\p y}\geq 0\ \
\text{on}\ \ N.$$ Similar to the Case 1 in the proof of Proposition
\ref{lc2}, one can obtain a contradiction by using the
non-oscillation Lemma \ref{lb5} and Remark \ref{rl4} for
$\psi_{\ld,0}$ in $D\cap\{-2R<x<-R\}\cap\{\psi_{\ld,0}>0\}$,
provided that $R$ is sufficiently large.

{\bf Case 3.} $k_{1, \ld,0}(0)\in(-\infty,+\infty)$ (see Figure
\ref{f12}).

\begin{figure}[!h]
\includegraphics[width=100mm]{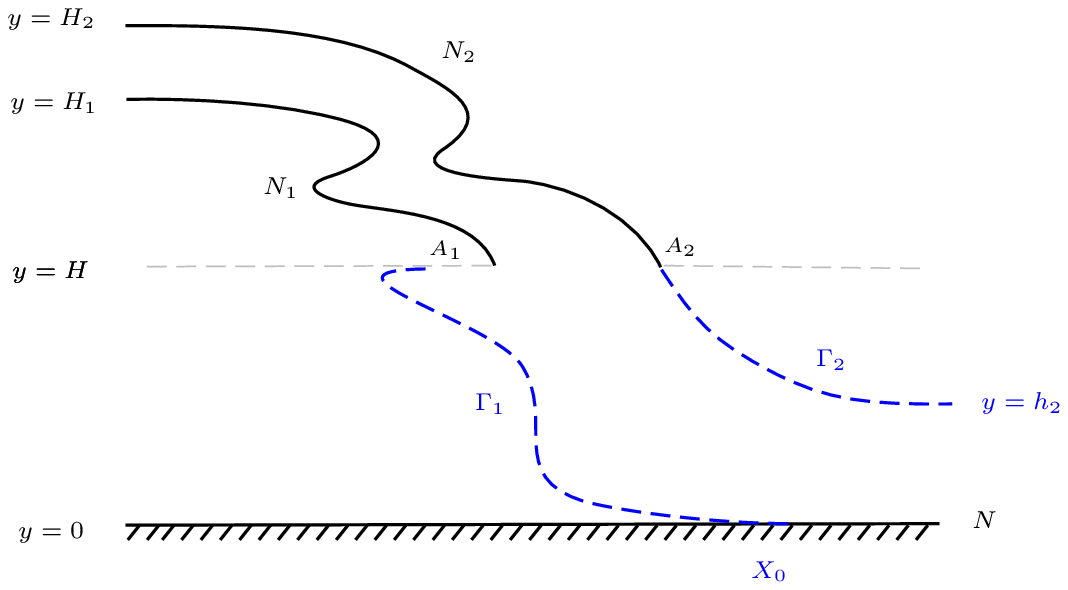}
\caption{Case 3.}\label{f12}
\end{figure}

Set $X_0=(k_{1,\ld,0}(0),0)$. It follows from the results in Section
9 in \cite{ACF3} and Section 11 in Chapter 3 in \cite{FA1} that the
continuous fit conditions imply the smooth fit conditions, we have
that the free boundary $\Gamma_{1,\ld,0}$ is $C^1$-smooth at $X_0$.
Furthermore, $\g\psi_{\ld,0}$ is uniformly continuous in a $\{\psi_{
\ld,0}>0\}$-neighborhood of $X_0$, and thus $$\text{ $|\g\psi_{
\ld,0}|=\f{\p\psi_{\ld,0}}{\p y}=\sqrt{2\ld}$\quad at $X_0$.}$$

Consider a function $$\text{$\o_0(x,y)=\min\{\gamma y,Q\}$\quad for
$\gamma=\f{\sqrt{2\ld-2gh_2}+\sqrt{2\ld}}2$.}$$ It follows from the
Step 3 in the proof of Proposition \ref{lc1} that
$$\psi_{\ld,0}(x,y)\rightarrow\min\{\sqrt{2\ld-2gh_2}y,Q\}\ \ \text{as}\ \ x\rightarrow+\infty.$$
Since $\gamma>\sqrt{2\ld-2gh_2}$, so $\psi_{ \ld,0}\leq\o_0$ in the
far field, which together with the maximum principle gives that
$$\psi_{\ld,0}(x,y)\leq\o_0(x,y)\ \ \text{in}\ \ D\cap\{\o_0<Q\}.$$
In view of $\psi_{\ld,0}=\o_0=0$ at $X_0$, thus we have
$$\sqrt{2\ld}=\f{\p\psi_{\ld,0}}{\p y}=\f{\p\psi_{\ld,0}}{\p\nu}\leq\f{\p\o_0}{\p\nu}=\f{\p\o_0}{\p
y}=\gamma=\f{\sqrt{2\ld-2gh_2}+\sqrt{2\ld}}2 \ \ \text{at}\ \ X_0,
$$ which leads to a contradiction, due to $h_2>0$.

\end{proof}

\begin{remark}
Proposition \ref{lc3} implies that the impinging jet under gravity
with continuous fit conditions must possess two asymptotic
directions in far fields. This conclusion coincides the results on
impinging jet in absence of gravity in \cite{CDW}.
\end{remark}

As a consequence of Proposition \ref{lc1}-Proposition \ref{lc3}, we
can get
\begin{theorem}\label{lc4} For any $Q>2\sqrt{gH^3}$, there exist an effluent flux $Q_1\in(0,Q)$ and a $\ld>\f{\max\{Q_1^2,(Q-Q_1)^2\}}{2H^2}+gH$,
such that there exists a solution $(u,v,p,\Gamma_1,\Gamma_2)$ to the
impinging jet flow problem, where $$u=\f{\p\psi_{\ld,Q_1}}{\p y},\
v=-\f{\p\psi_{\ld,Q_1}}{\p x} \ \text{and}\
p=p_{atm}+\ld-\f{|\g\psi_{\ld,Q_1}|^2}{2}-gy,$$ and
$$\Gamma_i=\Gamma_{i,\ld,Q_1}:x=k_{i,\ld,Q_1}(y)=k_i(y)\ \text{for $
y\in(h_i,H)$ and $i=1,2$}.$$

\end{theorem}

\begin{proof}

It follows from Proposition \ref{lc1} and Proposition \ref{lc3} that
$$\ Q_1\in(0,Q)\ \ \text{and}\ \ \ld\geq\f{\max\{Q_1^2,(Q-Q_1)^2\}}{2H^2}+gH.$$
Furthermore, Proposition \ref{lc1} and Proposition \ref{lc2} give
that
$$k_{1,\ld,Q_1}(H)=-1\ \
\text{and}\ \ k_{2,\ld,Q_1}(H)=1,$$ and
 $$\text{$\lim_{y\rightarrow
h_1^+}k_{1,\ld,Q_1}(y)=-\infty$ and $\lim_{y\rightarrow
h_2^+}k_{2,\ld,Q_1}(y)=+\infty$},$$ where $h_1$ and $h_2$ are
uniquely determined by \eqref{b3}. It follows from Lemma \ref{lf8}
in the appendix that the free boundary $\Gamma_{1,\ld,Q_1}$ and
$\Gamma_{2,\ld,Q_1}$ are analytic, and the Bernoulli's law gives
that $p=p_{atm}$ on $\Gamma_{1,\ld,Q_1}\cup\Gamma_{2,\ld,Q_1}$.

 Next, we will show that
$\ld>\f{\max\{Q_1^2,(Q-Q_1)^2\}}{2H^2}+gH$. If not, without loss of
generality, one may assume that $Q_1\geq\f{Q}2>\sqrt{gH^3}$ and
$\ld=\f{Q_1^2}{2H^2}+gH$. Since $Q_1=\sqrt{2\ld-2gh_1}h_1$, it
follows from the proof of Proposition \ref{lb0} that $h_1=H$. This
implies that the free boundary $\Gamma_{1,\ld,Q_1}$ is empty, which
leads to a contradiction.

Therefore, the continuous fit conditions are fulfilled for the free
boundary $\Gamma_{i,\ld,Q_1}$ at $A_i$ ($i=1,2$). It follows from
the results in Section 9 in \cite{ACF3} and Section 11 in Chapter 3
in \cite{FA1} that the continuous fit conditions imply the smooth
fit conditions, hence, $N_i\cup\Gamma_{i,\ld,Q_1}$ is $C^1$-smooth
at $A_i$, that is
$$k_{1,\ld,Q_1}'(H-0)=g'_1(H+0)\ \
\text{and}\ \ k_{2,\ld,Q_2}(H-0)=g'_2(H+0).$$ Furthermore,
$\g\psi_{\ld,Q}$ is uniformly continuous in a $\{\psi_{
\ld,Q_1}>0\}$-neighborhood of $A_1$, and $\g\psi_{\ld,Q_1}$ is
uniformly continuous in a $\{\psi_{\ld,Q_1}<Q\}$-neighborhood of
$A_2$.

It remains to show that \be\label{d1} v=-\f{\p\psi_{\ld,Q_1}}{\p
x}<0 \ \quad
\text{in}~~(\O\cap\{0<\psi_{\ld,Q_1}<Q\})\cup\Gamma_{1,\ld,Q_1}\cup\Gamma_{2,\ld,Q_1}.\ee

By virtue of the monotonicity of $\psi_{\ld,Q_1}$ with respect to
$x$, one has
$$\psi_{\ld,Q_1}(x_1,y)\geq\psi_{\ld,Q_1}(x_2,y) \ \ \text{for any} \ \
x_1\geq x_2,$$ which implies that $$\text{ $-\f{\p
\psi_{\ld,Q_1}}{\p x}\leq 0$ in
$\O_0=\O\cap\{0<\psi_{\ld,Q_1}<Q\}$.}$$

Consider $v=-\f{\p \psi_{\ld,Q_1}}{\p x}$ in $\O_0$, which solves
the Laplace equation in $\O_0$. The strong maximum principle gives
that $v<0$ in $\O_0$. Finally, we claim that
$$v<0\ \ \ \text{on}~~\Gamma_{1,\ld,Q_1}\cup\Gamma_{2,\ld,Q_1}.$$
If not, without loss of generality, suppose that there exists a
$X_0=(x_0,y_0)\in\Gamma_{2,\ld,Q_1}$, such that $v(X_0)=0$. Since
the free boundary $\Gamma_{2,\ld,Q_1}$ is analytic at $X_0$, we can
assume that $\nu=(0,1)$ is the outer normal vector to
$\Gamma_{2,\ld,Q_1}$ at $X_0$. Then Hopf's lemma implies that
\be\label{d2}v_y=\f{\p v}{\p\nu}>0\ \ \text{at}\ \ X_0.\ee

It follows from Proposition \ref{lb1} that
$u^2+v^2=|\g\psi_{\ld,Q_1}|^2=2\ld-2gy$ on $\Gamma_{2,\ld,Q_1}$, and
\be\label{d3}\f{\p (u^2+v^2)}{\p s}=\f{\p (2\ld-2gy)}{\p
s}=(0,-2g)\cdot(1,0)=0\ \ \text{at}\ \ X_0,\ee where $s=(1,0)$ is
the tangential vector of $\Gamma_{2,\ld,Q_1}$ at $X_0$. On the other
hand, it follows from \eqref{d2} that
$$\f{\p (u^2+v^2)}{\p s}=2uu_x+2vv_y=-2uv_y=-2\sqrt{2\ld-2gy_0}v_y<0\ \ \text{at}\ \
X_0,$$ which contradicts to \eqref{d3}.

Since $v<0$ on $\Gamma_{1,\ld,Q_1}\cup\Gamma_{2,\ld,Q_1}$, the
implicit function theorem gives that $x=k_{i,\ld,Q_1}(y)$ is
$C^1$-smooth for any $x\in(h_i,H)$, $i=1,2$.

%Denote
%$$u=\f{\p\psi_{\ld,Q_1}}{\p y},\ v=-\f{\p\psi_{\ld,Q_1}}{\p x} \ \text{and}\ p=P_{atm}+\ld-\f{|\g\psi_{\ld,Q_1}|^2}{2}-gy,$$ and
%$$\Gamma_i=\Gamma_{i,\ld,Q_1}:x=k_{i,\ld,Q_1}(y)=k_i(y)\ \text{for}\
%y\in(h_i,H).$$ Hence, we obtain a solution
%$(u,v,p,\Gamma_1,\Gamma_2)$ to the impinging jet flow problem.

\end{proof}

\section{The asymptotic behaviors in the far field}
Since $k_{2,\ld,Q_1}(h_2+0)=+\infty$, it follows from the similar
arguments in Step 3 in the proof of Proposition \ref{lc2} that there
exists a sequence $\{\psi_{\ld,Q_1}(x+n,y)\}$, such that
$$\psi_{\ld,Q_1}(x+n,y)\rightarrow \sqrt{2\ld-2gh_2}y+Q_1 \ \
\text{uniformly in}\ \ C^{2,\alpha}(G),$$ for any $G\subset\subset
(-\infty,+\infty)\times(0,h_2)$. This gives that
$$(-v(x,y),u(x,y))=\nabla\psi_{\ld,Q_1}(x,y)\rightarrow
(0,\sqrt{2\ld-2gh_2})$$ uniformly in any compact subset of
$(0,h_2)$, as $x\rightarrow+\infty$. It follows from Bernoulli's law
that
$$p(x,y)\rightarrow p_{atm}+g(h_2-y)\ \text{uniformly in
any compact subset of $(0,h_2)$, }$$ as $x\rightarrow+\infty$.

Furthermore, it holds that
$$\g (u,v)\rightarrow0\ \ \text{and}\ \g p\rightarrow\left(0,-g\right),$$
uniformly in any compact subset of $(0,h_2)$, as
$x\rightarrow+\infty$.

By using the similar arguments, one gets
$$\left(u,v,p\right)\rightarrow \left(-\sqrt{2\ld-2gh_1},0,
p_1(y)\right),\ \g (u,v)\rightarrow 0 \ \text{and}\ \ \g
p\rightarrow \left(0,-g\right)$$
 uniformly in any compact subset of $(0,h_1)$ as $
x\rightarrow -\infty$, where $p_1(y)=p_{atm}+g(h_1-y)$.

Next, we will obtain the asymptotic behavior of the impinging jet
flow in the upstream.
 Define the function
$\psi_n(x,y)=\psi_{\ld,Q_1}(x-n,y)$ for $x<\f n2$. For any compact
subset $G$ of $S=(-\infty,+\infty)\times(H_1,H_2)$, with the aid of
the assumptions of the nozzle walls $N_1$ and $N_2$ in the inlet, it
follows from the standard elliptic estimates that we have
\be\label{d7}\|\psi_n\|_{C^{2,\alpha}(G)}\leq C(G) \ \ \text{for
sufficiently large}\ n, \ \ 0<\alpha<1.\ee Arzela-Ascoli lemma gives
that there exists a subsequence still labeled by $\psi_n$, such that
 \be\label{d8}\psi_n\rightarrow\psi_0 \ \
\text{uniformly in}\ \ C^{2,\beta}(G)\ \ \text{for \ some}\
0<\beta<\alpha.\ee Furthermore, $\psi_0$ satisfies \be\label{d9}
 \left\{\ba{ll}\Delta\psi_0=0 \ \ \text{in}\
\ S,\\
\psi_0(x,H_1)=0,\ \psi_0(x,H_2)=Q,\\
0\leq\psi_0\leq Q\ \ \text{in}\ \ S. \ea\right. \ee

It is easy to check that the boundary value problem \eqref{d9}
possesses a unique solution \be\label{d10}\psi_0= \f{Q}{H_2-H_1}y\ \
\text{in}\ \ S.\ee

Hence, this together with \eqref{d8} yields that
$$\left(u,v\right)\rightarrow \left(\f{Q}{H_2-H_1},0\right),\ \g
(u,v)\rightarrow 0$$
 uniformly in any compact subset of $(H_1,H_2)$ as $
x\rightarrow -\infty$. It then follows form the Bernoulli's law that
$$p\rightarrow p_0(y)\ \text{and}\ \ \g p\rightarrow
\left(0,-g\right)$$
 uniformly in any compact subset of $(H_1,H_2)$ as $
x\rightarrow -\infty$, where
$p_0(y)=p_{atm}+\ld-\f{Q^2}{2(H_2-H_1)^2}-gy$.

%\begin{prop}\label{l55}The Impinging jet flow satisfies the following asymptotic behavior
%in far fields,
%$$\left(u(x,y),v(x,y),p(x,y)\right)\rightarrow \left(u_0,0, p_0\right),\ \g (u,v)\rightarrow 0\ \text{and}\ \ \g p\rightarrow (0,1)$$
% uniformly in any compact subset of $(H_2,H_1)$ as $
%x\rightarrow -\infty$, where $u_0=\f{Q_1+Q_2}{H_1-H_2}$ and
%$p_0=\mathcal{B}-\f{(Q_1+Q_2)^2}{2(H_1-H_2)^2}-\f y2$.

%Similarly, the Impinging jet flow satisfies the following asymptotic
%behavior in downstream,
%$$\left(u(x,y),v(x,y),p(x,y)\right)\rightarrow \left(\sqrt{\ld+h_1},0, p_1\right),\ \g (u,v)\rightarrow 0\ \text{and}\ \ \g p\rightarrow (0,1)$$
% uniformly in any compact subset of $(-H,-h_1)$ as $
%x\rightarrow +\infty$, and
%$$\left(u(x,y),v(x,y),p(x,y)\right)\rightarrow \left(-\sqrt{\ld+h_2},0, p_2\right),\ \g (u,v)\rightarrow 0\ \text{and}\ \ \g p\rightarrow (0,1)$$
% uniformly in any compact subset of $(-H,-h_1)$ as $
%x\rightarrow -\infty$ where $p_1=\mathcal{B}-\f{\ld+h_1+y}2$ and
%$p_2=\mathcal{B}-\f{\ld+h_2+y}2$, $h_1$ and $h_2$ are determined
%uniquely by$$Q_1=\sqrt{\ld+h_1}(H-h_1)\ \ \text{and}\ \
%Q_2=\sqrt{\ld+h_2}(H-h_2).$$
%\end{prop}

\section{The properties of the interface $\Gamma$}

In this section, we will investigate the properties of the interface
$\Gamma=\O\cap\{\psi_{\ld,Q_1}=Q_1\}$ between two fluids with
different downstreams. This is another important difference between
the impinging jet flows and the general jet flows.

\begin{prop}\label{lg1} The interface
$\Gamma=\O\cap\{\psi_{\ld,Q_1}=Q_1\}$ can be denoted by $x=k(y)$ for
$y\in(0,H_3)$, where $H_3=\f{Q_1(H_2-H_1)}{Q}+H_1$. Furthermore,
$\lim_{y\rightarrow0^+}k(y)$ exists and is finite, and $k'(0+0)=0$.

\end{prop}
\begin{proof}

Recall that $\O_0=\O\cap\{0<\psi_{\ld,Q_1}<Q\}$. Since $v<0$ in
$\O_0$, the interface $\Gamma$ is a $y$-graph, and the implicit
function theorem implies that $\Gamma$ can be denoted by a
$C^1$-smooth function $x=k(y)$ for any $y\in(0,H_3)$, where
$H_3=\f{Q_1(H_2-H_1)}{Q}+H_1$ is nothing but the asymptotic height
of the interface $\Gamma$ in upstream. It follows from the
asymptotic behavior of $\psi_{\ld,Q_1}$ in Section 4 that
$$\lim_{y\rightarrow H_3^-}k(y)=-\infty\ \ \text{and $k(y)$ is finite for any $y\in(0,H_3)$}.$$

Next, we will show that \be\label{e1}\lim_{y\rightarrow 0^+}k(y)\ \
\text{exists and is finite}.\ee

Suppose that there exist two sequences $y_n\downarrow0$ and $\t
y_n\downarrow0$, such that
\be\label{e2}\lim_{n\rightarrow+\infty}k(y_n)=x_1\ \ \text{and}\ \
\lim_{n\rightarrow+\infty}k(\t y_n)=x_2.\ee

Without loss of generality, one may assume that $x_1>x_2$. We claim
that \be\label{e3}\f{\p\psi_{\ld,Q_1}}{\p y}=0\ \ \text{on the
segment $I=\{(x,0)\mid x_2<x<x_1\}$}.\ee In fact, suppose that there
exists a point $X_0=(x_0,0)\in I$, such that $$\text{
$\f{\p\psi_{\ld,Q_1}}{\p y}\neq0$ at $X_0$.}$$ Without loss of
generality, assume that $\f{\p\psi_{\ld,Q_1}(X_0)}{\p y}>0$. Then
one has \be\label{e4}\psi_{\ld,Q_1}(x_0,y)<Q_1\ \ \text{for}\ \
0<y<\e,\ee for small $\e>0$, due to $\psi_{\ld,Q_1}(x_0,0)=Q_1$.

Due to the monotonicity of $\psi(x,y)$ with respect to $x$, it
follows from \eqref{e2} that
$$\psi_{\ld,Q_1}(x_0,\t y_n)\geq\psi_{\ld,Q_1}(k(\t y_n),\t y_n)=Q_1\ \text{and}\
\psi_{\ld,Q_1}(x_0,y_n)\leq\psi_{\ld,Q_1}(k(y_n),y_n) =Q_1,$$ for
any small $\t y_n,y_n\in(0,\e)$, which contradicts to \eqref{e4}.

Take $X_0=(x_0,0)\in I$, then $x_2<x_0-\e<x_1$ for small $\e>0$.
Denote $\psi_\e(x,y)=\psi_{\ld,Q_1}(x-\e,y)$, the monotonicity of
$\psi_{\ld,Q_1}(x,y)$ with respect to $x$ implies that
$$\psi_\e(x,y)\leq\psi_{\ld,Q_1}(x,y)\ \ \text{in $\O$.}$$ In view of $\psi_{\ld,Q_1}=\psi_\e=Q_1$ at $X_0$, then Hopf's
lemma shows that
$$\f{\p(\psi_{\ld,Q_1}-\psi_\e)}{\p y}>0\ \ \text{at $X_0$},$$ which contradicts to
\eqref{e3}. Hence, $\lim_{y\rightarrow 0^+}k(y)$ exists.

Next, we will show that
$$-\infty<\lim_{y\rightarrow0^+}k(y)<+\infty.$$
If not, without loss of generality, one may assume that
$\lim_{y\rightarrow0^+}k(y)=+\infty$. The asymptotic behavior of
$\psi_{\ld,Q_1}$ implies that there exists a sufficiently large
$R_0>0$, such that
\be\label{e5}\text{$\sqrt{2\ld-2gH}\leq|\g\psi_{\ld,Q_1}|\leq
\sqrt{2\ld}$}\ee  in any subdomain of $\{(x,y)\mid x\geq R_0,
0<y<h_2\}$. Denote $G=\{(x,y)\mid x<k(y), 0<y<h_2\}$. Then one has
$$\Delta\psi_{\ld,Q_1}=0\ \ \text{and $\psi_{\ld,Q_1}<Q_1$ in}\ \ G_{R_0},$$ where
$G_{R_0}=G\cap\{R_0<x<2R_0\}$. The maximum principle gives that
\be\label{e6}\f{\p\psi_{\ld,Q_1}}{\p\nu}\geq 0\ \ \text{on}\ \ N,\ \
\text{and}\ \ \f{\p\psi_{\ld,Q_1}}{\p\nu}\geq \sqrt{2\ld-2gH}\ \
\text{on}\ \ \Gamma,\ee where $\nu$ is the outer normal vector.

It follows from \eqref{e5} and \eqref{e6} that
$$\ba{rl}\sqrt{2\ld-2gH}R_0\leq&\int_{(\Gamma\cup N)\cap\{R_0<x<2R_0\}}\f{\p\psi_{\ld,Q_1}}{\p\nu}dS\\
=&\int_{G_{R_0}}\Delta\psi_{\ld,Q_1}dxdy\\
&-\int_{\p G_{R_0}\cap\{x=R_0\}}\f{\p \psi_{\ld,Q_1}} {\p
x}dy+\int_{\p
G_{R_0}\cap\{x=2R_0\}}\f{\p \psi_{\ld,Q_1}} {\p x}dy\\
\leq& 2\sqrt{2\ld}H,\ea $$ which leads to a contradiction, provided
that $R_0$ is sufficiently large.

Finally, we will show that the interface $\gamma$ intersects the
ground $N$ perpendicularly.

Denote $S=(k(0),0)$ as the stagnation point. Let
$\t\psi(x,y)=\psi_{\ld,Q_1}(k(0)+x,y)-Q_1$. It follows from
Proposition \ref{lb1} that $\t\psi(x,y)$ is harmonic in
$B_r(0)\cap\{y>0\}$ for some $r>0$ and $\t\psi$ vanishes on
$B_r(0)\cap\{y=0\}$. Hence, $\t\psi$ can be extended to a harmonic
function in $B_r(0)$. Then $\{\t\psi=0\}$ consists of arcs forming
equal angles at the origin. By virtue of the previous arguments,
there exists a continuous arc $\gamma$ initiating at $0$ and
$\t\psi$ vanishes on $\gamma$. Hence, $\gamma$ must intersect
$\{y>0\}$ orthogonally at $0$, which implies that
$$k'(0+0)=0.$$

\end{proof}

\section{Appendix}

In this section, we will list some important lemmas for the
minimizer $\psi_{\ld,Q_1,\mu}$, which have been established in
\cite{AC1,ACF1,ACF3}.

It follows from Lemma 3.2 in \cite{AC1} and Lemma 3.1 in \cite{ACF1}
that the following lemma holds.
\begin{lemma}\label{lf1}
There exists a universal constant $C^*$ such that, for any disc
$B_r(X_0)\subset\Omega_\mu$, if
$$\f1r\fint_{\partial B_r(X_0)}\psi_{\ld,Q_1,\mu} dS\geq C^*\sqrt{2\ld},\ \ \text{then $\psi_{\ld,Q_1,\mu}>0$ in $B_r(X_0)$}.$$
Similarly, if $$\f1r\fint_{\partial B_r(X_0)}(Q-\psi_{\ld,Q_1,\mu})
dS\geq C^*\sqrt{2\ld},\ \ \text{then $\psi_{\ld,Q_1,\mu}<Q $ in
$B_r(X_0)$}.$$
\end{lemma}

By using the similar arguments for Lemma 3.4 in \cite{AC1} and Lemma
2.4 in \cite{ACF3}, one can have the following lemma.
\begin{lemma}\label{lf3} There exists a universal positive constant $c^*$,
such that for any disc $B_r(X_0)$ with $X_0\in D_\mu$, if
$$\f{1}{r}\fint_{\p B_r(X_0)}\psi_{\ld,Q_1,\mu} dS\leq c^*\sqrt{2\ld-2gH},~~and~~\psi_{\ld,Q_1,\mu}<Q~~in~~B_r(X_0),$$
then $\psi_{\ld,Q_1,\mu}=0$ in $B_{\f r8}(X_0)\cap D_\mu$;
similarly, if
$$\f{1}{r}\fint_{\p B_r(X_0)}\left(Q-\psi_{\ld,Q_1,\mu}\right) dS\leq
c^*\sqrt{2\ld-2gH},~~and~~\psi_{\ld,Q_1,\mu}>0~~in~~B_r(X_0),$$ then
$\psi_{\ld,Q_1,\mu}=Q$ in $B_{\f r8}(X_0)\cap D_\mu$.\end{lemma}
Lemma \ref{lf3} implies the following non-degeneracy
 lemma.
\begin{lemma}\label{lf4}For any $X_0\in \overline{\{\psi_{\ld,Q_1,\mu}>0\}\cap D_\mu}$, if $\psi_{\ld,Q_1,\mu}<Q$ in $B_r(X_0)$ for some $r>0$,
 then \be\label{f25}\f{1}{r}\fint_{\p B_r(X_0)}\psi_{\ld,Q_1,\mu} dS\geq c^*\sqrt{2\ld-2gH}.\ee In
 particular,
 \be\label{f26}\sup_{\p B_r(X_0)}\psi_{\ld,Q_1,\mu}\geq c^*\sqrt{2\ld-2gH} r.\ee
Similarly, the result holds with $\psi_{\ld,Q_1,\mu}$ replaced by
$Q-\psi_{\ld,Q_1,\mu}$.

 \end{lemma}

It follows from Lemma 5.1 in \cite{ACF1} that we have the following
bounded gradient lemma.
\begin{lemma}\label{lf7}
Let $X_0=(x_0,y_0)$ be a free boundary point in $D_\mu$ and
$B_r(X_0)\subset B_R(X_0)\subset D_\mu$. Then
$$|\nabla\psi_{\ld,Q_1,\mu}(x,y)|\leq C\ \ \text{in}\ \ B_r(X_0),$$ where
$C$ depends only on $\ld$ and $\left(1-\f rR\right)^{-1}$, but it is
independent of $Q$.
\end{lemma}

Since $Q>2\sqrt{gH^3}$, it follows from Remark \ref{rb1} that
$\ld\geq\f{\max\{Q_1^2,(Q-Q_1)^2\}}{8H^2}+gH\geq \f32gH$, we have
that $\sqrt{2\ld-2gy}$ is analytic for any $y\in(0,H)$. Then the
regularity of the free boundary is obtained in Theorem 8.4 in
\cite{AC1}.
\begin{lemma}\label{lf8}
The free boundary $D_\mu\cap\p\{0<\psi_{\ld,Q_1,\mu}<Q\}$ is locally
analytic.
\end{lemma}

\begin{remark}\label{rl1}
Those lemmas still hold for the local minimizer $\psi_{\ld,Q_1}$ to
the variational problem $(P_{\ld,Q_1})$. A local minimizer
$\psi_{\ld,Q_1}$ to the variational problem $(P_{\ld,Q_1})$ means
that \be\label{h1}J_{E}(\psi_{\ld,Q_1})\leq J_{E}(\psi)\ \ \text{for
any $\psi\in H^1(E)$},\ \ \psi=\psi_{\ld,Q_1}\ \ \text{on} \ \p
E,\ee for any bounded domain $E\subset \O$ with smooth boundary,
where the functional
$$J_E(\psi)=\int_{E}|\nabla \psi|^2+\sqrt{2\lambda-2gy}
\chi_{\{0<\psi<Q\}\cap D} dxdy.$$
\end{remark}

\bigskip

{\bf Conflict of interest. } The authors declare that they have no
conflict of interest.

 \

\bibliographystyle{plain}

\end{document}